\documentclass[reqno,pdftex]{amsart}
\numberwithin{equation}{section}

\usepackage[paper=letterpaper,twoside=false,text={28pc,43pc},left=2cm]{geometry}
\usepackage{times}
\usepackage{multirow}
\usepackage{enumitem}

\usepackage{draftwatermark}
  \SetWatermarkLightness{0.95}

\usepackage{fancyhdr}
\usepackage{datetime}
\usepackage{float}
\usepackage{caption}
\usepackage[textwidth=6cm,shadow]{todonotes}

\usepackage[colorlinks=true, pdfstartview=FitH, linkcolor=blue,
        citecolor=blue, urlcolor=blue]{hyperref}

\usepackage{keycommand}
    \begingroup
      \makeatletter
      \catcode`\/=8 %
      \@firstofone
        {\endgroup
          \renewcommand{\ifcommandkey}[1]{%
            \csname @\expandafter \expandafter \expandafter
            \expandafter \expandafter \expandafter  \expandafter
            \kcmd@nbk \commandkey {#1}//{first}{second}//oftwo\endcsname } }

\newkeycommand\funcU[delta=\delta,m=m,q=q,H=H][1]{
  U_{\commandkey{q},\commandkey{m}} \!\left( #1 ; \commandkey{delta},\commandkey{H} \right) }
\newkeycommand\funcV[delta=\delta,m=m,q=q,H=H][1]{
  V_{\commandkey{q},\commandkey{m}} \!\left( #1 ; \commandkey{delta},\commandkey{H} \right) }
\newkeycommand\funcW[q=q][1]{
  W_{\commandkey{q}}{( #1 )}}
\newkeycommand\funcY[m=m,q=d,R=R,x=x][1]{
  Y_{\commandkey{q},\commandkey{m},\commandkey{R}} \! \left( \commandkey{x},  #1 \right)}
\newkeycommand\funcF[q=\chi,m=m,R=R,H={H_2}][1]{
  F_{\commandkey{q},\commandkey{m},\commandkey{R}} \! \left( #1; \commandkey{H}\right)}
\newkeycommand\funcG[q=q,m=m,R=R,H=H,Htwo={H_2}][1]{
  G_{\commandkey{q}, \commandkey{m}, \commandkey{R}} \! \left(#1; \commandkey{H},\commandkey{Htwo} \right) }
\newkeycommand\funcHone[q=d,m=m,R=R][1]{
  H^{(1)}_{\commandkey{q},\commandkey{m},\commandkey{R}}\! \left( #1 \right) }
\newkeycommand\funcHtwo[q=d,m=m,R=R,H={H_2}][1]{
  H^{(2)}_{\commandkey{q},\commandkey{m},\commandkey{R}} \! \left(#1; \commandkey{H} \right) }
\newkeycommand\funcBone[q=d,m=m,R=R,r=r,H=H_2][1]{
  B^{(1)}_{\commandkey{q},\commandkey{m},\commandkey{R}} \! \left(#1; \commandkey{r}, \commandkey{H} \right) }
\newkeycommand\funcBtwo[q=d,m=m,R=R,r=r,H=H][1]{
  B^{(2)}_{\commandkey{q},\commandkey{m},\commandkey{R}} \! \left(#1; \commandkey{r} \right) }
\newkeycommand\funcB[q=d,m=m,R=R,r=r,H={H,H_2}]{
  B_{\commandkey{q},\commandkey{m},\commandkey{R}} \! \left(\commandkey{r}, \commandkey{H} \right) }
\newkeycommand\funcPhi[c1=c_1,c2=c_2,lambda=\lambda,mu=\mu,u=u]{
  \Phi\!\left(\commandkey{u} ; \commandkey{c1},\commandkey{c2},\commandkey{lambda},\commandkey{mu} \right) }

\makeatletter
\g@addto@macro\bfseries{\boldmath}
\makeatother

\newtheorem{theorem}{Theorem}[section]
\newtheorem{lemma}[theorem]{Lemma}
\newtheorem{prop}[theorem]{Proposition}
\newtheorem{cor}[theorem]{Corollary}
\theoremstyle{definition}
\newtheorem{definition}[theorem]{Definition}

\renewcommand{\mod}[1]{{\ifmmode\text{\rm\ (mod~$#1$)}\else\discretionary{}{}{\hbox{ }}\rm(mod~$#1$)\fi}}
\newcommand{\ep}{\varepsilon} \renewcommand{\epsilon}{\varepsilon}
\newsymbol\nmid 232D

\newcommand{\lcalc}{{\tt lcalc}}

\newcommand{\C}{{\mathbb C}}

\newcommand{\Zchi}{{\mathcal Z(\chi)}}
\newcommand{\Zchistar}{{\mathcal Z(\chi^*)}}
\newcommand{\GRH}[1]{\ensuremath{\text{GRH}(#1)}}

\newcommand{\erfc}{\mathop{\rm erfc}}

\DeclareMathOperator{\Ei}{Ei}
\DeclareMathOperator{\Li}{Li}
\newcommand{\RR}{Ramar\'e--Rumely}

\renewcommand{\phi}{\varphi}

\begin{document}

\title[Explicit bounds for primes in arithmetic progressions]{Explicit bounds for primes in arithmetic progressions \\  \today\ \currenttime}
\author{Michael A. Bennett}
\address{Department of Mathematics \\ University of British Columbia \\ Room 121, 1984 Mathematics Road \\ Vancouver, BC, Canada V6T 1Z2}
\email{bennett@math.ubc.ca}
\author{Greg Martin}
\address{Department of Mathematics \\ University of British Columbia \\ Room 121, 1984 Mathematics Road \\ Vancouver, BC, Canada V6T 1Z2}
\email{gerg@math.ubc.ca}
\author{Kevin O'Bryant}
\address{Department of Mathematics \\ City University of New York,
  College of Staten Island and The Graduate Center \\ 2800 Victory
  Boulevard \\ Staten Island, NY, USA 10314}
\email{truculentmath@icloud.com}
\author{Andrew Rechnitzer}
\address{Department of Mathematics \\ University of British Columbia \\ Room 121, 1984 Mathematics Road \\ Vancouver, BC, Canada V6T 1Z2}
\email{andrewr@math.ubc.ca}
\begin{abstract}
We derive explicit upper bounds for various counting functions for primes in arithmetic progressions. By way of example, if $q$ and $a$ are integers with $\gcd(a,q)=1$ and $3 \leq q \leq 10^5$, and $\theta(x;q,a)$ denotes the sum of the logarithms of the primes $p \equiv a \mod{q}$ with $p \leq x$, we show that
$$
 \big| \theta (x; q, a) - {x}/{\phi (q)} \big| < \frac1{160} \frac{x}{\log x}
$$
for all $x \geq 8 \cdot 10^9$, with significantly sharper constants obtained for individual moduli~$q$. We establish
inequalities of the same shape for the other standard prime-counting functions $\pi(x;q,a)$ and $\psi(x;q,a)$, as well
as inequalities for the $n$th prime congruent to $a\mod q$ when $q\le 1200$. For moduli $q>10^5$, we find even stronger
explicit inequalities, but only for much larger values of~$x$. Along the way, we also derive an improved explicit lower
bound for $L(1,\chi)$ for quadratic characters $\chi$, and an improved explicit upper bound for exceptional zeros.
\end{abstract}
\maketitle
\tableofcontents

\thispagestyle{empty}

\section{Introduction and statement of results}

The Prime Number Theorem, proved independently by Hadamard \cite{Had} and de la Vall\'ee Poussin \cite{dlV} in 1896, states that
\begin{equation}  \label{pi x}
\pi (x) = \sum_{\substack{p \leq x \\ p \text{ prime}}} 1 \sim \frac{x}{\log x},
\end{equation}
or, equivalently, that
\begin{equation}  \label{theta x and psi x}
\theta (x) = \sum_{\substack{p \leq x \\ p \text{ prime}}} \log p \sim x
\; \; \mbox{ and } \; \; \psi (x) =  \sum_{\substack{p^n \leq x \\ p \text{ prime}}} \log p \sim x,
\end{equation}
where by $f(x) \sim g(x)$ we mean that $\lim_{x \rightarrow \infty} f(x)/g(x)=1$.
Quantifying these statements by deriving explicit bounds upon the error terms
\begin{equation} \label{class}
\left| \pi (x) - \Li(x) \right|, \; \; \left| \theta (x) - x \right| \; \; \mbox{ and } \;
\;
\left| \psi (x) - x \right|
\end{equation}
is a central problem in multiplicative number theory (see for example Ingham~\cite{Ing} for classical work along these lines).
Here, by $\Li(x)$ we mean the function defined by
\begin{equation} \label{Li def}
\Li (x) = \int_2^x \frac{dt}{\log t} \sim \frac{x}{\log x}.
\end{equation}

Our interest in this paper is the consideration of similar questions for primes in arithmetic progressions. Let us define, given relatively prime positive integers $a$ and $q$,
\begin{equation}  \label{theta xqa and psi xqa}
\theta (x; q, a ) = \sum_{\substack{p \leq x\\p \equiv a \mod{q}}} \log p
\; \; \mbox{ and } \; \; \psi (x; q, a ) = \sum_{\substack{p^n \leq x \\ p^n \equiv a \mod{q}}} \log p,
\end{equation}
where the sums are over primes $p$ and prime powers $p^n$, respectively. We further let
\begin{equation}  \label{pi xqa}
\pi (x; q, a ) = \sum_{\substack{p \leq x \\ p \equiv a \mod{q}}} 1
\end{equation}
denote the number of primes up to $x$ that are congruent to $a$ modulo $q$. We are interested in upper bounds, with explicit constants, for the analogues to equation~\eqref{class}, namely the error terms
\begin{equation} \label{class2}
\left| \pi (x; q, a) - \frac{\Li(x)}{\phi (q) } \right|, \; \; \left| \theta (x; q, a) -
\frac{x}{\phi(q)} \right|, \; \; \mbox{ and } \; \;
\left| \psi (x; q, a) - \frac{x}{\phi(q)} \right|.
\end{equation}
Such explicit error bounds can take two shapes. The first, which we will term bounds of {\it Chebyshev-type}, are upper bounds upon the error terms that are small multiples of the main term in size, for example inequalities of the form
\begin{equation} \label{sharkey}
\left| \psi (x; q, a) - \frac{x}{\phi (q)} \right| < \delta_{q,a} \frac{x}{\phi (q)},
\end{equation}
for (small) positive $\delta_{q,a}$ and all suitably large values of $x$.
The second, which we call bounds of {\it de la Vall\'ee Poussin-type}, have the feature that the upper bounds upon the error are of genuinely smaller order than the size of the main term (and hence, in particular, imply the Prime Number Theorem for the corresponding arithmetic progression, something that is not true of inequality~\eqref{sharkey}).

Currently, there are a number of explicit inequalities of Chebyshev-type in the literature. In McCurley~\cite{Mc1}, we find such bounds for ``non-exceptional'' moduli $q$ (which is to say, those $q$ for which the associated Dirichlet $L$-functions have no real zeros near $s=1$), valid for large values of~$x$. McCurley \cite{Mc2} contains analogous bounds in the case $q=3$. Ramar\'e and Rumely \cite{RR} refined these arguments to obtain reasonably sharp bounds of Chebyshev-type for all $q \leq 72$ and various larger composite $q \leq 486$; the first author \cite{Benn} subsequently extended these results to primes $73 \leq q \leq 347$. Very recently, these results have been sharpened further for all moduli  $q \leq 10^5$ by Kadiri and Lumley~\cite{KaLu}.

Bounds of de la Vall\'ee Poussin-type are rather less common, however, other than the classical case where one considers all primes (that is, when $q=1$ or $2$), where such inequalities may be found in famous and oft-cited work of Rosser and Schoenfeld \cite{RS} (see also \cite{RS2,Sc} for subsequent refinements). When $ q\geq 3$, however, the only such result currently in the literature in explicit form may be found in a 2002 paper of Dusart \cite{Du}, which treats the case $q=3$. Our goal in the paper at hand is to deduce explicit error bounds of de la Vall\'ee Poussin-type for all moduli $q \geq 3$, for each of the corresponding functions $\psi (x; q, a), \theta  (x; q, a)$ and $\pi  (x; q, a)$.
In each case with $3 \leq q \leq 10^5$, exact values of the constants $c_\psi(q), c_{\theta}(q), c_{\pi}(q), x_\psi(q), x_\theta(q)$,
and $x_\pi(q)$ defined in our theorems can be found in data files accessible at:
\begin{center}
\texttt{\href{http://www.nt.math.ubc.ca/BeMaObRe/}{\url{http://www.nt.math.ubc.ca/BeMaObRe/}}}
\end{center}
We prove the following results.

\begin{theorem} \label{main psi theorem}
Let $q \ge 3$ be an integer and let $a$ be an integer that is coprime to~$q$. There exist explicit positive constants $c_\psi(q)$ and $x_\psi(q)$ such that
\begin{equation}  \label{ultimate psi bound}
 \bigg| \psi (x; q, a) - \frac{x}{\phi (q)} \bigg| < c_\psi(q) \frac{x}{\log x}  \quad\text{for all } x \ge x_\psi(q).
\end{equation}
Moreover, $c_\psi(q)$ and $x_\psi(q)$ satisfy $c_\psi (q) \le c_0(q)$ and $x_\psi(q) \le x_0(q)$, where
\begin{equation} \label{c_0(q) definition}
c_0(q) = \begin{cases}
\frac{1}{840}, &\text{if } 3 \le q\le 10^4, \\
\frac{1}{160}, &\text{if } q>10^4,
\end{cases}
\end{equation}
and
\begin{equation} \label{x_0(q) definition}
x_0(q) = \begin{cases}
8 \cdot 10^9, &\text{if } 3\le q\le 10^5, \\
\exp(0.03 \sqrt q\log^3q), &\text{if } q>10^5.
\end{cases}
\end{equation}
\end{theorem}

Similarly, for $\theta (x; q, a)$ and $\pi (x; q, a)$ we have:

\begin{theorem} \label{main theta theorem}
Let $q \ge 3$ be an integer and let $a$ be an integer that is coprime to~$q$. There exist explicit positive constants $c_\theta(q)$ and $x_\theta(q)$ such that
\begin{equation}  \label{ultimate theta bound}
 \bigg| \theta (x; q, a) - \frac{x}{\phi (q)} \bigg| < c_\theta(q) \frac{x}{\log x}  \quad\text{for all } x \ge x_\theta(q).
\end{equation}
Moreover, $c_\theta (q) \le c_0(q)$ and $x_\theta(q) \le x_0(q)$, where $c_0(q)$ and $x_0(q)$ are as defined in equations~\eqref{c_0(q) definition} and \eqref{x_0(q) definition}, respectively.
\end{theorem}

\begin{theorem} \label{main pi theorem}
Let $q \ge 3$ be an integer and let $a$ be an integer that is coprime to~$q$. There exist explicit positive constants $c_\pi(q)$ and $x_\pi(q)$ such that
\begin{equation}  \label{ultimate pi bound}
 \bigg| \pi (x; q, a) - \frac{\Li(x)}{\phi (q)} \bigg| < c_\pi(q) \frac{x}{(\log x)^2}  \quad\text{for all } x \ge x_\pi(q).
\end{equation}
Moreover, $c_\pi (q) \le c_0(q)$ and $x_\pi(q) \le x_0(q)$, where $c_0(q)$ and $x_0(q)$ are as defined in equations~\eqref{c_0(q) definition} and \eqref{x_0(q) definition}, respectively.
\end{theorem}

See Appendices~\ref{ssec cpsithetapi} and~\ref{ssec xpsithetapi} for more details on these various constants. We note here that many of our
results, including those stated here, required considerable computations; the relevant computational details are available at 
\begin{center}
\texttt{\href{http://www.nt.math.ubc.ca/BeMaObRe/}{\url{http://www.nt.math.ubc.ca/BeMaObRe/}}}
\end{center}
and are discussed in
 Appendix~\ref{computational appendix}.

 The upper bounds $c_0(q)$ and $x_0(q)$ are, typically, quite far from the actual values of, say, $c_{\theta}(q)$ and $x_{\theta}(q)$. By way of example, for $3 \leq q \leq 10$, we have
 {\small{$$
 \begin{array}{c|c|c|c|c|c|c}
 q & c_{\psi}(q) & c_{\theta}(q) & c_{\pi}(q) & x_{\psi}(q) & x_{\theta}(q) & x_{\pi}(q) \\ \hline
  3  &  0.0003964  &  0.0004015  &  0.0004187 &   576470759   & 7932309757  &  7940618683  \\
   4  &  0.0004770   &  0.0004822  &  0.0005028 &   952930663   & 4800162889  &  5438260589  \\
   5  &  0.0003665  &  0.0003716  &  0.0003876 &   1333804249  & 3374890111  &  3375517771  \\
   6  &  0.0003964  &  0.0004015  &  0.0004187 &   576470831   & 7932309757  &  7940618683  \\
   7  &  0.0004584  &  0.0004657  &  0.0004857 &   686060664   & 1765650541  &  1765715753  \\
   8  &  0.0005742  &  0.0005840   &  0.0006091 &   603874695   & 2261078657  &  2265738169  \\
   9  &  0.0005048  &  0.0005122  &  0.0005342 &   415839496   & 929636413   &  929852953   \\
   10 &  0.0003665  &  0.0003716  &  0.0003876 &   1333804249  & 3374890111  &  3375517771  \\
 \end{array}
 $$}}
For instance, in case $q=3$ and $a\in\{1,2\}$, Theorem \ref{main theta theorem}, using the true values of $c_{\theta}(3)$ and $x_{\theta}(3)$,  rather than their upper bounds $c_0(3)$ and $x_0(3)$, yields the inequality
\begin{equation}  \label{theta q=3}
 \bigg| \theta (x; 3, a) - \frac{x}{2} \bigg| < 4.015 \cdot 10^{-4}  \frac{x}{\log x}  \quad\text{for all } x \ge 7{,}932{,}309{,}757.
 \end{equation}
Here the constant $4.015 \cdot 10^{-4}$ sharpens the corresponding value $0.262$ in Dusart \cite{Du} by a factor of roughly $650$. We remark that $x\ge 7{,}932{,}309{,}757$ is the best-possible range of validity for the error bound~\eqref{theta q=3}; indeed this is true for each $x_\psi(q), x_\theta(q)$, and $x_\pi(q)$, for $3 \leq q \leq 10^5$.

For $3 \leq q \leq 10^5$, we observe that (as a consequence of our proofs), we have
$$
c_\psi(q) \leq  c_\theta(q) \leq  c_\pi(q)  \leq c_0(q).
$$
For larger moduli $q>10^5$, the inequalities
$$
c_\psi(q) \leq c_0(q), \;  c_\theta(q) \leq c_0(q), \; \mbox{ and } \;  c_\pi(q)  \leq c_0(q)
$$
are actual equalities by our definitions of the left-hand sides, and similarly
$$
x_\psi(q) = x_\theta(q) = x_\pi(q) = x_0(q) = \exp(0.03 \sqrt q\log^3q),
$$
for these large moduli. We note that one can obtain a significantly smaller value for $x_0(q)$ if one assumes that Dirichlet $L$-functions modulo $q$ have no exceptional zeros (see Proposition~\ref{no exceptional zero prop}, which sharpens the results of McCurley~\cite{Mc1} mentioned above). Theorems \ref{main psi theorem} and \ref{main theta theorem}, even if one appeals only to the inequalities  $c_\psi(q) \leq c_0(q)$ and $c_\theta(q) \leq c_0(q)$, sharpen Theorem 1 of Ramar\'e and Rumely \cite{RR} for $q \geq 3$ and every other choice of parameter considered therein.

An almost immediate consequence of Theorem \ref{main pi theorem}, just from applying the result for $q=3$ and performing some routine computations (see Appendix \ref{ssec comp tiny theta} for details), is that
\begin{equation} \label{PiLi}
\left| \pi (x) -  \Li (x) \right| < 0.0008375  \frac{x}{\log^2 x} \; \; \mbox{ for all } x \geq 1{,}474{,}279{,}333.
\end{equation}
While, asymptotically, this result is inferior to the state of the art for this problem, it does provide some modest  improvements on results in the recent literature for certain ranges of $x$. By way of example, it provides a stronger error bound than Theorem 2 of Trudgian \cite{Tru16} for all
$1{,}474{,}279{,}333 \leq x < 10^{621}$ (and sharpens corresponding results in \cite{BeDu} and \cite{Du18} in
much smaller ranges).

Exploiting the fact that $\Li(x)$ is predictably close to $x/\log x$, we can readily deduce from Theorem~\ref{main pi
theorem} the following two results, which are proved in Section~\ref{theta to pi section}. We define $p_n(q,a)$ to be the $n$th
smallest prime that is congruent to $a$ modulo~$q$.

\begin{theorem} \label{pi nice lower bound theorem}
Let $q \ge 3$ be an integer, and let $a$ be an integer that is coprime to~$q$. Suppose that $c_\pi(q)\phi(q) < 1$. Then for $x>x_\pi(q)$,
\begin{equation}  \label{pi be big}
\frac{x}{\phi(q)\log x} < \pi(x;q,a) < \frac{x}{\phi(q)\log x} \left(1 + \frac{5}{2\log x} \right)
\end{equation}
\end{theorem}

\noindent 
We remark that Dusart~\cite{Du} proved the lower bound in Theorem~\ref{pi nice lower bound theorem} in the case $q=3$.

\begin{theorem}  \label{pnqa bounds theorem}
Let $q \ge 3$ be an integer, and let $a$ be an integer that is coprime to~$q$. Suppose that $c_\pi(q)\phi(q) < 1$. Then either $p_n(q,a) \leq
x_\pi(q)$ or
\begin{equation}
n\phi(q) \log(n\phi(q)) <  p_n(q,a) < n\phi(q) \big( \log(n\phi(q)) + \tfrac43\log\log(n\phi(q))
\big).
\end{equation}
\end{theorem}

Thanks to our computations of the constants $c_\pi(q)$, we can produce a very explicit version of the above two results for certain moduli~$q$ (see Appendix~\ref{ssec pipn comp} for details).

\begin{cor}\label{cor nice pi pn bounds}
Let $1\le q \le 1200$
  be an integer, and let $a$ be an integer that is coprime to~$q$.
\begin{itemize}
\item For all $x \ge 50q^2$, we have
$$
\frac{x}{\phi(q)\log x} < \pi(x;q,a) < \frac{x}{\phi(q)\log x} \left(1 + \frac{5}{2\log x} \right).
$$
\item For all positive integers $n$ such that $p_n(q,a) \ge 22q^2$, we have
$$
n\phi(q) \log(n\phi(q)) < p_n(q,a) < n\phi(q) \big( \log(n\phi(q)) + \tfrac43\log\log(n\phi(q)) \big).
$$
\end{itemize}
\end{cor}
The lower bounds $50q^2$ and $22q^2$ present here have no especially deep meaning; they simply arise from fitting  envelope functions to the results of routine computations for  $x < x_\pi(q)$ and $1 \leq q \leq 1200$.

Bounds like those provided by Theorems~\ref{main psi theorem},  \ref{main theta theorem}, and~\ref{main pi theorem} are of a reasonable size for most purposes, when combined with tractable auxiliary computations for the range up to~$x_0(q)$. We may, however, weaken the error bounds to produce analogous results that are easier still to use, in that they apply for smaller values of~$x$ (see Section~\ref{ssec comp tiny theta} for the details of the computations involved).

\begin{cor} \label{tiny-theta-cor}
Let $a$ and $q$ be integers with $1 \le q \le 10^5$ and $\gcd (a,q) =1$.
If $x \geq 10^3$, then
\begin{align*}
 \bigg| \psi (x; q, a) - \frac{x}{\phi (q)} \bigg| &< 0.19 \frac{x}{\log x} \\
 \bigg| \theta (x; q, a) - \frac{x}{\phi (q)} \bigg| &< 0.40 \frac{x}{\log x} \\
 \bigg| \pi (x; q, a) - \frac{\Li(x)}{\phi (q)} \bigg| &< 0.53 \frac{x}{\log^2 x}.
\end{align*}
Moreover, if $x \geq 10^6$, then
\begin{align*}
 \bigg| \psi (x; q, a) - \frac{x}{\phi (q)} \bigg| &< 0.011 \frac{x}{\log x} \\
 \bigg| \theta (x; q, a) - \frac{x}{\phi (q)} \bigg| &< 0.024 \frac{x}{\log x} \\
 \bigg| \pi (x; q, a) - \frac{\Li(x)}{\phi (q) } \bigg| &< 0.027 \frac{x}{\log^2 x}.
\end{align*}
\end{cor}


In another direction, if we want somewhat sharper uniform bounds and are willing
to permit the parameter $x$ to be very large, we have the following corollary (see Appendix~\ref{ssec verify cor18} for
details of the computation). We remark that for $q\ge 58$ we can weaken the restriction on $x$ to $x \ge \exp(0.03
\sqrt q\log^3q)$.

\begin{cor} \label{uniform-cor}
Let $a$ and $q$ be integers with $q \geq 3$ and $\gcd (a,q) =1$. Suppose that $x \geq \exp(8 \sqrt q\log^3q)$. Then
$$
\max \left\{  \bigg| \psi (x; q, a) - \frac{x}{\phi (q)} \bigg|,  \bigg| \theta (x; q, a) - \frac{x}{\phi (q)} \bigg| \right\}
< \frac1{160} \frac{x}{\log x}
$$
and
$$
 \bigg| \pi (x; q, a) - \frac{\Li(x)}{\phi (q)} \bigg| < \frac1{160} \frac{x}{\log^2 x}.
$$

\end{cor}

Finally, to complement our main theorems, we should mention one last result, summarizing our computations for ``small'' values of the parameter $x$ (and extending and generalizing Theorem 2 of Ramar\'e and Rumely \cite{RR}) :

\begin{theorem} \label{small-x-theorem}
Let $q$ and $a$ be  integers with $1 \leq q \leq 10^5$ and $\gcd (a,q)=1$, and suppose that $x \leq x_2(q)$, where
\begin{equation}  \label{x2 definition}
  x_2(q) = \begin{cases}
  10^{12} &  \text{if } q=1\\
    x_2(q/2), &\text{if } q \equiv 2 \mod 4 \\
4\cdot 10^{13}, &\text{if } q\in\{3,4,5\}, \\
10^{13}, &\text{if } 5< q\le 100, q \not\equiv 2 \mod 4 \\
10^{12}, &\text{if } 100< q\le 10^4,  q\not \equiv 2 \mod 4\\
10^{11}, &\text{if } 10^4< q\le 10^5, q \not \equiv 2 \mod 4.
\end{cases}
\end{equation}
We have
$$
\max_{1 \leq y \leq x} \left| \psi (y; q, a) - \frac{y}{\phi (q)} \right| \leq 1.745 \sqrt{x},
$$
$$
\max_{1 \leq y \leq x} \left| \theta (y; q, a) - \frac{y}{\phi (q)} \right| \leq 2.072 \sqrt{x}
$$
and
$$
\max_{1 \leq y \leq x} \left| \pi (y; q, a) - \frac{\Li(y)}{\phi (q)} \right| \leq 2.734 \frac{\sqrt{x}}{\log x}.
$$
\end{theorem}
It is worth observing that
the bounds here may be sharpened for (most) individual moduli $q$ (the extremal cases for each function correspond to $q=2$). We provide such bounds and links to related data for moduli $3 \leq q \leq 10^5$ in Appendix~\ref{ssec comp bpsi}.

The outline of the paper is as follows. In Section \ref{Sec2}, we derive an explicit upper bound for
$| \psi (x; q, a) - {x}/{\phi (q)} |$, valid
for the ``small'' moduli $3\le q \le 10^5$.
In Section \ref{Sec4}, this bound is carefully refined into a form which is suitable for explicit calculation; we establish Theorem~\ref{main psi theorem} for these small moduli at the end of Section~\ref{assembly section}. In Section~\ref{Sec5}, we move from bounds for approximating $\psi (x; q, a)$ to analogous bounds for $\theta(x; q, a)$ and $\pi(x; q, a)$. In particular, we establish Theorem~\ref{main theta theorem} for these moduli at the end of Section~\ref{psi to theta section}, and Theorems~\ref{main pi theorem}--\ref{pnqa bounds theorem} for small moduli (as well as Corollary~\ref{cor nice pi pn bounds}) in Section~\ref{theta to pi section}.

Section \ref{Sec6} contains upper bounds for $|\psi(x;q,a) - x/\phi(q)|$, $|\theta(x;q,a) - x/\phi(q)|$, and $|\pi(x;q,a) - \Li(x)/\phi(q)|$ for larger moduli $q>10^5$. We establish Theorems~\ref{main psi theorem} and~\ref{main theta theorem} for these large moduli in Section~\ref{sec53} (see the remark before Corollary~\ref{same result as for small q cor}), and Theorem~\ref{main pi theorem} for these moduli in Section~\ref{sec54}. Indeed, in those sections, we also deduce a number of explicit results with stronger error terms (saving greater powers of $\log x$), as well as analogous results for an improved range of $x$ that hold under the assumption that there are no exceptional zeros for the relevant Dirichlet $L$-functions. Finally, in Appendix~\ref{computational appendix}, we provide details for our explicit computations, with links to files containing all our data. We provide a summary of the notation defined throughout the paper in Appendix~\ref{reference section}.

Before we proceed, a few remarks on our methods are in order. The error terms \eqref{class} depend fundamentally upon the distribution of the zeros of the Riemann zeta function, as evidenced by von Mangoldt's formula:
$$
\lim_{\ep\to0} \frac{\psi(x-\ep)+\psi(x+\ep)}2 = x - \sum_{\rho} \frac{x^{\rho}}{\rho} - \log 2 \pi + \frac{1}{2} \log \left( 1 - \frac{1}{x^2} \right),
$$
where the sum is over the zeros $\rho$ of the Riemann zeta function in the critical strip, in order of increasing $| \mbox{Im } \rho |$. Deriving good approximations for $\psi (x; q, a )$, $\theta (x; q, a )$, and $\pi (x; q, a )$ depends in a similar fashion upon understanding the distribution of the zeros of Dirichlet $L$-functions. Note that, as is traditional in this subject, our approach takes as a starting point von Mangoldt's formula, and hence we are led to initially derive bounds for $\psi(x; q, a )$, from which our estimates for $\theta(x; q, a )$ and $\pi(x; q, a )$ follow. The fundamental arguments providing the connection between zeros of Dirichlet $L$-functions and explicit bounds for error terms in prime counting functions derive from classic work of Rosser and Schoenfeld~\cite{RS}, as extended by McCurley~\cite{Mc1}, and subsequently by Ramar\'e and Rumely~\cite{RR} and Dusart~\cite{Du}. The main ingredients involved include explicit zero-free regions for Dirichlet $L$-functions by Kadiri~\cite{Ka} and McCurley~\cite{Mc3}, explicit estimates for the zero-counting function for Dirichlet $L$-functions by Trudgian~\cite{Tru}, and the results of large-scale computations of Platt~\cite{Pla2}, all of which we cite from the literature. Other necessary results include lower bounds for $L(1,\chi)$ for quadratic characters $\chi$, upper bounds for exceptional zeros of $L$-functions with associated character $\chi$, and explicit inequalities for $b(\chi)$, the constant term in the Laurent expansion of $\frac{L'}L(s,\chi)$ at $s=0$ (see Definition~\ref{mchi bchi def} below). In each of these cases, our results sharpen existing explicit inequalities and thus might be of independent interest:

\begin{prop}  \label{L1chi prop}
If $\chi$ is a primitive quadratic character with conductor $q > 6677$, then
$\displaystyle
L(1,\chi) > \frac{12}{\sqrt q}$.
\end{prop}

\begin{prop}  \label{FLM beta prop}
Let $q\ge3$ be an integer, and let $\chi$ be a quadratic character modulo~$q$. If $\beta > 0$ is a real number for which $L(\beta,\chi)=0$, then
\[
\beta \le 1 - \frac{40}{\sqrt q\log^2 q}.
\]
\end{prop}

\begin{prop}  \label{bchi prop}
Let $q\ge10^5$ be an integer, and let $\chi$ be a Dirichlet character\mod q.
Then $|b(\chi)| \le 0.2515 q \log q$.
\end{prop}

Proposition~\ref{L1chi prop} is established in Section~\ref{magma sec}. For larger values of $q$, we can improve on Proposition~\ref{L1chi prop} by a little more than a factor of~$10$; see Lemma~\ref{L1chi prop} for a more precise statement. Propositions~\ref{FLM beta prop} and~\ref{bchi prop} are established in Sections~\ref{Sec6.1} and~\ref{Sec6.2}, respectively. We also remark that under the assumption that $L(s,\chi)$ has no exceptional zero, our proof would yield a substantially stronger explicit bound of the shape $C\sqrt q\log q$; however, such an improvement is immaterial to our eventual applications. Notice that the conclusion of Proposition~\ref{bchi prop} holds for both primitive and imprimitive
characters $\chi$.

Throughout our work, we have made every effort to avoid specifying many of our ``free'' parameters, such as  a constant $R$ that defines the size of a zero-free region for Dirichlet $L$-functions  (even though, at the end of the day, we do make specific choices of these parameters). The reason for this is to make it easy to sharpen our bounds in the future when one has available stronger zero-free regions (and more computational power). The constants present in, for example, Theorem \ref{main psi theorem}, decrease roughly as a linear function in~$R$. We have chosen to split our ``small $q$'' and ``large $q$'' results at the modulus $q=10^5$ (even though Platt's calculations extend through the modulus $4\cdot10^5$) partially due to limitations of computational time and partially because it is a convenient round number.

\section{Preparation of the upper bound for $|\psi(x;q,a) - x/\phi(q)|$, for $q\le 10^5$} \label{Sec2}

In this section, we will derive our initial upper bound upon $|\psi(x;q,a) - x/\phi(q)|$ for ``small'' moduli~$q$, that is, for $q \leq 10^5$. This bound (given as Proposition \ref{now we see what needs maximizing}) will turn out to be independent of~$x$ except for a single complicated function $\funcF{x}$, defined in Definition~\ref{phim def}, multiplied by various powers of $\log x$. Our starting point is an existing version of the classical explicit formula for $\psi(x;q,a)$ in terms of zeros of Dirichlet $L$-functions; by the end of this section, all dependence on the real parts of these zeros will be removed, and the dependence on their imaginary parts will be confined to the single function $\funcF{x}$.
In this (and, indeed, in subsequent) sections, our operating paradigm
is that any function that can be easily programmed, and whose values can be calculated to arbitrary precision in a negligible amount of time, is suitable for our purposes, even when there remains a layer of notational complexity that we would find difficult to work with analytically. Of course, our choices when we do eventually optimize these various functions are guided by our heuristics (and hindsight) about which pieces of our upper bounds are most significant in the end.

Along the way, we will use as input existing explicit bounds for the number of zeros of  $N(T,\chi)$ (see Proposition~\ref{quoting Trudgian} below), and we will derive an explicit upper bound, contingent on \GRH{1}, for the sum of $1/\sqrt{\beta^2+\gamma^2}$ over all zeros $\beta+i\gamma$ of a given Dirichlet $L$-function (see Lemma~\ref{just the partial summation}). We mention also that the explicit formula we use contains a parameter $\delta$ that can be chosen to be constant to obtain bounds of Chebyshev-type. However, we must choose $\delta$ to be a function of $x$ that decreases to $0$ in order to obtain our bounds of de la Vall\'ee Poussin-type; we make that choice of $\delta$ in displayed equation (\ref{minimizing delta}) (and motivate our choice in the remarks following that equation).

We pause to clarify some terminology and notation.
Throughout this paper, $q$ will be a positive integer (we will usually assume that $q\ge3$), and $a$ will be a positive integer that is relatively prime to~$q$. There are $\phi(q)$ Dirichlet characters with modulus $q$; when we use ``modulus'' or ``$\mod q$'' in this way, we always allow both primitive and imprimitive characters. On the other hand, the conductor of a character is the modulus of the primitive character that induces it, so that the same character can simultaneously have modulus $q$ and conductor $d<q$. For a Dirichlet character $\chi\mod q$, the symbol $q^*$ always denotes the conductor of $\chi$, and $\chi^*$ denotes the primitive character\mod{q^*} that induces~$\chi$.

For any Dirichlet character $\chi\mod q$, the Dirichlet $L$-function is defined as usual by
\begin{equation}  \label{L function def}
L(s,\chi) = \sum_{n=1}^{\infty} \frac{\chi (n)}{n^s}
\end{equation}
when $\Re s>1$, and by analytic continuation for other complex numbers~$s$. We adopt the usual convention of letting $\rho = \beta+i\gamma$ denote a zero of $L(s,\chi)$, so that $\beta=\Re\rho$ and $\gamma=\Im\rho$ by definition; and we define
\begin{equation}  \label{Zchi def}
\Zchi = \{ \rho \in\C \colon 0<\beta < 1, \, L(\rho,\chi)=0 \}
\end{equation}
to be the set of zeros of $L(s,\chi)$ inside the critical strip (technically a multiset, since multiple zeros, if any, are included according to their multiplicity). Notice in particular that the set $\Zchi$ does not include any zeros on the imaginary axis, even when $\chi$ is an imprimitive character; consequently, if $\chi$ is induced by another character $\chi^*$, then $\Zchi=\mathcal{Z}(\chi^*)$.

We recall, by symmetry and the functional equation for Dirichlet $L$-functions, that if $\rho=\beta+i\gamma\in\Zchi$ then also $1-\bar\rho=1-\beta+i\gamma\in\Zchi$.
Finally, we say such an $L$-function satisfies \GRH{H}, the generalized Riemann hypothesis up to height $H$, if
\[
\beta+i\gamma \in \Zchi \text{ and } |\gamma| \le H \implies \beta=\frac12.
\]

\subsection{Previous work based on the explicit formula}

We quote the following proposition from \RR~\cite[Theorem 4.3.1, p.\ 415]{RR}. The proposition, which also appears in Dusart's work~\cite[Theorem 2, pp.\ 1139--40]{Du}, is a modification of McCurley's arguments~\cite[Theorem 3.6]{Mc1} that themselves hearken back to Rosser~\cite{Ro}.

\begin{prop}  \label{quoted prop ABC}
Let $q$ be a positive integer, and let $a$ be an integer that is coprime to~$q$. Let $x>2$ and $H \geq 1$ be real numbers, let $m$ be a positive integer, and let $\delta$ be a real number satisfying $0<\delta<\frac{x-2}{mx}$. Suppose that every Dirichlet $L$-function with modulus $q$ satisfies \GRH{1}. Then
\begin{equation}  \label{two double sums}
\frac{\phi(q)}x \bigg| \psi(x;q,a) - \frac x{\phi(q)} \bigg| <
\funcU{x} + \frac{m\delta}2 + \funcV{x} + \funcW{x},
\end{equation}
where we define
\begin{flalign}
& \label{A def} A_m(\delta) = \frac1{\delta^m} \sum_{j=0}^m \binom mj (1+j\delta)^{m+1} \\
& \label{U def} \funcU{x} = A_m(\delta) \sum_{\chi\mod q} \sum_{\substack{\rho\in\Zchi \\ |\gamma|>H}} \frac{x^{\beta-1}}{|\rho(\rho+1)\cdots(\rho+m)|} \\
& \label{V def} \funcV{x} = \Big(1+\frac{m\delta}2\Big) \sum_{\chi\mod q} \sum_{\substack{\rho\in\Zchi \\ |\gamma|\le H}} \frac{x^{\beta-1}}{|\rho|} \\
& \label{W def} \funcW{x} = \frac{\phi(q)}{x} \bigg( \Big( \frac12 + \sum_{p\mid q} \frac1{p-1} \Big) \log x + 4\log q + 13.4 \bigg).
\end{flalign}
\end{prop}

To offer some context, the genesis of this upper bound is the classical explicit formula for $\psi(x;q,a)$, smoothed by $m$-fold integration over an interval near $x$ of length~$\delta x$. The term $\funcU{x}$ bounds the contribution of the large zeros to this smoothed explicit formula (in which the factor $A_m(\delta)$ arises from some combinatorics of the multiple integration), while the term $\funcV{x}$ bounds the contribution of the small zeros. The term $\frac{m\delta}2$ arises when recovering the original difference $\psi(x;q,a) - x/{\phi(q)}$ from its smoothed version. Finally, \RR\ work only with primitive characters, in contrast to McCurley, to avoid the zeros of $L(s,\chi)$ on the imaginary axis (see \cite[p.\ 399]{RR}, although their remark on \cite[p.\ 414]{RR} is easy to misconstrue). This choice, which we follow (as evidenced by the definition of $\Zchi$ in equation~\eqref{Zchi def}), simplifies the analytic arguments but results in a mild error on the prime-counting side, which is bounded by $\funcW{x}$. In practice, we will be choosing $\delta$ so that the first term $\funcU{x}$ is almost exactly $\frac\delta2$; for most moduli $q$, that term together with the quantity $\frac{m\delta}2$ will provide the dominant contribution to our eventual upper bound. For very small moduli $q$, however, it is the term $\funcV{x}$ that provides the dominant contribution.

We remark that the aforementioned work of Kadiri and Lumley~\cite{KaLu} incorporates a different smoothing mechanism that is inherently more flexible than simple repeated integration; such an approach would be a promising avenue for possible sharpening of our results.

In this upper bound, which is a function of $x$ for any given modulus $q$, the parameters $m$, $\delta$, and $H$ are at
our disposal to choose. We will, in each case, choose $H \leq 10^8/q$, so that every Dirichlet $L$-function with modulus
$q$ satisfies \GRH{H} by Platt's computations~\cite{Pla2}; this choice allows for a strong bound for $\funcV{x}$.
Without some choice of $\delta$ that tended to $0$ as $x$ tends to infinity, it would be
impossible to achieve a de la Vall\'ee Poussin-type bound, because of the term $\frac{m\delta}2$ in the upper bound; our
choice, as it turns out, will be a specific function of $x$ and the other parameters which decays roughly like $\exp(-c\sqrt{\log x})$ for large~$x$. Finally, after the bulk of the work
done to estimate the above upper bound, we will compute the resulting expression for various integer values of $m$ and
select the minimal such value. It will turn out that we always choose
$m\in\{6,7,8,9\}$, for $q \leq 10^5$, although we have no theoretical explanation for how we could have predicted these
choices to be optimal in practice.

\subsection{Some useful facts about the zeros of $L$-functions}

The quantities defined in equations~\eqref{U def} and~\eqref{V def} are both sums over zeros of Dirichlet $L$-functions, and we
will require some knowledge of the distribution of those zeros. That information is essentially all classical, except
that of course we require explicit constants in every estimate, and we can also take advantage of much more extensive
modern computations. Specifically, we draw information from three sources: Trudgian's work on the zeros of the Riemann
$\zeta$-function and Dirichlet $L$-functions with explicit constants, Platt's computations of many zeros of Dirichlet
$L$-functions, and direct computation using Rubinstein's \lcalc\ program~\cite{Rlcalc}.

\begin{definition} \label{NT def}
We write $N(T,\chi)$ for the standard counting function for zeros of $L(s,\chi)$ with $0 < \beta < 1$ and $|\gamma| \leq T$. In other words,
$$
N(T,\chi) = \#\{ \rho\in\Zchi \colon |\gamma| \le T\},
$$
counted with multiplicity if there are any multiple zeros.
\end{definition}

We turn now to explicit bounds for the zero-counting functions $N(T,\chi)$, beginning with a bound when $\chi$ is the principal character.

\begin{prop}  \label{quoting Trudgian principal}
Let $\chi_0$ be the principal character for any modulus $q$. If $T>e$, then
\begin{equation} \label{mike-stuff2}
  \bigg| N(T,\chi_0) - \bigg( \frac T{\pi} \log \frac{T}{2\pi e} + \frac74 \bigg) \bigg| < 0.34\log T + 3.996 + \frac{25}{24\pi T}.
\end{equation}
\end{prop}

\begin{proof}
We adopt the standard notation $N(T)$ for the number of zeros of $\zeta(s)$ in the critical strip whose imaginary part lies between $0$ and $T$, as well as $S(T) = \frac1\pi \arg \zeta(\frac12+iT)$ for the normalized argument of the zeta-function on the critical line. Trudgian~\cite[Theorem 1]{Tru12} gives the explicit estimate
\begin{equation} \label{mike-stuff}
|S(T)| \le 0.17\log T + 1.998,
\end{equation}
valid for $T>e$. It is well known that the error term in the asymptotic formula for $N(T)$ is essentially controlled by $S(T)$; for an explicit version of this relationship, Trudgian~\cite[equation (2.5)]{Tru14} gives
  \[
  \bigg| N(T) - \bigg( \frac T{2\pi} \log \frac{T}{2\pi e} + \frac78 \bigg) \bigg| \le \frac1{4\pi}\arctan\frac1{2T} + \frac T{4\pi}\log\bigg( 1+\frac1{4T^2} \bigg) + \frac1{3\pi T} + |S(T)|
  \]
for $T\ge1$.
In our notation, $N(T,\chi_0)$ is exactly equal to $2N(T)$ (since the former counts zeros lying both above and below the imaginary axis). Using the inequalities $\arctan y \le y$ and $\log(1+y)\le y$ which are valid for $y\ge 0$, it follows from (\ref{mike-stuff})  that the quantity on the left-hand-side of inequality (\ref{mike-stuff2}) is bounded above by twice
$$
 \frac1{4\pi}\arctan\frac1{2T} + \frac T{4\pi}\log\bigg( 1+\frac1{4T^2} \bigg) + \frac1{3\pi T} + 0.17\log T + 1.998
$$
and hence
$$
\bigg| N(T,\chi_0) - \bigg( \frac T{\pi} \log \frac{T}{2\pi e} + \frac74 \bigg) \bigg| \le 2\bigg( \frac1{4\pi}\frac1{2T} + \frac T{4\pi}\frac1{4T^2} + \frac1{3\pi T} + 0.17\log T + 1.998 \bigg),
$$
which is equivalent to the asserted bound.
\end{proof}

\begin{definition} \label{C1C2}
Set $C_1=0.399$ and $C_2=5.338$.
\end{definition}

\begin{prop}  \label{quoting Trudgian}
Let $\chi$ be a character with
conductor $q^*$. If $T\ge1$, then
\begin{equation} \label{pumpkin}
\bigg| N(T,\chi) - \frac T\pi \log  \frac{q^*T}{2\pi e}\bigg| < C_1  \log(q^*T) +C_2.
\end{equation}
\end{prop}

\begin{proof}
If $\chi$ is nonprincipal, this follows immediately from Trudgian~\cite[Theorem 1]{Tru} (which sharpens
McCurley~\cite[Theorem 2.1]{Mc1}). For $\chi$ principal, we have $q^*=1$ and the desired inequality is implied by
Proposition~\ref{quoting Trudgian principal}, provided $T\ge1014$. For $1\le T\le 1014$, we
may verify the bound computationally (see Appendix~\ref{zeta comp sec}), completing the proof.
\end{proof}

It is worth mentioning that the main result of \cite{Tru} contains a number of inequalities like equation~\eqref{pumpkin}, with various values for $C_1$ and $C_2$. The one we have quoted here is the best for small values of $q^*T$, but could be improved for larger $q^*T$; the end result of such a modification to our proof is negligible.

\begin{definition} \label{h3}
We define \[h_3(d) = \begin{cases} 30{,}610{,}046{,}000, & \text{if } d = 1, \\ 10^8/d, & \text{if } 1<d\le 10^5. \end{cases}\]
\end{definition}

Platt~\cite{Pla2} has verified computationally that every Dirichlet $L$-function with conductor $q^* \le 4 \cdot 10^5$ satisfies \GRH{10^8/q^*} (see~\cite{Pla} for more details of these computations). Platt~\cite{Pla3} has also checked that $\zeta(s)$ satisfies \GRH{30{,}610{,}046{,}000}, confirming unpublished work of Gourdon~\cite{Gourdon}. Therefore,

\begin{prop}[Platt] \label{GRH}
Let $\chi$ be a character with conductor $d\le 10^5$. If $\rho=\beta+i\gamma$ is a zero of $L(s,\chi)$ and $|\gamma| \leq h_3(d)$, then $\beta=1/2$.
\end{prop}

\subsection{Upper bounds for $\funcV{x}$, exploiting verification of GRH up to bounded height}

We begin by a standard partial summation argument relating the inner sum in $\funcV{x}$ to the zero-counting function $N(T,\chi)$; we state our result in a form that has some flexibility built in.

\begin{definition}\label{Theta(d,t)}
Let $d$ and $t$ be positive real numbers. We set
  \[\Theta(d,t) =\frac{1}{2\pi} \log ^2\left(\frac{d t}{2 \pi  e}\right)-\frac{C_1 \log (e d t)+C_2}{t},\]
which is a convenient antiderivative of the upper bound implicit in Proposition~\ref{quoting Trudgian}:
  \[\frac\partial{\partial t} \Theta(d,t) = \frac1{t^2}\left(\frac{t}{\pi}\log \frac{dt}{2\pi e}+C_1 \log dt + C_2 \right).\]
\end{definition}

\begin{definition} \label{def: phistar}
Let $\phi^*(d)$ denote the number of primitive characters with modulus~$d$. Thus, $\sum_{d|q} \phi^*(d) = \phi(q)$, and we have the exact formula (see~\cite[page 46]{IK})
  $$ \phi^*(d) = d \prod_{p \parallel d} \left(1-\frac 2p\right) \prod_{p^2|d} \left(1-\frac 1p\right)^2.$$
\end{definition}

\begin{definition}\label{Have you heard the good nus?}
Suppose that $\chi$ is a character with conductor $q^*$. For $H_0 \geq 1$, we define
  \[
    \nu_1(\chi,H_0) = -\Theta(q^*,H_0) -\frac{N(H_0,\chi)}{H_0}+\sum_{\substack{\rho\in{\Zchistar} \\ |\gamma|\le H_0}} \frac{1}{\sqrt{\gamma^2+1/4}},
  \]
while for $0\leq H_0<1$ we define
  \begin{multline*}
    \nu_1(\chi,H_0)
        = -\Theta(q^*,1)
            + \sum_{\substack{\rho\in{\Zchistar} \\ |\gamma|\le H_0}} \frac{1}{\sqrt{\gamma^2+1/4}} \\
            + \left(\frac{1}{\sqrt{H_0^2+1/4}}-1\right)  \left\lfloor \frac 1\pi \log \frac{q^*}{2\pi e}+C_1 \log q^*+C_2 \right\rfloor - \frac{N(H_0,\chi)}{\sqrt{H_0^2+1/4}}.
  \end{multline*}
We further define, for each positive integer $q$ and each function $H_0$ from the set of Dirichlet characters\mod q to the nonnegative real numbers, the functions
  \begin{align*}
    \nu_2(q,H_0) &= \sum_{\chi \mod q} \nu_1(\chi,H_0(\chi)), \\
    \nu_3(q,H) &= - \phi(q)\Big(\frac{1}{2\pi} +\frac{C_1}{H}\Big)+\frac{1}{2\pi}\sum_{d | q} \phi^*(d) \log ^2\Big(\frac{dH}{2 \pi }\Big)\\
  \end{align*}
  and set
  $$
   \nu(q,H_0,H) = \nu_2(q,H_0)+\nu_3(q,H).
   $$
We will limit the abuse of notation by using the function $H_0$ involved in $\nu_2$ and $\nu$ only to fill in the $H_0$-arguments of the function $\nu_1$ in sums over characters.
\end{definition}

\begin{lemma}\label{just the partial summation}
Let $\chi$ be a character with conductor $q^*$, and let $H$ and $H_0$ be real numbers satisfying $H\geq 1$ and $0\leq H_0\leq H$. If $\chi$ satisfies \GRH{\max\{H_0,1\}}, then
  \[
  \sum_{\substack{\rho\in{\Zchi} \\ |\gamma|\le H}} \frac{1}{|\rho|} < \nu_1(\chi,H_0)+\frac{1}{2\pi} \log ^2\Big(\frac{q^*H}{2 \pi }\Big) - \frac{1}{2\pi} -\frac{C_1}{H}.
  \]
\end{lemma}

\begin{proof}
Let $\chi^*$ be the character that induces $\chi$, so that $\Zchi=\Zchistar$. First, we assume that $1 \leq H_0\leq H$. If $|\gamma|\leq H_0$ then $|\rho|=\sqrt{\gamma^2+(1/2)^2}$ by our assumption of \GRH{H_0}; on the other hand, if $|\gamma| > H_0$, then we have the trivial bound $|\rho| > |\gamma|$. As a result,
  \[\sum_{\substack{\rho\in{\Zchi} \\ |\gamma|\le H}} \frac{1}{|\rho|} \le \sum_{\substack{\rho\in{\Zchistar} \\ |\gamma|\le H_0}} \frac{1}{\sqrt{\gamma^2+1/4}} + \sum_{\substack{\rho\in{\Zchistar} \\ H_0<|\gamma|\le H}} \frac{1}{|\gamma|}. \]
Using partial summation,
\begin{align*}
\sum_{\substack{\rho\in{\Zchistar} \\ H_0<|\gamma|\le H}} \frac{1}{|\gamma|}
&= \int_{H_0}^H \frac{dN(T,\chi^*)}{T} \notag \\
&= \frac{N(T,\chi^*)}{T}\bigg|_{H_0}^H - \int_{H_0}^H N(T,\chi^*) \,d\bigg( \frac1T \bigg) \\
&= \frac{N(H,\chi^*)}{H} - \frac{N(H_0,\chi^*)}{H_0} + \int_{H_0}^H \frac{N(T,\chi^*)}{T^2} \,dT.
\end{align*}
We now use Proposition~\ref{quoting Trudgian} and Definition~\ref{Theta(d,t)}:
  \begin{align*}
  \int_{H_0}^H \frac{N(T,\chi^*)}{T^2} \,dT
    &< \int_{H_0}^H \frac1{T^2}\left(\frac{T}{\pi}\log \frac{q^*T}{2\pi e}+C_1 \log q^*T + C_2 \right)\, dT\\
    &= \Theta(q^*,H) - \Theta(q^*,H_0).
  \end{align*}
Proposition~\ref{quoting Trudgian} also gives us
   \[ \frac{N(H,\chi^*)}H < \frac 1\pi \log \frac{q^*H}{2\pi e} + \frac{C_1\log q^*H + C_2}{H},\]
from which it follows, with Definition~\ref{Theta(d,t)}, that
  \[ \frac{N(H,\chi^*)}H  + \Theta(q^*,H) < \frac{1}{2\pi} \log ^2\Big(\frac{q^*H}{2 \pi }\Big) - \frac{1}{2\pi} -\frac{C_1}{H}.\]
Combining these gives us
  \begin{equation}\label{eq:just the big sum}
    \sum_{\substack{\rho\in{\Zchistar} \\ H_0<|\gamma|\le H}} \frac{1}{|\gamma|}
    < - \frac{N(H_0,\chi^*)}{H_0} - \Theta(q^*,H_0) +\frac{1}{2\pi} \log ^2\Big(\frac{q^*H}{2 \pi }\Big) - \frac{1}{2\pi} -\frac{C_1}{H},
  \end{equation}
which, by the definition of $\nu_1(\chi,H_0)$ for $H_0\geq 1$, concludes this case.

We now consider $0\leq H_0 <1$. We need to bound a sum over zeros $\rho=\beta+i\gamma$ with $|\gamma|\leq H$, which we break into three pieces
  \[
   \sum_{\substack{\rho\in{\Zchistar} \\ |\gamma|\le H}} \frac{1}{|\rho|}
   = \sum_{\substack{\rho\in{\Zchistar} \\ |\gamma|\le H_0}} \frac{1}{|\rho|} + \sum_{\substack{\rho\in{\Zchistar} \\ H_0 < |\gamma|\le 1}} \frac{1}{|\rho|} +
   \sum_{\substack{\rho\in{\Zchistar} \\ 1<|\gamma|\le H}} \frac{1}{|\rho|}.
\]
The second sum on the right-hand side has $N(1,\chi)-N(H_0,\chi)$ terms, each of which is bounded by
  \[\frac{1}{|\rho|} \leq \frac{1}{|\gamma|} \leq \frac{1}{\sqrt{H_0^2+1/4}}\]
thanks to \GRH{1}.
The first and third sums on the right-hand side have already been treated in the argument above; in particular, by equation~\eqref{eq:just the big sum},
  \[\sum_{\substack{\rho\in{\Zchistar} \\ 1< |\gamma|\le H}} \frac{1}{|\rho|}
        \leq -N(1,\chi) -\Theta(q^*,1)+ \frac{1}{2\pi} \log ^2\Big(\frac{q^*H}{2 \pi }\Big) - \frac{1}{2\pi} -\frac{C_1}{H}.
  \]
Therefore
  \begin{multline*}
    \sum_{\substack{\rho\in{\Zchistar} \\ |\gamma|\le H}} \frac{1}{|\rho|} \leq
            \sum_{\substack{\rho\in{\Zchistar} \\ |\gamma|\le H_0}} \frac{1}{\sqrt{\gamma^2+1/4}}
            + \frac{N(1,\chi)-N(H_0,\chi)}{\sqrt{H_0^2+1/4}}
            - N(1,\chi) -\Theta(q^*,1)\\ +\frac{1}{2\pi} \log ^2\Big(\frac{q^*H}{2 \pi }\Big) - \frac{1}{2\pi} -\frac{C_1}{H}.
  \end{multline*}
Now by Proposition~\ref{quoting Trudgian},
  \begin{multline*}
    \frac{N(1,\chi)-N(H_0,\chi)}{\sqrt{H_0^2+1/4}}  - N(1,\chi)
      = \left(\frac{1}{\sqrt{H_0^2+1/4}}-1\right) N(1,\chi) - \frac{N(H_0,\chi)}{\sqrt{H_0^2+1/4}} \\
      \le \left(\frac{1}{\sqrt{H_0^2+1/4}}-1\right)  \left\lfloor \frac 1\pi \log \frac{q^*}{2\pi e}+C_1 \log q^*+C_2 \right\rfloor - \frac{N(H_0,\chi)}{\sqrt{H_0^2+1/4}},
  \end{multline*}
and the proof is complete.
\end{proof}

\begin{lemma}\label{lemma: V bound}
Let $q$ and $m$ be positive integers, and $x,\delta,H$ be real numbers satisifying $x>2$ and $0<\delta < \frac{x-2}{mx}$. Let $H_0$ be a function on the characters modulo $q$ satisfying $0\leq H_0(\chi) \le H$. If every Dirichlet $L$-function with modulus $q$ satisfies \GRH{H}, then
  \[\funcV{x} < \Big(1+\frac{m\delta}{2}\Big) \frac{\nu(q,H_0,H)}{\sqrt x}.\]
\end{lemma}

\begin{proof}
By our assumption of \GRH{H}, we have $x^{\beta-1}=x^{-1/2}$, and therefore by Lemma~\ref{just the partial summation},
  \begin{align*}
  \funcV{x}
    &= \left(1+\frac{m\delta}{2}\right) \sum_{\chi \mod q} \sum_{\substack{\rho\in{\Zchi} \\ |\gamma|\le H}} \frac{x^{\beta-1}}{|\rho|} \\
    &= \frac{1+m\delta/2}{\sqrt x}
               \sum_{\chi \mod q} \sum_{\substack{\rho\in{\Zchi} \\ |\gamma|\le H}} \frac{1}{|\rho|} \\
    &< \frac{1+m\delta/2}{\sqrt x}
               \sum_{\chi \mod q} \left( \nu_1(\chi,H_0(\chi))+\frac{1}{2\pi} \log ^2\Big(\frac{q^*H}{2 \pi }\Big) - \frac{1}{2\pi} -\frac{C_1}{H} \right).
  \end{align*}
By Definition~\ref{Have you heard the good nus?},
  \[ \sum_{\chi \mod q} \nu_1(\chi,H_0(\chi)) = \nu_2(q,H_0)\]
and
  \[ \sum_{\chi \mod q} \left(\frac{1}{2\pi} \log ^2\Big(\frac{q^*H}{2 \pi }\Big) - \frac{1}{2\pi} -\frac{C_1}{H} \right) = \nu_3(q,H),\]
concluding this proof, as $\nu(q,H_0,H)=\nu_2(q,H_0)+\nu_3(q,H)$.
\end{proof}

\subsection{Further estimates related to vertical distribution of zeros of Dirichlet $L$-functions}  \label{further estimates section}

We continue by defining certain elementary functions, which we shall use when our analysis calls for upper bounds on the zero-counting functions $N(T,\chi)$ from the previous sections, and establishing some simple inequalities for them.

\begin{definition} \label{LtqH def}
Let $d, u, \ell $ be positive real numbers satisfying $1\leq \ell \le u$.
Define
\[
M_d(\ell,u) = \frac u\pi\log \left( \frac{du}{2\pi e}  \right)- \frac \ell\pi\log \left( \frac{d\ell}{2\pi e} \right) + C_1 \log (d^2 \ell u) + 2C_2,
\]
so that
\begin{equation} \label{ninja}
\frac{\partial}{\partial u} M_d(\ell,u) = \frac1\pi\log \left( \frac{du}{2\pi} \right) + \frac{C_1}{u}.
\end{equation}
Note that for fixed $d$ and $\ell$, we have $M_d(\ell,u) \ll u \log u$.
\end{definition}

Clearly, $N(u,\chi) - N(\ell,\chi)$ counts the number of zeros of $\chi$ with height between $\ell$ and $u$. The following lemma is the reason we have introduced $M_d(\ell,u)$.

\begin{lemma}  \label{N less than L}
Let $\chi$ be a character with conductor $d$, and let $\ell$ and $u$ be real numbers satisfying $1 \le \ell \le u$. Then $N(u,\chi) - N(\ell,\chi) < M_d(\ell,u)$.
\end{lemma}

\begin{proof}
The assertion follows immediately from subtracting the two inequalities
  \begin{align*}
    N(u,\chi) &< \frac u\pi \log \frac{d u}{2\pi e} + C_1 \log du + C_2 \\
    N(\ell,\chi) &> \frac \ell\pi \log \frac{d \ell}{2\pi e} - C_1 \log d\ell - C_2,
  \end{align*}
each of which is implied by Proposition~\ref{quoting Trudgian}.
\end{proof}

\begin{lemma} \label{counting zeros simpler upper bounds lemma}
Let $d, u$ and $\ell$ be real numbers satisfying $d \ge 1$ and $15 \le \ell \le u$. Then
\(
M_d(\ell,u) <\frac u\pi \log du.
\)
\end{lemma}

\begin{proof}
Set
\begin{align*}
\epsilon &= \pi \left( \frac u\pi \log(du)-M_d(\ell,u) \right) \\
&= u \log(2\pi e) - 2C_2 \pi - C_1 \pi \log (d^2 \ell u)
             + \ell \log\left(\tfrac{\ell d}{2\pi e}\right),
\end{align*}
so that we need to prove that $\epsilon>0$. First, we have
  \[ \frac{\partial \epsilon}{\partial u} = \log (2  \pi e)-\frac{C_1 \pi}{u}, \qquad\frac{\partial \epsilon }{\partial d} = \frac{\ell}{d}-\frac{2 C_1 \pi}{d},\]
which are positive for $u>C_1 \pi / \log (2\pi e) \approx 0.44$ and $\ell>2C_1 \pi \approx 2.51$, while by hypothesis $u \ge \ell \ge 15$. Thus, we may assume that $u=\ell$ and $d=1$. We then have
 \[\epsilon = (\ell - 2 C_1\pi) \log \ell -2 C_2 \pi,\]
which is clearly an increasing function of $\ell $ and is already positive at $\ell=15$.
\end{proof}

\subsection{Preliminary statement of the upper bound for $|\psi(x;q,a) - x/\phi(q)|$}

Our remaining goal for this section is to establish Proposition~\ref{now we see what needs maximizing}, which is an upper bound for $|\psi(x;q,a) - x/\phi(q)|$ in which the dependence on $x$ has been confined to functions of a single type (to be defined momentarily). Building upon the work of the previous two sections, we invoke certain hypotheses on the horizontal distribution of the zeros of Dirichlet $L$-functions to estimate many of the terms in the upper bound of Proposition~\ref{quoted prop ABC}. We have left these hypotheses in parametric form for much of this paper, in order to facilitate the incorporation of future improvements; for our present purposes, we shall be citing work of Platt and Kadiri (see Proposition~\ref{R value prop}) to confirm the hypotheses for certain values of the parameters.

\begin{definition} \label{Upsilon and Psi def}
Let $q$ be a positive integer, and let $m$, $r$, $x$, and $H$ be positive real numbers satisfying $x\ge 1$ and $H\geq 1$. Define
\begin{align*}
\Upsilon_{q,m}(x;H) &= \sum_{\chi\mod q} \sum_{\substack{\rho\in\Zchi \\ |\gamma|>H}} \frac{x^{\beta-1}}{|\gamma|^{m+1}} \\
\Psi_{q,m,r}(x;H) &= H^{m+1} \Upsilon_{q,m}(x;H) (\log x)^r.
\end{align*}
\end{definition}

\begin{definition} \label{H_0}
  For integers $m$ with $3\leq m\le 25$, define real numbers $H_1(m)$ according to the following table:
$$
\begin{array}{|c||c|c|c|c|c|c|c|c|} \hline
        m & 3      & 4      & 5     & 6     & 7 & 8 & 9 & \ge 10\\ \hline
H_1(m) & 1011 & 391 & 231 & 168 & 137 & 120 & 109 & 102\\ \hline
\end{array}
$$
\end{definition}
For the values of $m$ we will actually choose, later in this paper, we note that the product  $m H_1(m)$ is roughly constant (and somewhat less than $1000$).

\begin{lemma}\label{lemma: Upsilon bound (easy)}
Let $q$ and $m$ be integers satisfying
$3\le q \le 10^5$ and $3\leq m \le 25$, and let $x$ and $H$ be real
numbers satisfying $x \ge 1000$ and $H \ge H_1(m)$. Then
  \[\Upsilon_{q,m}(x;H) < \left(\frac{x-2}{2mx}\right)^{m+1}.\]
\end{lemma}

\begin{proof}
Since $\beta<1$ for every $\rho=\beta+i\gamma \in \Zchi$, we have by partial summation
  \begin{align*}
\sum_{\substack{\rho\in\Zchi \\ |\gamma|>H}} \frac{x^{\beta-1}}{|\gamma|^{m+1}}
    &< \sum_{\substack{\rho\in\Zchi \\ |\gamma|>H}} \frac{1}{|\gamma|^{m+1}} \\
    &= \int_H^{\infty} \frac{d(N(u,\chi)-N(H,\chi))}{u^{m+1}}\,du \\
    &= \left. \frac{N(u,\chi)-N(H,\chi)}{u^{m+1}} \right|_H^\infty
                 + (m+1) \int_H^\infty \frac{N(u,\chi)-N(H,\chi)}{u^{m+2}} \, du \\
    &= (m+1) \int_H^\infty \frac{N(u,\chi)-N(H,\chi)}{u^{m+2}} \, du,
  \end{align*}
since $N(u,\chi)-N(H,\chi) \leq N(u,\chi) \ll u \log u$.
From the assumption that $H\ge100>15$, Lemmas~\ref{N less than L} and~\ref{counting zeros simpler upper bounds lemma} thus imply the inequalities
$$
N(u,\chi)-N(H,\chi) < \frac{u}{\pi} \log (q^*u) \le \frac{u}{\pi} \log (qu)
$$
(where $q^*$ is the conductor of $\chi$), whereby
  \begin{align}
    \Upsilon_{q,m}(x;H)
      &< \sum_{\chi\mod q} \frac{m+1}\pi \int_H^\infty \frac{u \log (qu)}{u^{m+2}}\, du \notag\\
      &= \frac{\phi(q)}{H^m} \frac{m+1}\pi \frac{m\log qH+1}{m^2} \notag\\
      &\le \frac{10^5}{H^m} \frac{m+1}\pi \frac{m\log(10^2H)+1}{m^2} \label{eq:qgone}\\
      &< \frac{10^5}{100^m} \frac{m+1}\pi \frac{m\log(10^7)+1}{m^2} \notag
   \end{align}
by monotonicity in $H$ and~$q$. On the other hand, monotonicity also implies that
  \[\left(\frac{x-2}{2mx}\right)^{m+1} \ge \left(\frac{499}{1000m}\right)^{m+1}\]
for $x\ge1000$. It therefore suffices to check that
  \[
  \frac{10^5}{400^m} \frac{m+1}\pi \frac{m\log(10^7)+1}{m^2} < \left(\frac{499}{1000m}\right)^{m+1}
  \]
for $11 \leq m \leq 25$, which is a simple exercise.

For each $m$ between $3$ and $10$, we carry on from line~\eqref{eq:qgone}, using $H\geq H_1(m)$, but otherwise continuing in the same way.
\end{proof}

At this point, we rewrite Proposition~\ref{quoted prop ABC}, with a particular choice for $\delta$ and some other manipulations that, with foresight, are helpful.
\begin{definition}  \label{alpha def}
Let $m$ be a positive integer and $\delta$ a positive real number. We set $\alpha_{m,0} = 2^m$ and, for $1\le k\le m+1$,
\[
\alpha_{m,k} = \binom{m+1}k \sum_{j=0}^m \binom mj j^k.
\]
We note that
\[
A_m(\delta) = \sum_{k=0}^{m+1} \alpha_{m,k} \delta^{k-m}.
\]
\end{definition}

\begin{prop}\label{now we see what needs maximizing}
Let $q$ and $m$ be integers satisfying $3\le q \le 10^5$ and $3\le m \le 25$, and let $a$ be an integer that is coprime to~$q$. Let $x$, $x_2$, and $H$ be real numbers with $x\ge x_2 \ge 1000$ and $H \ge H_1(m)$.
Let $H_0$ be a function on the characters modulo $q$ with $0\leq H_0(\chi) \leq H$ for every such character. If every Dirichlet $L$-function with modulus $q$ satisfies \GRH{H}, then
\begin{flalign}  
& \frac{\phi(q)}x \bigg| \psi(x;q,a) - \frac x{\phi(q)} \bigg| \log x \notag\\
   & < \funcW{x_2}\log x_2 + \nu(q,H_0,H) \frac{\log x_2}{ \sqrt{x_2}} \label{eq:take 2, line2}\\
   & + \frac mH \Psi_{q,m,m+1}(x;H)^{\frac 1{m+1}} \left(1+\frac{\nu(q,H_0,H)}{\sqrt {x_2}}\right)^{\frac{m}{m+1}} \label{eq:take 2, line3} \\
   & + \sum_{k=0}^m \frac{\alpha_{m,k}}{2^{m-k}H^{k+1}}\Psi_{q,m,\frac{m+1}{k+1}}(x;H)^{\frac{k+1}{m+1}}
            \left(1+\frac{\nu(q,H_0,H)}{\sqrt {x_2}}\right)^{\frac{m-k}{m+1}} \label{eq:take 2, line4}\\
   & +\frac{2\alpha_{m,m+1}}{H^{m+2}} \Psi_{q,m,\frac{m+1}{m+2}}(x;H)^{\frac{m+2}{m+1}}. \label{eq:take 2, line5}
\end{flalign}
\end{prop}

\noindent
We note in passing that since $\alpha_{m,0}=2^m$, the term on line~\eqref{eq:take 2, line3} is identical to the $k=0$ term on line~\eqref{eq:take 2, line4} except for the factor of $m$ on the former line. We will combine these terms together in the analogous Definition~\ref{D def} below.

\begin{proof}
Our starting point is Proposition~\ref{quoted prop ABC}: for any real number $0<\delta<\frac{x-2}{mx}$,
\[
\frac{\phi(q)}x \bigg| \psi(x;q,a) - \frac x{\phi(q)} \bigg| < \funcU{x} + \frac{m\delta}2 + \funcV{x} + \funcW{x},
\]
where the notation is defined in equations~\eqref{A def}--\eqref{W def}.
Since trivially
\[
\sum_{\substack{\rho\in\Zchi \\ |\gamma|>H}} \frac{x^{\beta-1}}{|\rho(\rho+1)\cdots(\rho+m)|} < \sum_{\substack{\rho\in\Zchi \\ |\gamma|>H}} \frac{x^{\beta-1}}{|\gamma|^{m+1}},
\]
a comparison of equation~\eqref{U def} and Definition~\ref{Upsilon and Psi def} shows that
$$
\funcU{x} < A_m(\delta) \Upsilon_{q,m}(x;H).
$$
Using Lemma~\ref{lemma: V bound} to bound $\funcV{x}$, we therefore have
\begin{multline*}
  \frac{\phi(q)}x \bigg| \psi(x;q,a) - \frac x{\phi(q)} \bigg| \log x < A_m(\delta) \Upsilon_{q,m}(x;H) \log x+ \frac{m\delta}2\log x \\
  + \left(1+\frac{m\delta}{2}\right) \frac{\nu(q,H_0,H)}{\sqrt x}\log x + \funcW{x}\log x,
\end{multline*}
which we rewrite as
\begin{multline}
\frac{\phi(q)}x \bigg| \psi(x;q,a) - \frac x{\phi(q)} \bigg| \log x < \funcW{x}\log x + \nu(q,H_0,H) \frac{\log x}{\sqrt x} \\
+ m \left(1+\frac{\nu(q,H_0,H)}{\sqrt x}\right)\frac{\delta\log x}{2} + A_m(\delta) \Upsilon_{q,m}(x;H) \log x. \label{rewritten second line}
\end{multline}

It is easily seen from its definition~\eqref{W def} that $\funcW{x}\log x$, much like the function $(\log x)^2/x$, is decreasing for $x\ge1000>e^2$, and the same is true for $(\log x)/\sqrt x$. Therefore
$$
\funcW{x}\log x + \nu(q,H_0,H) \log x / \sqrt x \le \funcW{x_2}\log x_2 + \nu(q,H_0,H) (\log x_2)/\sqrt{x_2},
$$
which yields the terms on line~\eqref{eq:take 2, line2}.

We now set
\begin{equation}  \label{minimizing delta}
\delta = 2 \left( \frac{ \Upsilon_{q,m}(x;H)}{1+\nu(q,H_0,H)/\sqrt x}\right)^{\frac{1}{m+1}}.
\end{equation}
Our motivation for this choice is as follows. To achieve a de la Vall\'ee Poussin-type bound, we must choose $\delta$ tending to $0$ as $x$ increases. Since $A_m(\delta) \sim (2/\delta)^m$ when $\delta\to0$, we choose the value of $\delta$ that minimizes
  \[ m \left(1+\frac{\nu(q,H_0,H)}{\sqrt x}\right)\frac{\delta\log x}{2}+\left(\frac 2\delta \right)^m \Upsilon_{q,m}(x;H) \log x,\]
which is easily checked to be the right-hand side of equation~\eqref{minimizing delta}.
This value of $\delta$ is clearly positive, and Lemma~\ref{lemma: Upsilon bound (easy)} implies that $\delta<\frac{x-2}{mx}$; hence this $\delta$ is a valid choice. We now have
  \begin{flalign*}
&   m  \left(1+\frac{\nu(q,H_0,H)}{\sqrt x}\right)\frac{\delta\log x}{2} \\
   &= m \left(1+\frac{\nu(q,H_0,H)}{\sqrt x}\right) \left( \frac{ \Upsilon_{q,m}(x;H)}{1+\nu(q,H_0,H)/\sqrt x}\right)^{\frac{1}{m+1}} \log x \\
   &=\frac mH \big( H^{m+1} \Upsilon_{q,m}(x;H) \log^{m+1} x \big)^{\frac 1{m+1}}\left(1+\frac{\nu(q,H_0,H)}{\sqrt x}\right)^{\frac{m}{m+1}} \\
   &=\frac mH \Psi_{q,m,r}(x;H)^{\frac 1{m+1}}\left(1+\frac{\nu(q,H_0,H)}{\sqrt x}\right)^{\frac{m}{m+1}}
  \end{flalign*}
by Definition~\ref{Upsilon and Psi def}.
Certainly
$$
1+\nu(q,H_0,H)/\sqrt{x} \le 1+\nu(q,H_0,H)/\sqrt{x_2}
$$
for $x\ge x_2$, and therefore the first term on line~\eqref{rewritten second line} can be bounded above by the term on line~\eqref{eq:take 2, line3}.

Lastly, from Definition~\ref{alpha def},
  \begin{flalign*}
  & A_m(\delta)  \Upsilon_{q,m}(x;H) \log x \\
    &= \left( \sum_{k=0}^{m+1} \alpha_{m,k} \delta^{k-m} \right) \Upsilon_{q,m}(x;H) \log x  \\
    &= \left( \sum_{k=0}^{m+1} \alpha_{m,k} \left(2 \left( \frac{\Upsilon_{q,m}(x;H)}{1+\nu(q,H_0,H)/\sqrt x}\right)^{\frac{1}{m+1}} \right)^{k-m} \right) \Upsilon_{q,m}(x;H) \log x \\
    &= \sum_{k=0}^{m+1} \frac{\alpha_{m,k}}{2^{m-k}} \big( \Upsilon_{q,m}(x;H) (\log x)^{\frac{m+1}{k+1}} \big)^{\frac{k+1}{m+1}}
            \left(1+\frac{\nu(q,H_0,H)}{\sqrt x}\right)^{\frac{m-k}{m+1}}\\
    &=  \sum_{k=0}^{m+1} \frac{\alpha_{m,k}}{2^{m-k}H^{k+1}} \Psi_{q,m,\frac{m+1}{k+1}}(x;H)^{\frac{k+1}{m+1}}
            \left(1+\frac{\nu(q,H_0,H)}{\sqrt x}\right)^{\frac{m-k}{m+1}}
  \end{flalign*}
by Definition~\ref{Upsilon and Psi def}. For $0\le k \le m$, the factor $\left(1+{\nu(q,H_0,H)}/{\sqrt x}\right)^{\frac{m-k}{m+1}}$ is nonincreasing, hence is bounded by $\left(1+{\nu(q,H_0,H)}/{\sqrt{x_2}}\right)^{\frac{m-k}{m+1}}$, which accounts for the terms on line~\eqref{eq:take 2, line4}. Finally, when $k=m+1$, this factor is increasing but is bounded by $1$, which accounts for the term on line~\eqref{eq:take 2, line5}, thus completing the proof.
\end{proof}

Of note in Proposition~\ref{now we see what needs maximizing} is that the bound is independent  of $x$ except in the form of the terms $\Psi_{q,m,r}(x;H)$ for various values $\frac 45 \leq r \le m+1$. The next two sections are devoted to bounding functions of this form; those bounds will be inserted into the conclusion of Proposition~\ref{now we see what needs maximizing} at the end of Section~\ref{Sec4}, at which point we will be able to prove Theorem~\ref{main psi theorem} for moduli $q$ up to~$10^5$.

\section{Elimination of explicit dependence on zeros of Dirichlet $L$-functions}  \label{eliminate zeros sec}

From the work of the preceding section, it remains to establish an upper bound for the function $\Psi_{q,m,r}(x;H)$ that does not depend upon specific knowledge of the zeros of a given Dirichlet $L$-function. To achieve this, we will appeal to a zero-free region for such functions, together with estimates for $N(T,\chi)$.

\subsection{Estimates using a zero-free region for $L(s,\chi)$}

\begin{definition}  \label{hypothesis z}

Given positive real numbers $H_2$ and $R$, we say that a character $\chi$ with conductor $q^*$ satisfies Hypothesis Z$(H_2,R)$ if
every nontrivial zero $\beta+i\gamma$ of $L(s,\chi)$ satisfies either
\[
|\gamma| \le H_2 \text{ and } \beta=\tfrac12,    \qquad\text{or}\qquad  |\gamma| > H_2 \text{ and } \beta \le 1 - \frac1{R\log(q^* |\gamma|)}.
\]
In other words, zeros with small imaginary part (less than $H_2$ in absolute value) lie on the critical line, while zeros with large imaginary part lie outside an explicit zero-free region.

We say that a modulus $q$ satisfies Hypothesis Z${}_1(R)$ if every nontrivial zero $\beta+i\gamma$ of every Dirichlet $L$-function modulo $q$ satisfies
\[
\beta \le 1 - \frac1{R\log(q\max\{1,|\gamma|\})},
\]
except possibly for a single ``exceptional'' zero (which, as usual, will necessarily be a real zero of an $L$-function corresponding to a quadratic character---see~\cite[Sections 11.1--11.2]{MV}).
\end{definition}

\begin{definition} \label{phim def}
Let $m$ and $d$ be positive integers, and let $R, H, H_2, x$ and $u$ be positive real numbers satisfying $1 \le H \le H_2$. Let $\chi$ be a character with conductor $q^*$. Define the functions
\begin{align*}
g_{d,m}^{(1)}(H,H_2) &= \frac{H}{\pi  m^2}  \left((1+m \log \frac{d H}{2 \pi })-\big(\tfrac{H}{H_2}\big)^{m} (1+m \log\frac{d H_2}{2\pi})\right) \\
        &\qquad +\left( 2 \log (d H)+\frac{1}{m+1}\left(1-\big(\tfrac{H}{H_2}\big)^{m+1}\right) \right) C_1 + 2C_2 \\
g_{d,m}^{(2)}(H,H_2) &= \big(\tfrac{H}{H_2}\big)^m \frac{H}{2\pi m^2} \left(1+m\log \frac{d H_2}{2\pi}\right)  \\
        &\qquad +\big(\tfrac{H}{H_2}\big)^{m+1} \left(\frac{1}{2(m+1)}+ \log d H_2  \right) C_1 + \big(\tfrac{H}{H_2}\big)^{m+1} C_2 \\
g_{d,m,R}^{(3)}(x;H ,H_2) &= g_{d,m}^{(1)}(H,H_2) \cdot \frac{1}{x^{1/2}}+ g_{d,m}^{(2)}(H,H_2) \cdot \frac{x^{1/(R\log dH_2)}}{x} .
  \end{align*}
Further define
$$
\funcY{u} = u^{-(m+1)} x^{-1/(R \log du)} =\frac1{u^{m+1}} \exp\bigg( {-}\frac{\log x}{R\log du} \bigg) 
$$
and
$$
\funcF[q=\chi]{x} = \sum_{\substack{\rho\in\Zchi \\ |\gamma|>H_2}} \funcY[q=q^*]{|\gamma|}.
$$
Note that all of these functions are strictly positive.
\end{definition}

\begin{definition} \label{def: G}
Let $q$ and $m$ be positive integers, let $R, H, x$ and $u$ be positive real numbers with $H\ge1$, and let $H_2$ be a function on the divisors of $q$ satisfying $1 \leq H \le H_2(d)$ for $d\mid q$. Define
$$
\funcF[q=d]{x} = H^{m+1} \sum_{\substack{\chi \mod q \\ q^*=d}} \funcF[q=\chi,H={H_2(d)}]{x}
$$
and
  \begin{align*}
    \funcG{x} &= \sum_{\chi\mod q} \bigg( g_{q^*,m,R}^{(3)}(x;H ,H_2(q^*)) + \frac{H^{m+1}}2 \funcF[q=\chi,H={H_2(q^*)}]{x} \bigg) \\
    &= \sum_{d|q} \bigg( \phi^*(d) g_{d,m,R}^{(3)}(x;H ,H_2(d)) + \frac12 \funcF[q=d,H={H_2(d)}]{x} \bigg).
  \end{align*}
As before, we will use the function $H_2$ involved in $F_{d,m,R}$ and $G_{q,m,R}$ only to fill in the $H_2$-arguments of the functions defined earlier in this section.
\end{definition}

\begin{lemma} \label{Upsilon breakdown}
Let $q$ and $m$ be positive integers. Let $x,H$ and $R$ be real numbers satisfying $x>1$ and $H\geq 1$, and let $H_2$ be a function on the divisors of $q$ satisfying $H \le H_2(d)$ for $d\mid q$. Suppose that every character $\chi$ with modulus $q$ satisfies Hypothesis Z$(H_2(q^*),R)$, where $q^*$ is the conductor of~$\chi$. Then
  \[H^{m+1} \Upsilon_{q,m}(x;H) < \funcG{x}.\]
\end{lemma}

\begin{proof}
Note that it suffices, for a fixed character $\chi$ with conductor $d$, to establish the upper bound
 \begin{equation}\label{eq:U guts}
  \sum_{\substack{\rho\in\Zchi \\ |\gamma|>H}} \frac{x^{\beta-1}}{|\gamma|^{m+1}}
    < \frac{g_{d,m,R}^{(3)}(x;H ,H_2(d))}{H^{m+1}} + \frac 12 \funcF[H=H_2(d)]{x},
  \end{equation}
since multiplying by $H^{m+1}$ and summing this bound over all characters modulo~$q$ yields the statement of the proposition, by comparison to Definition~\ref{def: G}. We begin by using Hypothesis Z$(H_2(d),R)$ to write
  \begin{equation} \label{above and below H2d}
\sum_{\substack{\rho\in\Zchi \\ |\gamma|>H}} \frac{x^{\beta-1}}{|\gamma|^{m+1}}
= \frac{1}{\sqrt x} \sum_{\substack{\rho\in\Zchi \\ H <  |\gamma| \leq H_2(d)}} \frac{1}{|\gamma|^{m+1}}+ \frac 1x \sum_{\substack{\rho\in\Zchi \\ |\gamma| > H_2(d)}} \frac{x^{\beta}}{|\gamma|^{m+1}}.
  \end{equation}
By partial summation, integration by parts, and Lemma~\ref{N less than L}, we find that
  \begin{flalign} 
  &  \sum_{\substack{\rho\in\Zchi \\ H <  |\gamma|\leq H_2(d)}} \frac{1}{|\gamma|^{m+1}} \notag \\
  &  = \int_H^{H_2(d)} \frac{d(N(t,\chi) - N(H,\chi))}{t^{m+1}} \notag \\ 
  &  = \frac{N(H_2(d),\chi) - N(H,\chi)}{H_2(d)^{m+1}}+(m+1)\int_H^{H_2(d)} \frac{N(t,\chi)-N(H,\chi) }{ t^{m+2} } \,dt \notag \\ 
  &  < \frac{M_{d}(H,H_2(d))}{H_2(d)^{m+1}} + (m+1) \int_H^{H_2(d)} \frac{M_{d}(H,t)}{t^{m+2}} \,dt \notag \\ 
  &  = \frac{g_{d,m}^{(1)}(H,H_2(d))}{H^{m+1}}, \label{emus and llamas} 
  \end{flalign}
where the last equality follows from Definitions~\ref{LtqH def} and~\ref{phim def} and tedious but straightforward calculus.

We now turn to the zeros with height above $H_2(d)$, making use of the fact that $\beta+i\gamma$ is a nontrivial zero of $L(s,\chi)$ if and only if $1-\beta+i\gamma$ is such a zero, by the functional equation.
Consequently,
\begin{flalign*}
 \sum_{\substack{\rho\in\Zchi \\ |\gamma| > H_2(d)}} \frac{x^{\beta}}{|\gamma|^{m+1}} & = \frac12 \bigg( \sum_{\substack{\rho\in\Zchi \\ |\gamma| > H_2(d)}} \frac{x^{\beta}}{|\gamma|^{m+1}} + \sum_{\substack{\rho\in\Zchi \\ |\gamma| > H_2(d)}} \frac{x^{1-\beta}}{|\gamma|^{m+1}} \bigg) \\ 
 & = \frac 12 \sum_{\substack{\rho\in\Zchi \\ |\gamma| > H_2(d)}} \frac{x^{\beta}+x^{1-\beta}}{|\gamma|^{m+1}},
 \end{flalign*} 
since the two sums inside the parentheses are equal to each other.
For a fixed $x>1$, the function $x^\beta+x^{1-\beta}$ increases as $\beta$ moves away from $\frac12$ in either direction; and by Hypothesis Z$(H_2(d),R)$,
  \[\frac{1}{R \log d |\gamma|} \leq \min\{\beta,1-\beta\} \leq \max \{\beta,1-\beta\} \leq 1-\frac{1}{R \log d |\gamma|}.\]
Therefore,
  \begin{align*}
    \frac12 \sum_{\substack{\rho\in\Zchi \\ |\gamma| > H_2(d)}} \frac{x^{\beta}+x^{1-\beta}}{|\gamma|^{m+1}}
      \leq \frac 12  \sum_{\substack{\rho\in\Zchi \\ |\gamma| > H_2(d)}} \frac{x^{1/(R\log d|\gamma|)}+{x^{1-1/(R\log d|\gamma|)}}}{|\gamma|^{m+1}} \\
    = \frac{x^{1/(R\log dH_2(d))}}2 \sum_{\substack{\rho\in\Zchi \\ |\gamma| > H_2(d)}} \frac{1}{|\gamma|^{m+1}} + \frac x2 \funcF[q=\chi,H=H_2(d)]{x}.
  \end{align*}
Again by partial summation and some tedious calculus,
  \[
   \frac12 \sum_{\substack{\rho\in\Zchi \\ |\gamma| > H_2(d)}} \frac{1}{|\gamma|^{m+1}}
    < \frac{m+1}2 \int_{H_2(d)}^\infty \frac{ M_d(H_2(d),t) }{ t^{m+2} }\,dt = \frac{g_{d,m}^{(2)}(H,H_2(d))}{H^{m+1}},
    \]
from which we conclude that
\[
\frac 1x \sum_{\substack{\rho\in\Zchi \\ |\gamma| > H_2(d)}} \frac{x^{\beta}}{|\gamma|^{m+1}} < \frac{x^{1/(R\log dH_2(d))}}{x} \frac{g_{d,m}^{(2)}(H,H_2(d))}{H^{m+1}} + \frac 12 \funcF[q=\chi,H=H_2(d)]{x}.
\]
Combining this upper bound with equation~\eqref{above and below H2d} and inequality~\eqref{emus and llamas} establishes inequality~\eqref{eq:U guts}, thanks to Definition~\ref{phim def}, and thus completes the proof of the lemma.
\end{proof}

To turn Proposition \ref{now we see what needs maximizing} into something amenable to computation, in light of Lemma~\ref{Upsilon breakdown}, we are left with the problem of deriving an absolute upper bound for the quantity
  \[\Psi_{q,m,r}(x;H) = H^{m+1} \Upsilon_{q,m}(x;H) (\log x)^r\]
for various positive $r$; we will eventually obtain this in Proposition~\ref{Upsilon resovled}.
As
$$
g_{d,m,R}^{(3)}(x;H ,H_2(d)) = O(1/\sqrt x),
$$
it is an easy matter to majorize $g^{(3)} (\log x)^r$ for any $r$. The problem that remains, therefore,  is to deduce a bound upon
$$
\funcF{x} (\log x)^r,
$$
for various $r$.
Our bounds for this function consist of several pieces, each of which can be optimized using calculus; we simply add the individual maxima together to deduce a uniform upper bound for $\funcF{x}(\log x)^r$. That optimization, however, can only take place once we have provided bounds of a simpler form for these pieces.

\subsection{Conversion to integrals involving bounds for $N(T,\chi)$}  \label{elimination sec}

As we see in Definition~\ref{phim def}, the function $\funcF{x}$ still depends on the vertical distribution of zeros of Dirichlet $L$-functions\mod q. A standard partial summation argument, combined with the bounds on $N(T,\chi)$ we established in Section~\ref{further estimates section}, allows us to remove that dependence on zeros of $L$-functions in favor of more elementary functions.

\begin{definition} \label{H1 and H2 def}
Let $d$ and $m$ be positive integers, and suppose that $R>0$, $x\ge1$ and $H_2 \ge 1$ are real numbers. Define
\[
\funcHone{x} = \frac{1}d \exp\bigg( \sqrt{\frac{\log x}{R(m+1)}} \bigg)
\]
and
  \begin{align*}\funcHtwo{x} &= \max\{H_2,\funcHone{x}\} \\
 &= \begin{cases}
  H_2, & \text{if } 1 \leq x \leq \exp\left(R(m+1)\log^2(dH_2)\right), \\
  \funcHone{x}, & \text{if } x \geq \exp\left(R(m+1)\log^2(d H_2)\right).
\end{cases}\end{align*}
Straightforward calculus demonstrates that the function $\funcY{u}$ from Definition~\ref{phim def} is, as a function of $u$, increasing for $1/q<u<\funcHone{x}$ and decreasing for $u>\funcHone{x}$,
\end{definition}

\begin{prop} \label{integral bound for S tilde prop}
Let $m$ and $d$ be positive integers, let $H$, $H_2$, and $R$ be positive real numbers satisfying $1 \le H \le H_2$, and let $\chi$ be a character with conductor $d$ satisfying Hypothesis~Z$(H_2,R)$. Then
\begin{multline}  \label{first two terms}
\funcF{x} \le M_d(H_2,\funcHtwo{x}) \funcY{\funcHtwo{x}} \\
+ \int_{H_2}^\infty \left(\frac{\partial}{\partial u} M_d(H_2,u)\right) \funcY{u} \,du,
\end{multline}
where $\funcF{x}$ and $\funcY{u}$ are as in Definition~\ref{phim def} and $M_d(\ell,u)$ is as in Definition~\ref{LtqH def}.
\end{prop}

\begin{proof}
For this proof, write $Y(u)$ for $\funcY{u}$ and $H^{(2)}$ for $\funcHtwo{x}$.
Then, from Definition~\ref{phim def} and integration by parts,
\begin{flalign*}
&  \funcF{x} = \int_{H_2}^\infty Y(u) \,d \left( N(u,\chi) - N(H_2,\chi) \right) \\
 &  = \lim_{u\to\infty} \left( N(u,\chi) - N(H_2,\chi) \right)Y(u) - \left( N(H_2,\chi) - N(H_2,\chi) \right) Y(H_2) \\
  & \hskip12ex - \int_{H_2}^\infty \left( N(u,\chi) - N(H_2,\chi) \right) Y'(u)\,du \\
  & = \int_{H_2}^\infty \left( N(u,\chi) - N(H_2,\chi) \right) (-Y'(u))\,du,
 \end{flalign*}
where the limit equals $0$ because
$$
N(u,\chi) - N(H_2,\chi) < M_d(H_2,u) \ll u \log u,
$$
by Lemmas~\ref{N less than L} and~\ref{counting zeros simpler upper bounds lemma}, while $Y(u)< u^{-m-1}\leq u^{-2}$.
By the remarks in Definition~\ref{H1 and H2 def}, the $-Y'(u)$ factor is negative when $u<\funcHone{x}$ and positive when $u>\funcHone{x}$. Therefore, by Lemma~\ref{N less than L},
\begin{align*}
\funcF{x} &< \int_{H^{(2)}}^\infty \left( N(u,\chi) - N(H,\chi) \right) (-Y'(u))\,du \\
&< \int_{H^{(2)}}^\infty M_d(H_2,u) (-Y'(u))\,du. \\
\end{align*}
Via integration by parts, this last quantity is equal to
{\small{$$
-  \lim_{u\to\infty} M_d(H_2,u) Y(u) + M_d(H_2, H^{(2)}) Y(H^{(2)}) + \int_{H^{(2)}}^\infty \left(\frac{\partial}{\partial u} M_d(H_2,u) \right) Y(u) \,du.
$$}}
The limit here again equals $0$, yielding
\[
\funcF{x} \le M_d(H_2, H^{(2)}) Y(H^{(2)}) + \int_{H^{(2)}}^\infty \left(\frac{\partial}{\partial u} M_d(H_2,u) \right)Y(u)\,du.
\]
Since this last integrand is positive, we may extend the lower limit of integration from $H^{(2)}$ down to $H_2$ and still have a valid upper bound.
\end{proof}

The remainder of this section is devoted to finding an upper bound for the boundary term in equation~\eqref{first two terms}. Other than dealing with two cases depending on the size of $x$ relative to $H$, this optimization is simply a matter of calculus and notation.

\begin{definition} \label{boundary functions to maximize}
Let $d$ and $m$ be positive integers, and let $x,r,H,H_2$ and $R$ be real numbers satisfying $x>1$, $\frac14 < r \leq m+1$, and $x>1$. We define the functions
\begin{align*}
\funcBone{x}
  &= M_d(H_2,H_2) \cdot \funcY{H_2} (\log x)^r \\
  &= 2\left(C_1 \log(d H_2)+C_2\right) \cdot \frac{1}{H_2^{m+1}} \exp\bigg( {-}\frac{\log x}{R\log(dH_2)} \bigg) (\log x)^r, \\
  \end{align*}
$$
\funcBtwo{x}
  = \frac{d^m}{\pi} \bigg( \frac{\log^{r+1/2} x}{\sqrt{R(m+1)}}  \bigg) \exp\bigg( {-}\frac{2m+1}{\sqrt{R(m+1)}}\sqrt{\log x} \bigg)
  $$
  and
\begin{flalign*}
\funcB &= \big(\tfrac H{H_2} \big)^{m+1}  R^r (\log dH_2)^r  \\
 &\times  \max \bigg\{ M_d(H_2,H_2)  \bigg( \frac re \bigg)^r , \frac{ (m+1)^r \log^{r+1}(dH_2)}{\pi d^{m+1}H_2^{m}} \bigg\}.
  \end{flalign*}
\end{definition}

\begin{prop} \label{boundary term to maximize}
Let $d$ and $m$ be positive integers, and let $x$, $r$, $H$ and $H_2$ be real numbers satisfying $15\leq H \leq H_2$ and $\frac14 < r \leq m+1$.
If
$$
0 < \log x \le R(m+1) \log^2(d H_2),
$$
then
\begin{equation}\label{boundary small}
  M_d(H,\funcHtwo{x})  \funcY{ \funcHtwo{x} } (\log x)^r = \funcBone{x},
\end{equation}
while if $\log x > R(m+1) \log^2(d H_2)$, then
\begin{equation}\label{boundary large}
M_d(H,\funcHtwo{x}) \funcY{ \funcHtwo{x} } (\log x)^r < \funcBtwo{x}.
\end{equation}
\end{prop}

\begin{proof}
When $0<\log x \le R(m+1) \log^2(d H_2)$, we have $\funcHtwo{x} = H_2$ and so equation~\eqref{boundary small} follows.

On the other hand, when $\log x \geq  R(m+1) \log^2(dH_2)$, we have
$$
\funcHtwo{x} = \funcHone{x} \geq H_2 \geq 15,
$$
and so by Lemma~\ref{counting zeros simpler upper bounds lemma},
\begin{align*}
M_d(H,\funcHtwo{x}) &< \frac{\funcHone{x}}{\pi} \log \big( d\funcHone{x} \big) \\
&= \frac{1}{\pi d} \sqrt{\frac{\log x}{R(m+1)}} \cdot \exp\bigg( \sqrt{\frac{\log x}{R(m+1)}} \bigg)
\end{align*}
and
\begin{flalign*}
&\funcY{ \funcHtwo{x} } = \frac1{(\funcHone{x})^{m+1}} \cdot \exp\bigg( {-}\frac{\log x}{R\log(d\funcHone{x})} \bigg) \\
  &= d^{m+1} \exp\bigg( {-}\sqrt{\frac{(m+1)\log x}R} \bigg) \cdot \exp\bigg( {-}\sqrt{\frac{(m+1)\log x}R} \bigg)\\
  &= d^{m+1} \exp\bigg( {-}2\sqrt{\frac{(m+1)\log x}R} \bigg) .
\end{flalign*}
Therefore, as $2\sqrt{\frac{m+1}R} - \sqrt{\frac1{R(m+1)}} = \frac{2m+1}{\sqrt{R(m+1)}}$, we have
\begin{multline*}
M_d(H,\funcHtwo{x}) \funcY{ \funcHtwo{x} } (\log x)^r \\
  < \frac{d^m}{\pi} \bigg( \sqrt{\frac{\log x}{R(m+1)}} \bigg) \exp\bigg( {-}\frac{2m+1}{\sqrt{R(m+1)}}\sqrt{\log x} \bigg) (\log x)^r
  = \funcBtwo{x}.
\end{multline*}
\end{proof}

\begin{lemma} \label{simple calculus lemma}
Let $c_1$, $c_2$, $\lambda$, and $\mu$ be positive real numbers, and define
  \[ \funcPhi = c_1 \exp( -c_2\log^\lambda u ) \log^\mu u.\]
Then $\funcPhi$, as a function of $u$, is increasing for $1<u<u_0$ and decreasing for $u>u_0$, where
 $$u_0 = \exp\big( (\tfrac{\mu}{\lambda c_2})^{1/\lambda} \big).$$
In particular, $\funcPhi \le \funcPhi[u=u_0] = c_1 \big(\tfrac{\mu}{e\lambda c_2}\big)^{\mu/\lambda}$ for all $u\ge1$.
\end{lemma}

\begin{proof}
This is a straightforward calculus exercise.
\end{proof}

\begin{lemma} \label{boundary term maximized}
Let $d$ and $m$ be positive integers, and let $u$, $\mu$, $H$, $H_2$, and $R$ be positive real numbers satisfying $u>1$, $\mu\le m+1$, and $15\leq H \leq H_2$.
Then with $B^{(1)}$, $B^{(2)}$, and $B$ as in Definition~\ref{boundary functions to maximize}, we have the following inequalities:
\begin{enumerate}[label=(\roman*)] 
\item\label{item1} $H^{m+1}  \funcBone[r=\mu]{u} \leq \funcB[r=\mu]$;
\item\label{item2} If
$\log u \ge R(m+1) \log^2(dH_2)$,
then
$$
H^{m+1} \funcBtwo[r=\mu]{u} \leq \funcB[r=\mu].
$$
\end{enumerate}
\end{lemma}

\begin{proof}
Using the notation and final conclusion of Lemma~\ref{simple calculus lemma}, we find that
\begin{align*}
  H^{m+1}\funcBone[r=\mu]{u}
    &= \funcPhi[c1= H^{m+1}\cdot \frac{M_d(H_2,H_2)}{H_2^{m+1}}, c2 = \frac{1}{R \log (dH_2)},lambda=1] \\
    &\leq  H^{m+1}\cdot \frac{M_d(H_2,H_2)}{H_2^{m+1}} \bigg( \frac{\mu R\log(dH_2)}e \bigg)^\mu \\
    &= \big( \tfrac H{H_2} \big)^{m+1} R^\mu (\log dH_2)^\mu \cdot M_d(H_2,H_2) \big(\tfrac \mu e\big)^\mu \\
    &\leq \funcB[r=\mu],
\end{align*}
which establishes claim~\ref{item1}.

Next, observe that
\begin{equation} \label{Phi will be decreasing}
H^{m+1} \funcBtwo[r=\mu]{u} = \funcPhi[c1=\frac{H^{m+1}d^m}{\pi\sqrt{R(m+1)}}, c2=\frac{2m+1}{\sqrt{R(m+1)}},lambda=\frac12,mu=\mu+\frac12],
\end{equation}
which by Lemma~\ref{simple calculus lemma} is decreasing for
  \[ u > \exp\bigg( \bigg( \frac{\mu+1/2}{\frac 12 \cdot \frac{2m+1}{\sqrt{R(m+1)}}} \bigg)^{1/(1/2)} \bigg)
                     = \exp \bigg( R(m+1) \bigg( \frac{2\mu+1}{2m+1} \bigg)^2 \bigg) .\]
As $\log(dH_2) \ge \log15 > \frac53 \ge \frac{2\mu+1}{2m+1}$ under the hypotheses of this lemma, we know by the hypothesis of claim~\ref{item2} that $\log u > R(m+1) \big( \frac{2\mu+1}{2m+1} \big)^2$. It follows that the right-hand side of equation~\eqref{Phi will be decreasing} is indeed decreasing. Therefore,
  \begin{align*}
  H^{m+1}&\funcBtwo[r=\mu]{u} \\
  &\leq \funcPhi[c1=\frac{H^{m+1}d^m}{\pi\sqrt{R(m+1)}},  c2=\frac{2m+1}{\sqrt{R(m+1)}}, lambda=\frac12, mu=\mu+\frac12, u=\exp\left(R(m+1) \log^2(dH_2)\right)] \\
  &=\frac{H^{m+1}}{\pi d^{m+1}H_2^{2m+1}} R^\mu(m+1)^\mu \log^{2\mu+1}(dH_2) \\
  &= \big( \tfrac H{H_2} \big)^{m+1} R^\mu (\log dH_2)^\mu \cdot \frac{(m+1)^\mu \log^{\mu+1}(dH_2)}{\pi d^{m+1} H_2^m}\\
  &\le \funcB[r=\mu],
  \end{align*}
as claimed.
\end{proof}

We have thus bounded the first term on the right-hand side of equation~\eqref{first two terms}; it remains to treat the second term
\begin{equation}  \label{second term going Bessel}
 \int_{H_2}^\infty \left(\frac{\partial}{\partial u} M_d(H_2,u)\right) \funcY{u} \,du,
\end{equation}
which is the subject of Section~\ref{Sec4}.

\section{Optimization of the upper bound for $|\psi(x;q,a) - x/\phi(q)|$, for $q\le 10^5$} \label{Sec4}

\subsection{Estimation of integrals using incomplete modified Bessel functions}

We follow the strategy of previous work on explicit error bounds for prime counting functions, going back to Rosser and Schoenfeld~\cite{RS2}, of bounding integrals with the form given in equation~\eqref{second term going Bessel}. After some well-chosen changes of variables, we use two Taylor approximations of algebraic functions to construct a bounding integral whose antiderivative we can write down explicitly.

\begin{definition} \label{Inuk def}
Given positive real numbers $n,m,\alpha,\beta,\ell$, define an incomplete modified Bessel function of the first kind as
\[
I_{n,m}(\alpha,\beta;\ell) = \int_\ell^\infty \frac{(\log\beta u)^{n-1}}{u^{m+1}} \exp\bigg( {-}\frac\alpha{\log \beta u} \bigg) \,du.
\]
\end{definition}

\begin{prop} \label{convert to I notation prop}
Let $d$ and $m$ be positive integers, and let $x,H_2,R$ be positive real numbers.
Then
\begin{flalign*}
&\int_{H_2}^\infty \left(\frac{\partial}{\partial u} M_d(H_2,u) \right) \funcY{u} \,du \\
&\leq \frac1\pi I_{2,m}\bigg( \frac{\log x}R,q;H_2 \bigg)
+ \bigg( \frac1\pi \log\frac{1}{2\pi} + \frac{C_1}{H_2} \bigg) I_{1,m}\bigg( \frac{\log x}R,q;H_2 \bigg). \\
\end{flalign*}
\end{prop}

\begin{proof}
For this proof, write $Y(u)$ for $\funcY{u}$. If we put $\alpha=(\log x)/R$ and $\beta=d$, we see from Definition~\ref{phim def} that
  \[ Y(u) = \frac1{u^{m+1}} \exp\bigg( {-}\frac\alpha{\log\beta u} \bigg). \]
Using equation~\eqref{ninja}, and writing $\log \tfrac{du}{2\pi}  = \log\beta u + \log\tfrac{1}{2\pi}$,
\begin{flalign*}
& \int_{H_2}^\infty \left(\frac{\partial}{\partial u} M_d(H_2,u)\right) Y(u)\,du \\
 & = \int_{H_2}^\infty \bigg( \frac1\pi\log\frac{du}{2\pi} + \frac{C_1}u \bigg) \frac1{u^{m+1}} \exp\bigg( {-}\frac\alpha{\log\beta u} \bigg)\,du \\
 &  \leq \frac1\pi \int_{H_2}^\infty \frac{\log\beta u}{u^{m+1}} \exp\bigg( {-}\frac\alpha{\log\beta u} \bigg) \,du \\
  & + \bigg( \frac1\pi \log\frac{1}{2\pi} + \frac{C_1}{H_2} \bigg) \int_{H_2}^\infty \frac1{u^{m+1}} \exp\bigg( {-}\frac\alpha{\log\beta u} \bigg) \,du, \\
\end{flalign*}
since $u \geq H_2$, as required.
\end{proof}

\begin{definition} \label{Knu def}
Given positive constants $n$, $z$, and $y$, define the incomplete modified Bessel function
of the second kind (see for example \cite[page 376, equation 9.6.24]{AS})
  \[ K_n(z;y) = \frac12 \int_y^\infty u^{n-1} \exp\bigg( {-}\frac z2\bigg( u+\frac1u \bigg) \bigg) \,du. \]
\end{definition}

\begin{lemma}  \label{I to K cov lemma}
Given positive constants $n$, $m$, $\alpha$, $\beta$, and $\ell$,
\begin{equation*}
I_{n,m}(\alpha,\beta;\ell) = 2\beta^m \bigg( \frac\alpha m\bigg)^{n/2} K_n\bigg( 2\sqrt{\alpha m}; \sqrt{\frac m\alpha} \log(\beta \ell) \bigg).
\end{equation*}
In particular, if $n$, $m$, $x$, $R$, $d$, and $H_2$ are positive real numbers with $x>1$, then
\begin{equation*}
I_{n,m}\bigg( \frac{\log x}R,d;H_2 \bigg) = 2d^m \bigg( \frac{\log x}{mR} \bigg)^{n/2} K_n\bigg( 2\sqrt{\frac{m\log x}R}; \sqrt{\frac{mR}{\log x}} \log(dH_2) \bigg).
\end{equation*}
\end{lemma}

\begin{proof}
The first identity follows easily from the change of variables $u = \sqrt{\frac m\alpha} \log\beta t$ in Definition~\ref{Inuk def} of $I_{n,m}(\alpha,\beta;\ell)$; the second identity is immediate upon substitution.
\end{proof}

\begin{definition}  \label{erfc def}
For any real number $u$, define the complementary error function
$$
\displaystyle\erfc(u) = \frac2{\sqrt\pi} \int_u^\infty e^{-t^2}\,dt.
$$
\end{definition}

\begin{definition}  \label{K pieces def}
For positive real numbers $y$ and $z$, define
\begin{align*}
J_{1a}(z;y) &= \frac{3\sqrt{y}+8}{16 ze^{z(y+1/y)/2}}, \\
J_{1b}(z;y) &= \sqrt\pi \erfc\bigg( \sqrt{\frac z2} \bigg( \sqrt{y}-\frac1{\sqrt {y}} \bigg) \bigg) \frac{8z+3}{16\sqrt2\, z^{3/2}e^z}, \\
J_{2a}(z;y) &= \frac{(35 y^{3/2}+128 y+135y^{1/2}+128y^{-1}) z+105y^{1/2}+256}{256 z^2e^{z(y+1/y)/2}}, \\
J_{2b}(z;y) &= \sqrt\pi \erfc\bigg( \sqrt{\frac z2} \bigg( \sqrt{y}-\frac1{\sqrt {y}} \bigg) \bigg) \frac{128 z^2+240 z+105}{256\sqrt2\, z^{5/2}e^z}.
\end{align*}
\end{definition}

The next proposition is essentially \cite[equations (2.30) and (2.31)]{RS2}.

\begin{prop} \label{bounds for incomplete Bessels prop}
For $z,y>0$, we have $K_1(z;y) \le J_{1a}(z;y) + J_{1b}(z;y)$ and $K_2(z;y) \le J_{2a}(z;y) + J_{2b}(z;y)$.
\end{prop}

\begin{proof}
In Definition~\ref{Knu def}, make the change of variables
  \[u = 1+w^2+w\sqrt{w^2+2}, \quad du = 2 \bigg( w+\frac{w^2+1}{\sqrt{w^2+2}}\bigg) \,dw,\]
so that
$w = \frac1{\sqrt2}(\sqrt u-\frac1{\sqrt u})$ and hence $w^2 = \frac12(u+\frac1u)-1$. We obtain
\[
K_n(z;y) = e^{-z} \int_{v}^\infty (1+w^2+w\sqrt{w^2+2})^{n-1}\bigg( w+ \frac{w^2+1}{\sqrt{w^2+2}} \bigg) e^{-zw^2} \,dw,
\]
where $v = \frac1{\sqrt2}(\sqrt{y}-\frac1{\sqrt {y}})$. In particular,
\begin{align*}
K_1(z;y) &= e^{-z} \int_{v}^\infty \bigg( w+ \frac{w^2+1}{\sqrt{w^2+2}} \bigg) e^{-zw^2} \,dw \\ 
K_2(z;y) &= e^{-z} \int_{v}^\infty \bigg( 2w^3 + 2w + \frac{2w^4+4w^2+1}{\sqrt{w^2+2}} \bigg) e^{-zw^2} \,dw. 
\end{align*}
The inequalities
\begin{align*}
\frac{w^2+1}{\sqrt{w^2+2}} &\le \frac{3 w^2}{4\sqrt2}+\frac{1}{\sqrt2} \\ 
\frac{2w^4+4w^2+1}{\sqrt{w^2+2}} &\le \frac{35 w^4}{32 \sqrt2}+\frac{15 w^2}{4 \sqrt2}+\frac{1}{\sqrt2}, 
\end{align*}
(which are identical to \cite[equations (2.27) and (2.28)]{RS2}) can be verified by squaring both sides; consequently,
\begin{align*}
K_1(z;y) &\le e^{-z} \int_{v}^\infty \bigg( w+ \frac{3 w^2}{4\sqrt2}+\frac{1}{\sqrt2} \bigg) e^{-zw^2} \,dw \\ 
K_2(z;y) &\le e^{-z} \int_{v}^\infty \bigg( 2w^3 + 2w + \frac{35 w^4}{32 \sqrt2}+\frac{15 w^2}{4 \sqrt2}+\frac{1}{\sqrt2} \bigg) e^{-zw^2} \,dw. 
\end{align*}
Routine integration of the right-hand sides now gives
$$
K_1(z;y) \le e^{-z} \bigg( \frac{3 \sqrt{2}\, v+8 }{16 ze^{v^2 z}} + \sqrt\pi \erfc(v \sqrt{z}) \frac{8z+3}{16\sqrt2\, z^{3/2}} \bigg)
$$
and, similarly, $e^z K_2(z;y)$ is bounded above by
\begin{flalign*}
&  \frac{70 \sqrt{2}\, v^3 z+256 v^2 z+15 \sqrt{2}\, v (16 z+7)+256(z+1)}{256 z^2e^{v^2 z} } \\
&+ \sqrt\pi \erfc(v \sqrt{z}) \frac{128 z^2+240 z+105}{256\sqrt2\, z^{5/2}}. \\
\end{flalign*}
Substituting in $v = \frac1{\sqrt2}(\sqrt{y}-\frac1{\sqrt {y}})$, so that $v^2+1=(y+1/y)/2$, yields
$$
K_1(z;y) \le \frac{3y+8\sqrt{y}-3}{16 ze^{z(y+1/y)/2}\sqrt y} + \sqrt\pi \erfc\bigg( \sqrt{\frac z2} \bigg( \sqrt{y}-\frac1{\sqrt {y}} \bigg) \bigg) \frac{8z+3}{16\sqrt2\, z^{3/2}e^z},
$$
while $K_2(z;y)$ its bounded above by
\begin{flalign*}
& \frac{(35 y^3+128 y^{5/2}+135y^2-135y+128\sqrt y-35) z+105y^2+256y^{3/2}-105y}{256 z^2e^{z(y+1/y)/2}y^{3/2}}  \\
& + \sqrt\pi \erfc\bigg( \sqrt{\frac z2} \bigg( \sqrt{y}-\frac1{\sqrt {y}} \bigg) \bigg) \frac{128 z^2+240 z+105}{256\sqrt2\, z^{5/2}e^z}.
\end{flalign*}
The lemma now follows upon simply omitting the negative terms from the numerators in these upper bounds (and comparing with Definition~\ref{K pieces def}).
\end{proof}

\subsection{Elementary estimation of the complementary error function $\erfc(u)$}  \label{erfy section}

Some of the bounding functions in the previous section contain factors of the complementary error function $\erfc(u)$ evaluated at complicated arguments involving fractional powers of $\log x$. In this section, we establish simpler and reasonably tight upper bounds for factors of this type. Our first task, which culminates in Lemma~\ref{erfc decreasing lemma}, is to provide a general structure for the type of argument we will need. (We caution the reader that the temporary parameters $y$ and $z$ do not fill the same role that they did in the previous section.) Then in the rest of the section, leading up to Proposition~\ref{good riddance erfc}, we implement that argument with some specific numerical choices motivated by our ultimate invocation of the proposition.

\begin{lemma} \label{deal with decreasing}
Let $v$, $w$, $y$, $z$, $\mu$, and $\tau$ be positive constants with $v>\tau$ and $yz>w$.
Let $f(u)$ be a positive, differentiable function, and define
$$
g(u) = f\left(v-\frac uy \right) u^{2\mu} e^{-zu}.
$$
Suppose that
\begin{equation}  \label{log-d upper bound assumption}
{-}\frac{f'(u)}{f(u)} \le w \quad\text{for } u\le \tau.
\end{equation}
Then $g(u)$ is a decreasing function of $u$ for
\[
u \ge \max\bigg\{ y(v-\tau), \frac {2\mu}{z-w/y} \bigg\}.
\]
\end{lemma}

\begin{proof}
It suffices to show that $\log g(u)$ is decreasing. We have
\begin{flalign*}
\frac d{du} (\log g(u)) & = \frac d{du} \bigg( \log f\bigg(v-\frac uy\bigg) + {2\mu}\log u - zu \bigg) \\
& = -\frac{f'(v-u/y)}{yf(v-u/y)} + \frac {2\mu}u - z. \\
\end{flalign*}
Since $u\ge y(v-\tau)$, we have $v-u/y \le \tau$, and so by the assumption~\eqref{log-d upper bound assumption},
\[
\frac d{du} (\log g(u)) \le \frac wy + \frac {2\mu}u - z \le 0
\]
since $u \ge {2\mu}/(z-\frac wy)$.
\end{proof}

\begin{lemma}  \label{erfc log-d bound}
Given $\tau\ge0$, if we have $u\le \tau$, then
$$
-\frac{\erfc'(u)}{\erfc(u)} \le \tau+\sqrt{\tau^2+2}.
$$
\end{lemma}

\begin{proof}
Note that
\begin{equation}  \label{erfc log-d expression}
{-}\frac{\erfc'(u)}{\erfc(u)} = \frac2{\sqrt\pi} \frac1{e^{u^2}\erfc(u)}.
\end{equation}
When $u\le0$, since $\erfc(u) \ge 1$ we have
$$-\frac{\erfc'(u)}{\erfc(u)} \le \frac2{\sqrt\pi} < \sqrt2 \le
\tau+\sqrt{\tau^2+2}
$$
for all $\tau\ge0$. On the other hand, when $u\ge0$, we have \cite[equation 7.8.2]{NIST:DLMF}
\begin{equation}  \label{AS erfc bounds}
\frac1{u+\sqrt{u^2+2}} < e^{u^2} \frac{\sqrt\pi}2 \erfc(u) \le \frac1{u+\sqrt{u^2+4/\pi}}.
\end{equation}
In light of the identity~\eqref{erfc log-d expression}, the first inequality is equivalent to
$$
-\frac{\erfc'(u)}{\erfc(u)} \le u+\sqrt{u^2+2},
$$
which establishes the lemma as this function is increasing in~$u$.
\end{proof}

\begin{definition} \label{e def}
Given an integer $m\ge2$ and positive constants $\lambda$, $\mu$, and $R$, define for $x>1$ the function
\begin{equation*}  \label{h related to erfc def}
\Xi_{m,\lambda,\mu,R}(x) = \sqrt\pi \erfc \bigg( \sqrt{m\lambda} - \sqrt{\frac{\log x}{R\lambda}} \bigg) \exp\bigg( {-}2\sqrt{\frac{m\log x}R} \bigg) \log^{\mu} x,
\end{equation*}
where $\erfc$ is as given in Definition \ref{erfc def}.
\end{definition}

\begin{lemma}  \label{erfc decreasing lemma}
Let $m$, $\lambda$, $\mu$ and $R$ be positive constants. Choose $\tau\ge0$ and set $w=\tau+\sqrt{\tau^2+2}$. Suppose that $m\lambda>w^2/4$ and
\[
\sqrt{R\lambda} (\sqrt{m\lambda} - \tau) \ge \frac{2\mu}{2\sqrt{m/R} - w/\sqrt{R\lambda}},
\]
or equivalently that
\[
\mu \le (\sqrt{m\lambda} - w/2) (\sqrt{m\lambda} - \tau).
\]
Then the function $\Xi_{m,\lambda,\mu,R}(x)$ in Definition~\ref{e def} is a decreasing function of $x$ for $x \ge \exp \big( R\lambda (\sqrt{m\lambda} - \tau)^2 \big)$.
\end{lemma}

\begin{proof}
In Lemma~\ref{deal with decreasing} we let $f(u) = \sqrt\pi \erfc(u)$, and we set $v=\sqrt{m\lambda}$, $y=\sqrt{R\lambda}$, and $z=2\sqrt{m/R}$, so that $-\frac{f'(u)}{f(u)} \le w$ for $u\le \tau$ by Lemma~\ref{erfc log-d bound}. As $m\lambda>\tau^2$, we have $v>\tau$ and $yz>w$. By Lemma~\ref{erfc log-d bound}, condition~\eqref{log-d upper bound assumption} is satisifed. Then $g(\sqrt{\log x}) = \Xi_{m,\lambda,\mu,R}(x)$, and Lemma~\ref{deal with decreasing} guarantees that $\Xi_{m,\lambda,\mu,R}(x)$ is decreasing provided that
  \[\sqrt{\log x} \geq \max \bigg\{ \sqrt{R\lambda} (\sqrt{m\lambda} - \tau), \frac{2\mu}{2\sqrt{m/R} - w/\sqrt{R\lambda}} \bigg\}= \sqrt{R\lambda} (\sqrt{m\lambda} - \tau),\]
where the last equality is a hypothesis of this lemma.
\end{proof}

We now choose some specific values of the parameters that correspond to the range of exponents $\mu$, depending on $m$, for which we want to apply the previous lemma.

\begin{definition} \label{tau and omega def}
For integers $m\ge2$, define real numbers $\tau_m$ according to the following table:
$$
\begin{array}{|c||c|c|c|c|c|c|} \hline
m & 2 & 3 & 4 & 5 & 6 & 7 \\
\tau_m & 4.0726 & 5.2067 & 6.1454 & 6.9631 & 7.6967 & 8.3675 \\ \hline
m & 8 & 9 & 10 & 11 & 12 & \geq 13 \\
\tau_m & 8.9891 & 9.5709 & 10.1197 & 10.6405 & 11.1371 & 11.6126 \\ \hline
\end{array}
$$
Then, for any $m\ge2$, define
$\displaystyle
\omega_m = \frac2{\tau_m + \sqrt{\tau_m^2+4/\pi}}
$.
\end{definition}

\begin{lemma} \label{erfc better bound}
For a given $m\ge2$:
\begin{enumerate}
\item $m+\frac74 \le (\sqrt{m\lambda}-\tau_m)\big(\sqrt{m\lambda}-(\tau_m+\sqrt{\tau_m^2+2})/2 \big)$ holds for all
$\lambda\ge \log(10^8)$;
\item $\sqrt\pi \erfc(u) \le \omega_m e^{-u^2}$ when $u\ge \tau_m$.
\end{enumerate}
\end{lemma}

\begin{proof}
For part (a), since the right-hand side of the inequality is a convex function of $\lambda$, it suffices to check that for any given $m$, the right-hand side minus the left-hand side is positive and increasing at $\lambda=\log(10^8)$. Part (b) then follows from the upper bound in equation~\eqref{AS erfc bounds} in the form
\begin{equation*}
\sqrt\pi \erfc(u) \le e^{-u^2} \frac2{u+\sqrt{u^2+4/\pi}} \le e^{-u^2} \frac2{\tau_m+\sqrt{\tau_m^2+4/\pi}}= e^{-u^2} \omega_m.
\end{equation*}
\end{proof}

\begin{prop}  \label{good riddance erfc}
Let $m\ge2$ be given, let $\mu \le m+\frac74$ and $\lambda\ge \log(10^8)$, let $R$ be positive, and let $\Xi_{m,\lambda,\mu,R}(x)$ be as in Definition~\ref{e def}. Then:
\begin{enumerate}
\item $\Xi_{m,\lambda,\mu,R}(x)$ is a decreasing function of $x$ for $x \ge \exp \big( R\lambda (\sqrt{m\lambda} - \tau_m)^2 \big)$.
\item For $1 \le x \le \exp \big( R\lambda (\sqrt{m\lambda} - \tau_m)^2 \big)$, we have
$$
\Xi_{m,\lambda,\mu,R}(x) \le \omega_m e^{-m\lambda} \exp \bigg( {-} \frac{\log x}{R\lambda} \bigg) \log^{\mu} x.
$$
\end{enumerate}
\end{prop}

\begin{proof}
By Lemma~\ref{erfc better bound}(a), the hypotheses of Lemma~\ref{erfc decreasing lemma} are satisfied with $\tau=\tau_m$, which immediately establishes the proposition's first claim. We apply Lemma~\ref{erfc better bound}(b) with $u=\sqrt{m\lambda} - \sqrt{(\log x)/R\lambda}$, which is at least $\tau_m$ when
$$
x \le \exp \big( R\lambda (\sqrt{m\lambda} - \tau_m)^2 \big);
$$
the result is
\begin{align*}
\Xi_{m,\lambda,\mu,R}(x) \le \omega_m \exp \bigg( {-} \bigg( \sqrt{m\lambda} - \sqrt{\frac{\log x}{R\lambda}} \bigg)^2 \bigg) \exp\bigg( {-}2\sqrt{\frac{m\log x}R} \bigg) \log^{\mu} x \\
= \omega_m \exp \bigg( {-} m\lambda + 2\sqrt{\frac{m\log x}R} -\frac{\log x}{R\lambda} \bigg) \exp\bigg( {-}2\sqrt{\frac{m\log x}R} \bigg) \log^{\mu} x,
\end{align*}
which establishes the second claim.
\end{proof}

\subsection{Identification of maximum values of bounding functions via calculus}

As we move towards our upper bound for $|\psi(x;q,a) - x/\phi(q)|$, we will need to find the maximum values of various decreasing functions (of the type addressed in the previous two sections) multiplied by powers of $\log x$. Each individual such product can be bounded by elementary calculus that is straightforward---especially given our existing bounds on functions related to $\erfc(x)$ from Section~\ref{erfy section}---but notationally extremely unwieldy. We therefore encourage the reader to regard this section only as a necessary evil.

We can, however, make one possibly insightful remark before getting underway. The upper bound currently being derived for $\big| \psi(x;q,a) - x/\phi(q) \big| / (x/\log x)$ has several pieces, some of which we have already seen decay like a power of~$x$. The remaining pieces of the upper bound will be bounded by the functions in Definition~\ref{J def} below; and the sharp-eyed reader will notice that these functions too decay like $\exp( -\frac{\log x}{R\lambda} )$, which is to say, like a power of~$x$. (Of course, the functions do start off increasing for small values of $x$, so that there is a maximum value which we seek to identify.) This rate of decay seems too good to be true, since it would correspond to a zero-free strip of constant width (that is, a quasi-GRH). This apparent paradox can be resolved by noting that the functions in Definition~\ref{J def} are involved in the upper bound for the function $\funcU{x}$ (see Definition~\ref{U def}), which is a sum over only the zeros of the $L(s,\chi)$ with large imaginary part. It seems that such a function actually does decay like a power of $x$ initially, before slowing down to decay only like $\exp(-c\sqrt{\log x})$ as is consistent with the classical zero-free region; but, as it happens, the maxima of these functions occur for moderately sized $x$, for which the functions' envelopes are still decaying like a power of~$x$. (One can contrast this observation with Lemma~\ref{two-prong lemma}, in which we see (for large moduli~$q$) the expected rate of decay in the error term.)

\begin{definition}  \label{J def}
Given an integer $m\ge2$ and positive constants $r$, $x$, $\lambda$, $H_2$, and $R$, define
\begin{flalign*}
& P_{1a}(x;m,r,\lambda,H_2,R) = \frac1{H_2^m} \bigg( \frac{3R^{1/4}\lambda^{1/2}\log^{r-1/4}x}{16m^{3/4}} + \frac{(\log x)^r}{2m} \bigg) \exp\bigg( {-}\frac{\log x}{R\lambda} \bigg), \\
& P_{1b}(x;m,r,\lambda,H_2,R) = \frac{\omega_m}{H_2^m} \bigg( \frac{\log^{r+1/4} x}{2m^{3/4}R^{1/4}} + \frac{3R^{1/4}\log^{r-1/4}x}{32m^{5/4}} \bigg) \exp\bigg( {-}\frac{\log x}{R\lambda} \bigg), \\
&P_1(x;m,r,\lambda,H_2,R) = P_{1a}(x;m,r,\lambda,H_2,R) + P_{1b}(x;m,r,\lambda,H_2,R); \\
& P_{2a}(x;m,r,\lambda,H_2,R) = \frac1{H_2^m} \exp\bigg( {-}\frac{\log x}{R\lambda} \bigg) \bigg( \frac{\log^{r+1}x}{2 \lambda  m^2 R} + \frac{135\lambda^{1/2} \log^{r+1/4}x}{256 m^{5/4} R^{1/4}} \notag \\
&\hskip12ex + \frac{(m\lambda+1)(\log x)^r}{2 m^2} + \frac{35(2m \lambda+3) \lambda^{1/2}R^{1/4}\log^{r-1/4}x}{512 m^{7/4}} \bigg), \\
& P_{2b}(x;m,r,\lambda,H_2,R) = \frac{\omega_m}{H_2^m} \bigg( \frac{\log^{r+3/4}x}{2m^{5/4}R^{3/4}} + \frac{15\log^{r+1/4}x}{32m^{7/4}R^{1/4}} \\
 & \hskip12ex+ \frac{105R^{1/4}\log^{r-1/4}x}{1024m^{9/4}} \bigg) \exp\bigg( {-}\frac{\log x}{R\lambda} \bigg), \\
& P_2(x;m,r,\lambda,H_2,R) = P_{2a}(x;m,r,\lambda,H_2,R) + P_{2b}(x;m,r,\lambda,H_2,R).
\end{flalign*}
\end{definition}

\begin{definition}  \label{M def}
Given an integer $m\ge2$ and positive constants $r$, $\lambda$, $H_2$, and $R$, define
\begin{flalign*}
& Q_{1a}(m,r,\lambda,H_2,R) = \frac{R^r}{e^r H_2^m} \bigg( \frac{3e^{1/4}(r-1/4)^{r-1/4} \lambda^{r+1/4}}{16m^{3/4}} + \frac{r^r \lambda^r}{2m} \bigg), \\
& Q_{1b}(m,r,\lambda,H_2,R) = \frac{\omega_mR^r}{e^r H_2^m} \bigg( \frac{(r+1/4)^{r+1/4}\lambda^{r+1/4}}{2e^{1/4}m^{3/4}}  + \frac{3e^{1/4}(r-1/4)^{r-1/4}\lambda^{r-1/4}}{32m^{5/4}} \bigg), \\
& Q_1(m,r,\lambda,H_2,R) = Q_{1a}(m,r,\lambda,H_2,R) + Q_{1b}(m,r,\lambda,H_2,R); \\
& Q_{2a}(m,r,\lambda,H_2,R) = \frac{R^r}{e^r H_2^m} \bigg( \frac{(r+1)^{r+1}\lambda^r}{2em^2} + \frac{135(r+1/4)^{r+1/4} \lambda^{r+3/4}}{256e^{1/4} m^{5/4}} \\
&\qquad{}+ \frac{(m\lambda+1)r^r \lambda^r}{2m^2} + \frac{35e^{1/4}(2m \lambda+3)(r-1/4)^{r-1/4} \lambda^{r+1/4}}{512m^{7/4}} \bigg), \\
& Q_{2b}(m,r,\lambda,H_2,R) = \frac{\omega_mR^r}{e^r H_2^m} \bigg( \frac{(r+3/4)^{r+3/4}\lambda^{r+3/4}}{2e^{3/4}m^{5/4}} + \frac{15(r+1/4)^{r+1/4}\lambda^{r+1/4}}{32e^{1/4}m^{7/4}} \\
&\qquad{}+ \frac{105e^{1/4}(r-1/4)^{r-1/4}\lambda^{r-1/4}}{1024m^{9/4}} \bigg), \\
& Q_2(m,r,\lambda,H_2,R) = Q_{2a}(m,r,\lambda,H_2,R) + Q_{2b}(m,r,\lambda,H_2,R).
\end{flalign*}
\end{definition}

\begin{definition}  \label{z and y def}
Let $d$ and $m$ be positive integers with $m\ge 2$, and let $H_2,R$ and $x$ be positive real numbers with $x>1$. Define
\[
z_{m,R}(x) = 2\sqrt{\frac{m\log x}R} \quad\text{and}\quad  y_{d,m,R}(x;H_2) = \sqrt{\frac{mR}{\log x}} \log(dH_2).
\]
\end{definition}

\begin{lemma}  \label{algebra with z and y}
Let $m$, $R$, $x$, $d$, and $H_2$ be positive real numbers with $x>1$.
Then
\begin{flalign}
& \exp\bigg( {-}\frac{z_{m,R}(x)}2 \bigg( y_{d,m,R}(x;H_2)+\frac1{y_{d,m,R}(x;H_2)} \bigg) \bigg) \notag \\
&= \bigg( \frac{1}{dH_2} \bigg)^{m} \exp\bigg( {-}\frac{\log x}{R\log(dH_2)} \bigg) \label{ezy denominator identity}
\end{flalign}
and
\begin{flalign}
&\sqrt{\frac{z_{m,R}(x)}2} \bigg( \sqrt{y_{d,m,R}(x;H_2)}-\frac1{\sqrt {y_{d,m,R}(x;H_2)}} \bigg) \notag \\
& = \sqrt{m\log(dH_2)} - \sqrt{\frac{\log x}{R\log(dH_2)}}.  \label{erf argument identity}
\end{flalign}
\end{lemma}

\begin{proof}
Both identities follow quickly from $e^{-m\log(dH_2)} = (dH_2)^{-m}$ and the evaluations
$$
\frac{z_{m,R}(x)}2 \cdot y_{d,m,R}(x;H_2) = m \log(dH_2)
$$
and
$$
 \frac {z_{m,R}(x)}2\cdot\frac1{y_{d,m,R}(x;H_2)} = \frac{\log x}{R\log(dH_2)}.
$$
\end{proof}

\begin{lemma}  \label{1ax}
Let $r$, $m$, $R$, $x$, $d$, and $H_2$ be positive real numbers with $x>1$. Then
\begin{flalign*}
& (\log x)^r \cdot 2d^m \bigg( \frac{\log x}{mR} \bigg)^{1/2} J_{1a}\big( z_{m,R}(x);y_{d,m,R}(x;H_2) \big) \\
& \hskip16ex = P_{1a}\big(x;m,r,\log(dH_2),H_2,R\big). \\
\end{flalign*}
\end{lemma}

\begin{proof}
In this proof, we write $y$ for $y_{d,m,R}(x;H_2)$ and $z$ for $z_{m,R}(x)$. Using Definition~\ref{K pieces def} and the identity~\eqref{ezy denominator identity}:
\begin{flalign*}
&(\log x)^r  \cdot 2d^m \bigg( \frac{\log x}{mR} \bigg)^{1/2} J_{1a} (z; y) \\
&= (\log x)^r \cdot 2d^m \bigg( \frac{\log x}{mR} \bigg)^{1/2} \frac{3y+8\sqrt{y}}{16 ze^{z(y+1/y)/2}\sqrt y} \\
&= (\log x)^r \cdot 2d^m \bigg( \frac{\log x}{mR} \bigg)^{1/2} \frac{3y+8\sqrt{y}}{16 z\sqrt y} \bigg( \frac{1}{dH_2} \bigg)^{m} \exp\bigg( {-}\frac{\log x}{R\log(dH_2)} \bigg) \\
&= \frac{\log^{r+1/2} x}{H_2^m} \frac1{8\sqrt{mR}} (3\sqrt y+8) z^{-1} \exp\bigg( {-}\frac{\log x}{R\log(dH_2)} \bigg) \\
&= \frac{\log^{r+1/2} x}{H_2^m} \frac1{8\sqrt{mR}} \bigg( \frac{3m^{1/4}R^{1/4}\sqrt{\log(dH_2)}}{\log^{1/4}x} + 8 \bigg) \\
& \; \; \; \; \times \frac{\sqrt R}{2\sqrt {m\log x}} \exp\bigg( {-}\frac{\log x}{R\log(dH_2)} \bigg) \\
&= \frac1{H_2^m} \bigg( \frac{3R^{1/4}\sqrt{\log(dH_2)}\log^{r-1/4}x}{16m^{3/4}} + \frac{(\log x)^r}{2m} \bigg) \exp\bigg( {-}\frac{\log x}{R\log(dH_2)} \bigg),
\end{flalign*}
which establishes the lemma thanks to Definition~\ref{J def}.
\end{proof}

\begin{lemma}  \label{1a}
Let $r$, $m$, $R$, $x$, $\lambda$, and $H_2$ be positive real numbers with $x>1$ and $r>\frac14$. Then
$$
P_{1a}(x;m,r,\lambda,H_2,R) \le Q_{1a}(m,r,\lambda,H_2,R).
$$
\end{lemma}

\begin{proof}
By Lemma~\ref{simple calculus lemma}, the two summands in Definition~\ref{J def} for $P_{1a}$ are maximized at $\log x = (r-\frac14)R\lambda$ and $\log x = r R\lambda$, respectively. Inserting these respective values of $x$ into the two summands yields the upper bound
\begin{align*}
P_{1a}(x;m,r,\lambda,H_2,R) &\le \frac1{H_2^m} \bigg( \frac{3R^{1/4}\sqrt{\lambda}}{16m^{3/4}} \bigg( \frac{(r-1/4)R\lambda}e \bigg)^{r-1/4} + \frac1{2m} \bigg( \frac{r R\lambda}e \bigg)^r \bigg) \\
&= \frac{R^r}{e^r H_2^m} \bigg( \frac{3e^{1/4}(r-1/4)^{r-1/4} \lambda^{r+1/4}}{16m^{3/4}} + \frac{r^r \lambda^r}{2m} \bigg),
\end{align*}
which establishes the lemma thanks to Definition~\ref{M def}.
\end{proof}

\begin{lemma}  \label{2ax}
Let $r$, $m$, $R$, $x$, $d$, and $H_2$ be positive real numbers with $x>1$. Then
\begin{equation*}
(\log x)^r \cdot 2d^m \frac{\log x}{mR} J_{2a}\big( z_{m,R}(x);y_{d,m,R}(x;H_2) \big) = P_{2a}\big(x;m,r,\log(dH_2),H_2,R\big).
\end{equation*}
\end{lemma}

\begin{proof}
For this proof, write $y=y_{d,m,R}(x;H_2)$ and $z=z_{m,R}(x)$. Using Definition~\ref{K pieces def} and the identity~\eqref{ezy denominator identity}:
\begin{align*}
& (\log x)^r \cdot 2d^m \frac{\log x}{mR} J_{2a}(z,y) \\
&= (\log x)^r \cdot 2d^m \frac{\log x}{mR} \frac{(35 y^3+128 y^{5/2}+135y^2+128\sqrt y) z+105y^2+256y^{3/2}}{256 z^2e^{z(y+1/y)/2}y^{3/2}} \\
&= \frac{\log^{r+1}x}{128mRH_2^m} \bigg( \frac{35 y^{3/2}+128 y+135y^{1/2}+128y^{-1}}z + \frac{105y^{1/2}+256}{z^2} \bigg) \exp\bigg( {-}\frac{\log x}{R\log(dH_2)} \bigg) \\
&= \frac{\log^{r+1}x}{128mRH_2^m} \exp\bigg( {-}\frac{\log x}{R\log(dH_2)} \bigg) \times \bigg\{ \frac{R^{1/2}}{2m^{1/2}\log^{1/2} x} \bigg( \frac{35m^{3/4}R^{3/4} \log^{3/2}(dH_2)}{\log^{3/4} x}  \\
& \qquad + \frac{128m^{1/2}R^{1/2} \log(dH_2)}{\log^{1/2} x} + \frac{135m^{1/4}R^{1/4} \log^{1/2}(dH_2)}{\log^{1/4} x} \\
& \qquad {}+ \frac{128\log^{1/2} x}{m^{1/2}R^{1/2} \log(dH_2)} \bigg) + \frac R{4m\log x} \bigg( \frac{105m^{1/4}R^{1/4} \log^{1/2}(dH_2)}{\log^{1/4} x} +256 \bigg) \bigg\},
\end{align*}
which can be written as
\begin{align*}
& \frac1{H_2^m} \exp\bigg( {-}\frac{\log x}{R\log(dH_2)} \bigg) \times {} \\
& \qquad \bigg\{ \bigg( \frac{35R^{1/4} \log^{3/2}(dH_2)\log^{r-1/4}x}{256m^{3/4}} + \frac{\log(dH_2)(\log x)^r}{2m} + \frac{135\log^{1/2}(dH_2)\log^{r+1/4}x}{256m^{5/4}R^{1/4}} \\
&\qquad\qquad{}+ \frac{\log^{r+1} x}{2m^2R \log(dH_2)} \bigg) + \bigg( \frac{105R^{1/4} \log^{1/2}(dH_2)\log^{r-1/4}x}{512m^{7/4}} + \frac{(\log x)^r}{2m^2} \bigg) \bigg\} \\
&= \frac1{H_2^m} \exp\bigg( {-}\frac{\log x}{R\log(dH_2)} \bigg) \bigg( \frac{\log^{r+1}x}{2 \log(dH_2)  m^2 R} + \frac{135\log^{1/2}(dH_2) \log^{r+1/4}x}{256 m^{5/4} R^{1/4}} \\
&\qquad{}+ \frac{(m\log(dH_2)+1)(\log x)^r}{2 m^2} + \frac{35(2m \log(dH_2)+3) \log^{1/2}(dH_2)R^{1/4}\log^{r-1/4}x}{512 m^{7/4}} \bigg).
\end{align*}
\end{proof}

\begin{lemma}  \label{2a}
Let $r$, $m$, $R$, $x$, $\lambda$, and $H_2$ be positive real numbers with $x>1$ and $r>\frac14$. Then
$$
P_{2a}(x;m,r,\lambda,H_2,R) \le Q_{2a}(m,r,\lambda,H_2,R).
$$
\end{lemma}

\begin{proof}
By Lemma~\ref{simple calculus lemma}, the four summands in Definition~\ref{J def} for $P_{2a}$ are maximized at $\log x = (r+\ep)R\lambda$ for $\ep\in\{1,\frac14,0,-\frac14\}$. Inserting these respective values of $x$ into the two summands yields the upper bound
\begin{flalign*}
& P_{2a} (x;m,r,\lambda,H_2,R) \le \frac1{H_2^m} \bigg( \frac{((r+1)R\lambda)^{r+1}}{2e^{r+1} \lambda  m^2 R} + \frac{135\lambda^{1/2} ((r+1/4)R\lambda)^{r+1/4}}{256e^{r+1/4} m^{5/4} R^{1/4}} \\
& \hskip8ex + \frac{(m\lambda+1)(r R\lambda)^r}{2e^r m^2} + \frac{35(2m \lambda+3) \lambda^{1/2}R^{1/4}((r-1/4)R\lambda)^{r-1/4}}{512e^{r-1/4} m^{7/4}} \bigg) \\
&\hskip8ex= \frac{R^r}{e^r H_2^m} \bigg( \frac{(r+1)^{r+1}\lambda^r}{2em^2} + \frac{135(r+1/4)^{r+1/4} \lambda^{r+3/4}}{256e^{1/4} m^{5/4}} \\
&\hskip8ex+ \frac{(m\lambda+1)r^r \lambda^r}{2m^2} + \frac{35e^{1/4}(2m \lambda+3)(r-1/4)^{r-1/4} \lambda^{r+1/4}}{512m^{7/4}} \bigg),
\end{flalign*}
which establishes the lemma thanks to Definition~\ref{M def}.
\end{proof}

\begin{definition}  \label{x3 def}
Given integers $m\ge2$ and $d\ge3$ and positive constants $H_2$ and $R$, if $\tau_m$ is as given in Definition \ref{tau and omega def}, define
\[
x_3(m,d,H_2,R) = \exp\big( R\log(dH_2) \big(\sqrt{m\log(dH_2)} - \tau_m\big)^2 \big).
\]
\end{definition}

\begin{lemma}  \label{1bx}
Let $m\ge2$ be an integer, and let $r$, $R$, $x$, $d$, and $H_2$ be positive real numbers with $x>1$, $r\le m+1$, and $dH_2\ge10^8$. Then
\begin{multline*}
(\log x)^r \cdot 2d^m \bigg( \frac{\log x}{mR} \bigg)^{1/2} J_{1b}\big( z_{m,R}(x);y_{d,m,R}(x;H_2) \big) \\
\le \max \big\{ P_{1b}\big(x;m,r,\log(dH_2),H_2,R\big), P_{1b}\big( x_3(m,d,H_2,R); m,r, \log(dH_2),H_2,R\big) \big\}.
\end{multline*}
\end{lemma}

\begin{proof}
In this proof we write $y=y_{d,m,R}(x;H_2)$ and $z=z_{m,R}(x)$. We start with Definition~\ref{K pieces def}:
$$
\begin{array}{l}
(\log x)^r  \cdot 2d^m \bigg( \frac{\log x}{mR} \bigg)^{1/2} J_{1b}(z;y) \\
= (\log x)^r \cdot 2d^m \bigg( \frac{\log x}{mR} \bigg)^{1/2} \sqrt\pi \erfc\bigg( \sqrt{\frac z2} \bigg( \sqrt{y}-\frac1{\sqrt {y}} \bigg) \bigg) \frac{8z+3}{16\sqrt2\, z^{3/2}e^z} \\
= (\log x)^r \cdot 2d^m \bigg( \frac{\log x}{mR} \bigg)^{1/2} \sqrt\pi \erfc\bigg( \sqrt{m\log(dH_2)} - \sqrt{\frac{\log x}{R\log(dH_2)}} \bigg) \frac{8z+3}{16\sqrt2\, z^{3/2}e^z}
\end{array}
$$
by the identity~\eqref{erf argument identity}. Since $e^{-z} = \exp\big( {-}2\sqrt{(m\log x)/R} \big)$, we can express the right-hand side in terms of the function $\Xi_{m,\lambda,\mu,R}(x)$ defined in Definition~\ref{e def}, with $\mu=r+\frac12$ and $\lambda=\log(dH_2)$:
\begin{flalign}  
& (\log x)^r \cdot 2d^m \bigg( \frac{\log x}{mR} \bigg)^{1/2} J_{1b}\big( z_{m,R}(x);y_{d,m,R}(x;H_2) \big) \notag \\
& \hskip19ex = \frac{2d^m}{\sqrt{mR}} \frac{8z+3}{16\sqrt2\, z^{3/2}} \Xi_{m,\lambda,\mu,R}(x). \label{1bx come back to max} 
\end{flalign}

Suppose first that we have
$$
x \le x_3 = \exp \big( R\log(dH_2) (\sqrt{m\log(dH_2)} - \tau_m)^2 \big).
$$
Then by Proposition~\ref{good riddance erfc}(b),
\begin{flalign}
& 2d^m  \frac1{\sqrt{mR}} \frac{8z+3}{16\sqrt2\, z^{3/2}} \Xi_{m,\lambda,\mu,R}(x) \notag \\
&\le 2d^m \frac1{\sqrt{mR}} \frac{8z+3}{16\sqrt2\, z^{3/2}} \omega_m e^{-m\lambda} \exp \bigg( {-} \frac{\log x}{R\log(dH_2)} \bigg) \log^{r+1/2} x \notag \\
&= \frac{\omega_m}{8\sqrt2} \frac{\log^{r+1/2} x}{H_2^m\sqrt{mR}} ( 8z^{-1/2}+3z^{-3/2} ) \exp\bigg( {-}\frac{\log x}{R\log(dH_2)} \bigg) \notag \\
&= \frac{\omega_m}{8\sqrt2} \frac{\log^{r+1/2} x}{H_2^m\sqrt{mR}} \bigg( \frac{8R^{1/4}}{\sqrt 2 (m\log x)^{1/4}} + \frac{3R^{3/4}}{2\sqrt 2 (m\log x)^{3/4}} \bigg) \exp\bigg( {-}\frac{\log x}{R\log(dH_2)} \bigg) \notag \\
&= \frac{\omega_m}{H_2^m} \bigg( \frac{\log^{r+1/4} x}{2m^{3/4}R^{1/4}} + \frac{3R^{1/4}\log^{r-1/4}x}{32m^{5/4}} \bigg) \exp\bigg( {-}\frac{\log x}{R\log(dH_2)} \bigg) \\
&= P_{1b}\big(x;m,r,\log(dH_2),H_2,R\big)  \label{1bx came back to max}
\end{flalign}
by Definition~\ref{J def}. Combining the last two equations establishes the lemma in this range of~$x$.

Now suppose that $x \ge x_3$. By Proposition~\ref{good riddance erfc}(a), the function $\Xi_{m,\lambda,\mu,R}(x)$ is a decreasing function of $x$ in this range, while the function $(8z+3)/16\sqrt2 z^{3/2}$ is also a decreasing function of $x$. Therefore
\[
\frac{2d^m}{\sqrt{mR}} \frac{8z+3}{16\sqrt2\, z^{3/2}} \Xi_{m,\lambda,\mu,R}(x) \le \frac{2d^m}{\sqrt{mR}} \frac{8z(x_3)+3}{16\sqrt2\, z(x_3)^{3/2}} \Xi_{m,\lambda,\mu,R}(x_3);
\]
and then the calculation leading to~\eqref{1bx came back to max} shows that $P_{1b}\big(x_3;m,r,\log(dH_2),H_2,R\big)$ is an upper bound for the latter quantity, which establishes the lemma for this complementary range of~$x$ thanks to equation~\eqref{1bx come back to max}.
\end{proof}

\begin{lemma}  \label{1b}
Let $r$, $m$, $R$, $x$, $\lambda$, and $H_2$ be positive real numbers with $x>1$ and $r>\frac14$. Then
$$
P_{1b}(x;m,r,\lambda,H_2,R) \le Q_{1b}(m,r,\lambda,H_2,R).
$$
\end{lemma}

\begin{proof}
By Lemma~\ref{simple calculus lemma}, the two summands in Definition~\ref{J def} for $P_{1b}$ are maximized at $\log x = (r+\frac14)R\lambda$ and $\log x = (r-\frac14) R\lambda$, respectively. Inserting these respective values of $x$ into the two summands yields the following upper bound for $P_{1b}(x;m,r,d,H_2,R)$ :
\begin{align*}
& \frac{\omega_m}{H_2^m} \bigg( \frac1{2m^{3/4}R^{1/4}} \bigg( \frac{(r+1/4)R\lambda}e \bigg)^{r+1/4} + \frac{3R^{1/4}}{32m^{5/4}} \bigg( \frac{(r-1/4)R\lambda}e \bigg)^{r-1/4} \bigg) \\
&= \frac{\omega_mR^r}{e^r H_2^m} \bigg( \frac{(r+1/4)^{r+1/4}\lambda^{r+1/4}}{2e^{1/4}m^{3/4}}  + \frac{3e^{1/4}(r-1/4)^{r-1/4}\lambda^{r-1/4}}{32m^{5/4}} \bigg),
\end{align*}
which establishes the lemma, upon appealing to Definition~\ref{M def}.
\end{proof}

\begin{lemma}  \label{2bx}
Let $m\ge2$ be an integer, and let $r$, $R$, $x$, $d$, and $H_2$ be positive real numbers with $x>1$, $r\le m+1$, and $dH_2\ge10^8$. Then
\begin{multline*}
(\log x)^r \cdot 2d^m \frac{\log x}{mR} J_{2b}\big(z_{m,R}(x),y_{d,m,R}(x;H_2)\big) \\
\le \max\big\{ P_{2b}\big(x;m,r,\log(dH_2),H_2,R\big), P_{2b}\big( x_3(m,d,H_2,R);m,r,\log(dH_2),H_2,R\big) \big\}.
\end{multline*}
\end{lemma}

\begin{proof}
In this proof, for concision, we write $y$ for $y_{d,m,R}(x;H_2)$ and $z$ for $z_{m,R}(x)$. We start with Definition~\ref{K pieces def}:
\begin{flalign*}
& (\log x)^r  \cdot 2d^m \frac{\log x}{mR} J_{2b}\big( z_{m,R}(x);y_{d,m,R}(x;H_2) \big) \\
&= (\log x)^r \cdot 2d^m \frac{\log x}{mR} \sqrt\pi \erfc\bigg( \sqrt{\frac z2} \bigg( \sqrt{y}-\frac1{\sqrt {y}} \bigg) \bigg) \frac{128 z^2+240 z+105}{256\sqrt2\, z^{5/2}e^z} \\
&= (\log x)^r \cdot 2d^m \frac{\log x}{mR} \sqrt\pi \erfc\bigg( \sqrt{m\log(dH_2)} - \sqrt{\frac{\log x}{R\log(dH_2)}} \bigg) \\
& \times  \frac{128 z^2+240 z+105}{256\sqrt2\, z^{5/2}e^z},
\end{flalign*}
by identity~\eqref{erf argument identity}. Since $e^{-z} = \exp\big( {-}2\sqrt{(m\log x)/R} \big)$, we can write the last quantity here in terms of the function $\Xi_{m,\lambda,\mu,R}(x)$ defined in Definition~\ref{e def}, with $\mu=r+1$ and $\lambda=\log(dH_2)$:
\begin{flalign}
& (\log x)^r \cdot 2d^m \frac{\log x}{mR} J_{2b}\big( z_{m,R}(x);y_{d,m,R}(x;H_2) \big) \notag \\
& = \frac{2d^m}{mR} \frac{128 z^2+240 z+105}{256\sqrt2\, z^{5/2}} \Xi_{m,\lambda,\mu,R}(x).
\label{2bx come back to max} \\ \notag
\end{flalign}

Suppose first that $x \le x_3 = \exp \big( R\log(dH_2) (\sqrt{m\log(dH_2)} - \tau_m)^2 \big)$. Then by Proposition~\ref{good riddance erfc}(b),
\begin{flalign}
& 2d^m  \frac1{mR} \frac{128 z^2+240 z+105}{256\sqrt2\, z^{5/2}} \Xi_{m,\lambda,\mu,R}(x) \notag \\
&\le \frac{2d^m}{mR} \frac{128 z^2+240 z+105}{256\sqrt2\, z^{5/2}} \omega_m e^{-m\lambda} \exp \bigg( {-} \frac{\log x}{R\log(dH_2)} \bigg) \log^{r+1} x \notag \\
&= \frac{\omega_m}{128\sqrt2} \frac{\log^{r+1} x}{H_2^mmR} ( 128z^{-1/2}+240z^{-3/2}+105z^{-5/2} ) \exp\bigg( {-}\frac{\log x}{R\log(dH_2)} \bigg) \notag \\
&= \frac{\omega_m}{128\sqrt2} \frac{\log^{r+1} x}{H_2^mmR} \bigg( \frac{128R^{1/4}}{\sqrt 2 (m\log x)^{1/4}} + \frac{240R^{3/4}}{2\sqrt 2 (m\log x)^{3/4}} + \notag\\
& \hskip22ex \frac{105R^{5/4}}{4\sqrt 2 (m\log x)^{5/4}} \bigg) \exp\bigg( {-}\frac{\log x}{R\log(dH_2)} \bigg) \notag \\
&= \frac{\omega_m}{H_2^m} \bigg( \frac{\log^{r+3/4}x}{2m^{5/4}R^{3/4}} + \frac{15\log^{r+1/4}x}{32m^{7/4}R^{1/4}} + \frac{105R^{1/4}\log^{r-1/4}x}{1024m^{9/4}} \bigg) \exp\bigg( {-}\frac{\log x}{R\log(dH_2)} \bigg) \notag \\
&= P_{2b}\big(x;m,r,\log(dH_2),H_2,R\big)  \label{2bx came back to max}
\end{flalign}
by Definition~\ref{J def}. Combining the last two equations establishes the lemma in this range of~$x$.

Now suppose that $x \ge x_3$. By Proposition~\ref{good riddance erfc}(a), the function $\Xi_{m,\lambda,\mu,R}(x)$ is decreasing in this range, while the function $(128 z^2+240 z+105)/256\sqrt2\, z^{5/2}$ is also a decreasing function of $x$. Therefore
\begin{flalign*}
&\frac{2d^m}{\sqrt{mR}} \frac{128 z^2+240 z+105}{256\sqrt2\, z^{5/2}} \Xi_{m,\lambda,\mu,R}(x) \\
&\le \frac{2d^m}{\sqrt{mR}} \frac{128 z(x_3)^2+240 z(x_3)+105}{256\sqrt2\, z(x_3)^{5/2}} \Xi_{m,\lambda,\mu,R}(x_3);
\end{flalign*}
 and then the calculation~\eqref{2bx came back to max} shows that $P_{2b}\big(x_3;m,r,\log(dH_2),H_2,R\big)$ is an upper bound for the latter quantity, which establishes the lemma for this complementary range of~$x$, via  equation~\eqref{2bx come back to max}.
\end{proof}

\begin{lemma}  \label{2b}
Let $r$, $m$, $R$, $x$, $\lambda$, and $H_2$ be positive real numbers with $x>1$ and $r>\frac14$. Then
$$
P_{2b}(x;m,r,\lambda,H_2,R) \le Q_{2b}(m,r,\lambda,H_2,R).
$$
\end{lemma}

\begin{proof}
By Lemma~\ref{simple calculus lemma}, the three summands in Definition~\ref{J def} for $P_{2b}$ are maximized at $\log x = (r+\ep)R\log(dH_2)$ for $\ep\in\{\frac34,\frac14,-\frac14\}$. We therefore have the upper bound
\begin{flalign*}
&P_{2b}  (x;m,r,\lambda,H_2,R) \\
&\le \frac{\omega_m}{H_2^m} \bigg( \frac1{2m^{5/4}R^{3/4}} \bigg( \frac{(r+3/4)R\log(dH_2)}e \bigg)^{r+3/4} \\
&\qquad{} + \frac{15}{32m^{7/4}R^{1/4}} \bigg( \frac{(r+1/4)R\log(dH_2)}e \bigg)^{r+1/4} \\
&\qquad{} + \frac{105R^{1/4}}{1024m^{9/4}} \bigg( \frac{(r-1/4)R\log(dH_2)}e \bigg)^{r-1/4} \bigg) \\
&= \frac{\omega_mR^r}{e^r H_2^m} \bigg( \frac{(r+3/4)^{r+3/4}\log^{r+3/4}(dH_2)}{2e^{3/4}m^{5/4}} + \frac{15(r+1/4)^{r+1/4}\log^{r+1/4}(dH_2)}{32e^{1/4}m^{7/4}} \\
&\qquad{}+ \frac{105e^{1/4}(r-1/4)^{r-1/4}\log^{r-1/4}(dH_2)}{1024m^{9/4}} \bigg),
\end{flalign*}
which establishes the lemma thanks to Definition~\ref{M def}.
\end{proof}

\subsection{Assembly of the final upper bound for $|\psi(x;q,a) - x/\phi(q)|$}  \label{assembly section}

Finally, after the work of the preceding four sections, we have all of the tools necessary to assemble an explicit upper bound for $\funcF{x} (\log x)^r$. This goal, in turn, was the last step required to convert Proposition~\ref{now we see what needs maximizing} into an explicit upper bound for $\big| \psi(x;q,a) - x/\phi(q) \big|$ (see Theorem~\ref{code this funky theorem} below). The upper bound is rather complicated, but again our paradigm is that any function that can be easily programmed and computed essentially instantly is sufficient for our purposes. At the end of this section, we describe how we derive Theorem~\ref{main psi theorem} from the resulting upper bound.

\begin{definition}  \label{collect bounds def}
Let $d$ and $m$ be positive integers with $m\ge 2$, and let $r,H_2,R$ be positive real numbers. Define
\begin{multline*}
S_{d,m,R}(r,H,H_2) = \funcB + \frac 1\pi Q_2(m,r,\log(dH_2),H_2,R) H^{m+1}  \\ + \bigg( \frac1\pi \log\frac{1}{2\pi} + \frac{C_1}{H_2} \bigg) Q_1(m,r,\log(dH_2),H_2,R) H^{m+1},
\end{multline*}
where $\funcB$ is as in Definition \ref{boundary functions to maximize} and the $Q_j(m,r,\lambda,H_2,R)$ are as in Definition \ref{M def}.
\end{definition}

\begin{prop}  \label{F bounded by S prop}
Let $d$ and $m$ be positive integers with $m\ge2$, and let $r,R,H,H_2$ be positive real numbers such that $\frac14<r\le m+1$, $15 \leq H \leq H_2$, $dH_2 \ge 10^8$, and $\chi$ a character satisfying Hypothesis Z$(H_2,R)$. Then for all $x>1$, we have
\[H^{m+1}\funcF{x} (\log x)^r \le S_{d,m,R}(r,H,H_2).\]
\end{prop}

\begin{proof}
We proceed first under the assumption that $\log x \le R(m+1) \log^2(dH_2)$. Starting from Proposition~\ref{integral bound for S tilde prop}, we apply Proposition~\ref{boundary term to maximize} to conclude that necessarily $\funcF{x} (\log x)^r$ is bounded above by
\begin{equation}  \label{first B1}
\funcBone{x} + (\log x)^r \int_{H_2}^\infty \left(\frac{\partial}{\partial u} M_d(H_2,u)\right) \funcY{u} \,du.
\end{equation}
We then apply Proposition~\ref{convert to I notation prop}, Lemma~\ref{I to K cov lemma}, and Proposition~\ref{bounds for incomplete Bessels prop} to get
\begin{multline}  \label{second B1}
\funcF{x} (\log x)^r \le \funcBone{x} + (\log x)^r \cdot \frac1\pi 2d^m \frac{\log x}{mR} \times \\
\times \bigg( J_{2a}\bigg( 2\sqrt{\frac{m\log x}R}; \sqrt{\frac{mR}{\log x}} \log(dH_2) \bigg) + J_{2b}\bigg( 2\sqrt{\frac{m\log x}R}; \sqrt{\frac{mR}{\log x}} \log(dH_2) \bigg) \bigg) \\
+ (\log x)^r \bigg( \frac1\pi \log\frac{1}{2\pi} + \frac{C_1}{H_2} \bigg) 2d^m\bigg( \frac{\log x}{mR} \bigg)^{1/2} \times \\
\times \bigg( J_{1a}\bigg( 2\sqrt{\frac{m\log x}R}; \sqrt{\frac{mR}{\log x}} \log(dH_2) \bigg) + J_{1b}\bigg( 2\sqrt{\frac{m\log x}R}; \sqrt{\frac{mR}{\log x}} \log(dH_2) \bigg) \bigg).
\end{multline}
Now Lemmas~\ref{1ax},~\ref{2ax},~\ref{1bx}, and~\ref{2bx} yield
\begin{equation}  \label{still fn of x}
\begin{aligned}
\funcF{x} (\log x)^r \le & \funcBone{x} + \frac1\pi \left( P_{2a}(x;m,r,\log(dH_2),H_2,R) + M_2 \right) \\
& + \bigg( \frac1\pi \log\frac{1}{2\pi} + \frac{C_1}{H_2} \bigg) \left( P_{1a}(x;m,r,\log(dH_2),H_2,R) + M_1 \right),\\
\end{aligned}
\end{equation}
where $M_1$ and $M_2$ are
$$
\max\{P_{1b}(x;m,r,\log(dH_2),H_2,R),P_{1b}(x_3(m,d,H_2,R);m,r,\log(dH_2),H_2,R)\}
$$
and
$$
\max\{P_{2b}(x;m,r,\log(dH_2),H_2,R),P_{2b}(x_3(m,d,H_2,R);m,r,\log(dH_2),H_2,R)\},
$$
respectively.
Finally, Lemmas~\ref{boundary term maximized},~\ref{1a},~\ref{2a},~\ref{1b}, and~\ref{2b} give
\begin{align}
&H^{m+1} \funcF{x} (\log x)^r \le \funcB  \notag\\
&+ \frac1\pi \big( Q_{2a}(m,r,\log(dH_2),H_2,R) + Q_{2b}(m,r,\log(dH_2),H_2,R) \big) H^{m+1}
\notag\\
&+ \bigg( \frac1\pi \log\frac{1}{2\pi} + \frac{C_1}{H_2} \bigg) \big( Q_{1a}(m,r,\log(dH_2),H_2,R)
\notag \\
&+ Q_{1b}(m,r,\log(dH_2),H_2,R) \big) H^{m+1}, \label{Bs become B}
\end{align}
which establishes the proposition under the assumption $\log x \le R(m+1) \log^2(dH_2)$.

If, instead, $\log x > R(m+1) \log^2(dH_2)$, then the application of Proposition~\ref{boundary term to maximize} requires us to replace $\funcBone{x}$ by $\funcBtwo{x}$ in the expressions \eqref{first B1}, \eqref{second B1}, and \eqref{still fn of x}, but then Lemma~\ref{boundary term maximized} allows us to replace $\funcBtwo{x}$ by the term $\funcB$ in the transition from equation~\eqref{still fn of x} to equation~\eqref{Bs become B}, and so the end result is the same.
\end{proof}

\begin{definition}\label{def:G with log}
Let $q$ and $m$ be positive integers with $m\geq 2$, and let $x_2,r,H$ be positive real numbers satisfying $x_2>1$ and $H \ge 1$. Let $H_2$ be a function on the divisors of $q$ satisfying $H \le H_2(d)$ for $d\mid q$. We define
\begin{multline*}
\funcG{x_2,r} \\
= \sum_{d|q} \phi^*(d) \left( g_{d,m,R}^{(3)}(x_2;H ,H_2(d))(\log x_2)^r + \frac12 S_{d,m,R}(r,H,H_2(d))\right),
\end{multline*}
where $g^{(3)}_{d,m,R}$ is as in Definition~\ref{phim def} and $S_{d,m,R}$ is as in Definition~\ref{collect bounds def}.
\end{definition}

\begin{prop} \label{Upsilon resovled}
Let $q$ and $m$ be positive integers with $3 \le m \le 25$, and let
$x$, $x_2$, $r$, $R$, and $H$ be positive real numbers with $x \ge x_2 \ge e^{2m+2}$ and $\frac 14 < r \le m+1$ and $R\ge0.435$ and $H\ge H_1(m)$. Let $H_2$ be a function on the divisors of $q$ with $H_2(d) \ge \max\{H,10^8/d\}$ for all $d\mid q$, such that every character $\chi$ with modulus $q$ satisfies Hypothesis Z$(H_2(q^*),R)$, where $q^*$ is the conductor of~$\chi$. Then
  \[ \Psi_{q,m,r}(x;H) < \funcG{x_2,r}.\]
\end{prop}

\begin{proof}
By Definition~\ref{Upsilon and Psi def}, Lemma~\ref{Upsilon breakdown}, and Definition~\ref{def: G},
  \begin{align}
\Psi_{q,m,r}(x;H) &= H^{m+1}\Upsilon_{q,m}(x;H) (\log x)^r \notag \\
      &< \funcG{x} (\log x)^r \notag \\
      &= \sum_{d|q} \phi^*(d) g_{d,m,R}^{(3)}(x;H ,H_2(d))(\log x)^r \notag \\
      &+ \frac 12 \sum_{d|q} \funcF[q=d,H={H_2(d)}]{x} (\log x)^r. \label{eq: some of tthe way}
  \end{align}
The terms in the first summation are straightforward: by hypothesis,
$$
x\geq x_2 \geq e^{2m+2}\geq e^{2r},
$$
and so $(\log x)^r / x^\lambda$ is decreasing for any $\lambda\ge\frac12$. Consequently, by Definition~\ref{phim def},
  \begin{flalign*}
    & g_{d,m,R}^{(3)}  (x;H ,H_2(d))(\log x)^r \\
      &= g_{d,m}^{(1)}(H,H_2(d)) \cdot \frac{(\log x)^r}{x^{1/2}}+ g_{d,m}^{(2)}(H,H_2(d)) \cdot \frac{x^{1/(R\log dH_2(d))} (\log x)^r}{x } \\
      &\le  g_{d,m}^{(1)}(H,H_2(d)) \cdot \frac{(\log x_2)^r}{x_2^{1/2}}+ g_{d,m}^{(2)}(H,H_2(d)) \cdot \frac{x_2^{1/(R\log dH_2(d))} (\log x_2)^r}{x_2} \\
      &= g_{d,m,R}^{(3)}(x_2;H ,H_2(d))(\log x_2)^r.
  \end{flalign*}
(The hypotheses $R\ge0.435$ and $H_2(d) \ge H\ge H_1(m)\geq 102$, combined with $d\ge1$, ensure that the fraction at the end of the second line is of the form $(\log x)^r / x^\lambda$ with $\lambda\ge\frac12$.)

The terms in the second summation of \eqref{eq: some of tthe way} have been addressed, in essence, in Proposition~\ref{F bounded by S prop}. In particular, beginning with Definition~\ref{def: G},
  \begin{align*}
   \funcF[q=d,H={H_2(d)}]{x} (\log x)^r
      &=\sum_{\substack{\chi \mod q \\ q^*=d}} H^{m+1} \funcF[q=\chi,H={H_2(d)}]{x} (\log x)^r \\
      &\le \sum_{\substack{\chi \mod q \\ q^*=d}} S_{d,m,R}(r,H,H_2(d)) \\
      &= \phi^*(d) S_{d,m,R}(r,H,H_2(d)).
  \end{align*}
A comparison to Definition~\ref{def:G with log}  confirms that the last line of \eqref{eq: some of tthe way} is now seen to be bounded by $\funcG{x,r}$.
\end{proof}

The function we now define is ultimately what we compute to obtain our upper bounds for $|\psi(x;q,a) - x/\phi(q)|$ and hence is the main function we program into our code, although (of course) several auxiliary functions from earlier in this paper must also be programmed.

\begin{definition}  \label{D def}
Let $H_0$ be a function on the characters modulo $q$, and let $H_2$ be a function on the divisors of $q$.
Let $W_q(x)$ be as in Definition~\ref{W def}, $\nu(q,H_0,H)$ as in Definition~\ref{Have you heard the good nus?}, $\funcG{x,r}$ as in Definition~\ref{def:G with log}, and $\alpha_{m,k}$ as in Definition~\ref{alpha def}. Then define $D_{q,m,R}(x_2;H_0,H,H_2)$ by
$$
D_{q,m,R}(x_2;H_0,H,H_2) =\frac{1}{\phi(q)} \left( T_1 + T_2 + T_3 + T_4 \right),
$$
where
\begin{align*}
T_1 &= \nu(q,H_0,H) \frac{\log x_2}{\sqrt{x_2}} \\
 T_2  & = \frac{m+1}H \funcG{x_2,m+1}^{\frac 1{m+1}} \left(1+\frac{\nu(q,H_0,H)}{\sqrt {x_2}}\right)^{\frac{m}{m+1}} \\
 T_3  & = \sum_{k=1}^m \frac{\alpha_{m,k}}{2^{m-k}H^{k+1}} \funcG{x_2,\frac{m+1}{k+1}}^{\frac{k+1}{m+1}}
            \left(1+\frac{\nu(q,H_0,H)}{\sqrt {x_2}}\right)^{\frac{m-k}{m+1}} \\
 T_4  & = \frac{2\alpha_{m,m+1}}{H^{m+2}} \funcG{x_2,\frac{m+1}{m+2}}^{\frac{m+2}{m+1}} + \funcW{x_2}\log x_2.
\end{align*}

See Appendix~\ref{analysis of constants section} for an indication of which terms $T_i$ in this expression contribute the most to its value for the ranges of parameters most important for our purposes.
\end{definition}

\begin{theorem}  \label{code this funky theorem}
Let $3\le q\le 10^5$ be an integer, and let $a$ be an integer that is coprime to~$q$.
Let $3\le m\le 25$ be an integer, and let $x_2\ge e^{2m+2}$ and $H\ge H_1(m)$ and $R\ge0.435$ be real numbers.
Let $H_0$ be a function on the characters modulo $q$ with $0\leq H_0(\chi) \leq H$ for every such character.
Let $H_2$ be a function on the divisors of $q$ with $H_2(d) \ge \max\{H,10^8/d\}$ for all $d\mid q$, such that every character $\chi$ with modulus $q$ satisfies Hypothesis Z$(H_2(q^*),R)$, where $q^*$ is the conductor of~$\chi$.
Then for all $x\ge x_2$,
\[
\bigg| \psi(x;q,a) - \frac x{\phi(q)} \bigg| \,\bigg/\! \frac x{\log x} \le D_{q,m,R}(x_2;H_0,H,H_2),
\]
where $D_{q,m,R}(x_2;H_0,H,H_2)$ is as in Definition~\ref{D def}.
\end{theorem}

\begin{proof}
Combine Proposition~\ref{now we see what needs maximizing} (taking note of the remark following its statement) with Proposition~\ref{Upsilon resovled} and Definition~\ref{D def}.
\end{proof}

%

To apply Theorem~\ref{code this funky theorem}, we must use a value of $R$ for which it is guaranteed that Hypothesis Z$(10^8/q,R)$ is satisfied;
fortunately, suitable results are present in the literature, as we record in the following proposition. Once we do so, we will be able to complete the proof of Theorem~\ref{main psi theorem}.

\begin{prop}[Platt, Kadiri, Mossinghoff-Trudgian] \label{R value prop}
Let $1\le q\le 10^5$. Then $q$ satisfies Hypothesis Z$(10^8/q,5.6)$.
\end{prop}

\begin{proof}
By Definition~\ref{hypothesis z} we need to confirm, for every Dirichlet $L$-function modulo~$q$, that every nontrivial zero $\beta+i\gamma$  with $|\gamma| \le 10^8/q$ satisfies $\beta=\tfrac12$, and that every nontrivial zero with $|\gamma| > 10^8/q$ satisfies $\beta \le 1 - 1/{5.6\log(q|\gamma|)}$. For the values of $q$ under consideration, the first assertion was shown by Platt~\cite[Theorem~7.1]{Pla2}, while the second assertion was shown by Kadiri~\cite[Theorem~1.1]{Ka} for $q \geq 3$ and by Mossinghoff and Trudgian \cite{MoTr} for $q \in \{ 1, 2 \}$.
\end{proof}

\begin{proof}[Proof of Theorem~\ref{main psi theorem} for small moduli]
For any $3\le q\le 10^5$,
by Theorem~\ref{code this funky theorem} we obtain an admissible
value for $c_\psi(q)$ by computing $D_{q,m,R}(x_2;H_0,H,H_2)$ for any appropriate values of $m$, $R$, $x_2$, $H_0$, $H$ and $H_2$. We always choose $m\in\{6,7,8,9\}$ and  $R=5.6$, where the latter choice is valid by Proposition~\ref{R value prop}. Then we choose $x_2 = x_2(q)$ as in Definition~\ref{x2 definition} (this satisfies $x_2(q) \ge 10^{11} > e^{22} \geq e^{2m+2}$ as required).

We take $H_2(d)$ to be as large as possible, subject to having verified GRH up to that height for all primitive characters with conductor $d$. By~\cite{Pla3} and \cite{Pla2}, we set
  \[
    H_2(d) = \begin{cases}
      30,610,046,000, & \text{ if $d=1$,} \\
      10^8/d, & \text{if $1<d\leq 10^5$}.
    \end{cases}
    \]
That is, we take $H_2(d)=h_3(d)$ as per Definition~\ref{h3}. We optimize over $m\in\{6,7,8,9\}$ and $H \in [H_1(m),H_2(q)]$, and set $H_0$ according to $H$: for $1\leq d\leq 12$, we choose $H_0(d)$ to be the largest among $10^2,10^3,10^4$ that is smaller than $H$, for $12 < d \leq 1000$, $H_0(d)$ is the larger of $10^2,10^3$ that is smaller than $H$, for $1000 < d \leq 2500$ we take$H_0(d)=100$, for $2500<d\leq 10000$,  $H_0(d) = 10$, and, finally, for $10000<d<100000$ we choose $H_0(d)=0$.




These evaluations establish the inequality~\eqref{ultimate psi bound} for $x \ge x_2(q)$ or $x\ge x_2(\frac q2)$, respectively; we then compute by brute force the smallest positive real number $x_\psi(q)$ such that the inequality~\eqref{ultimate psi bound} holds for all $x \ge x_\psi(q)$ and all $\gcd (a,q)=1$.
See Appendix~\ref{ssec xpsithetapi} for a discussion of these computations.  With these values of $c_\psi$ and
$x_\psi(q)$ in hand, we verify the asserted inequalities $c_\psi(q) < c_0(q)$ and $x_\psi(q) <
x_0(q)$, where $c_0(q), x_0(q)$ are defined in equations~\eqref{c_0(q) definition} and~\eqref{x_0(q) definition}
respectively.
\end{proof}

\section{Deduction of the upper bounds upon $|\theta(x;q,a) - x/\phi(q)|$ and $|\pi(x;q,a) - \Li(x)/\phi(q)|$, for $q\le 10^5$} \label{Sec5}

In this section, we will focus upon obtaining bounds for $|\theta(x;q,a) - x/\phi(q)|$ and $|\pi(x;q,a) - \Li(x)/\phi(q)|$, for small values of $q$, given the bounds for $|\psi(x;q,a) - x/\phi(q)|$ derived in the preceding sections. We also define a variant $\theta_\#(x;q,a)$ of $\theta(x;q,a)$ (see equation~\eqref{thetaH def} below) and establish similar bounds for its error term.

\subsection{Conversion of bounds for $\psi(x;q,a)-x/\phi(q)$ to bounds for $\theta(x;q,a) - x/\phi(q)$}  \label{psi to theta section}

The difference between $\psi(x;q,a)$ and $\theta(x;q,a)$ is, of course, the contribution from the squares of primes, cubes of primes, and so on in the residue class $a\mod q$. We use standard estimates to bound these contributions, and assemble them into the function $\Delta(x;q)$ which we now define. As always, we adopt the viewpoint that any upper bound that can be easily programmed is sufficient for our purposes.

\begin{definition}  \label{xi defs}
Define $\xi_k(q)$ to be the number of $k$th roots of $1$ modulo~$q$. For fixed $k$, the function $\xi_k(q)$ is a multiplicative function of $q$, with values on prime powers given by certain greatest common divisors:
\[
\xi_k(p^r) = \begin{cases}
\gcd (k,p^{r-1}(p-1)), &\text{if $p$ is odd}, \\
\gcd (k,2)  \gcd (k,2^{r-2}), &\text{if $p=2$ and $r\ge2$}, \\
1, &\text{if $p^r=2^1$}.
\end{cases}
\]
Further, define $\xi_k(q,a)$ to be the number of $k$th roots of $a$ modulo~$q$, and note that for $\gcd (a,q)=1$, the quantity $\xi_k(q,a)$ equals either $\xi_k(q)$ or $0$ according to whether $a$ has $k$th roots modulo~$q$ or not.

Then, for real numbers $x>1$, define the functions
\[
\Delta_k(x;q) = \begin{cases}
\displaystyle \min\bigg\{ \frac{2\xi_k(q)}{\phi(q)} \bigg( 1 + \frac{\log (q^k)}{\log(x/q^k)} \bigg), 1 + \frac{k}{2\log x} \bigg\}, &\text{if } x>q^k, \\
\displaystyle1 + \frac{k}{2\log x}, &\text{if } 1<x\le q^k
\end{cases}
\]
and
\[
\Delta(x;q) = \sum_{k=2}^{\lfloor \log x/\log 2 \rfloor} \frac{\log x}{x^{1-1/k}} \Delta_k(x;q).
\]
The graph of $\Delta(x;3)$ is shown in Figure~\ref{fig:Delta(x;3)}. (The jump discontinuities occur each time $x$ passes a power of $2$, which is when the number of summands in the definition of $\Delta(x;q)$ increases.)
\end{definition}

\begin{figure}[bht]
\begin{center}
\begin{picture}(300,100)
\includegraphics[width=4in]{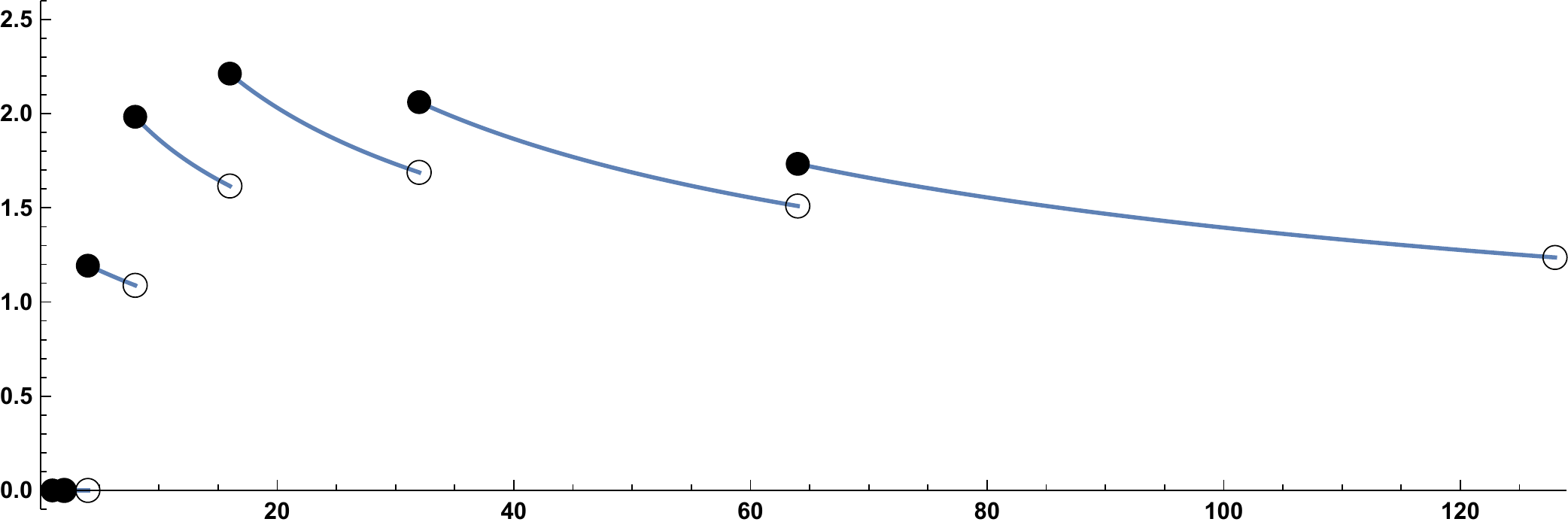}
\end{picture}\end{center}
\caption{$\Delta(x;3)$\label{fig:Delta(x;3)}}
\end{figure}

The following lemma makes it clear why we have defined these quantities.
\begin{lemma} \label{grody psi theta lemma1}
Let $q\ge3$ and let $\gcd (a,q)=1$. For all $x>1$,
\[
0 \le \frac{\psi(x;q,a) - \theta(x;q,a)}{x/\log x} \le \Delta(x;q).
\]
\end{lemma}

\begin{proof}
From their definitions, we have the exact formula
\begin{align*}
0 \le \psi(x;q,a) - \theta(x;q,a) &= \sum_{k=2}^{\lfloor \log x/\log 2 \rfloor} \sum_{\substack{b\mod q \\ b^k \equiv a\mod q}} \theta(x^{1/k};q,b).
\end{align*}
The number of terms in the inner sum is either $0$ or $\xi_k(q)$. Appealing to the Brun--Titchmarsh theorem~\cite[Theorem 2]{MV2},
\begin{flalign*}
\theta(x^{1/k};q,b) & \le \log(x^{1/k}) \, \pi(x^{1/k};q,b) \\
& < \log(x^{1/k}) \frac{2x^{1/k}}{\phi(q)\log(x^{1/k}/q)} = \frac{2x^{1/k}}{\phi(q)} \bigg( 1 + \frac{\log q^k}{\log(x/q^k)} \bigg), \\
\end{flalign*}
and therefore
\[
\sum_{\substack{b\mod q \\ b^k \equiv a\mod q}} \theta(x^{1/k};q,b) < x^{1/k} \cdot \frac{2\xi_k(q)}{\phi(q)} \bigg( 1 + \frac{\log q^k}{\log(x/q^k)} \bigg).
\]
Moreover, for $x>1$,
\[
\sum_{\substack{b\mod q \\ b^k \equiv a\mod q}} \theta(x^{1/k};q,b) \le \theta(x^{1/k}) < x^{1/k} + \frac{x^{1/k}}{2\log(x^{1/k})} = x^{1/k} \bigg( 1 + \frac{k}{2\log x} \bigg),
\]
where the second inequality was given by Rosser and Schoenfeld~\cite[Theorem 4, page 70]{RS}. We thus have, for $x>1$,
$$
\sum_{\substack{b\mod q \\ b^k \equiv a\mod q}} \theta(x^{1/k};q,b) \leq x^{1/k} \Delta_k(x;q).
$$
It follows that
\begin{align*}
0 \le \psi(x;q,a) - \theta(x;q,a)
   &= \sum_{\substack{p^k \leq x \\ p^k \equiv a \mod q \\ k\geq 2}} \log p = \sum_{k=2}^{\lfloor \log x/\log 2\rfloor} \sum_{\substack{p\leq x^{1/k} \\ p^k \equiv a \mod q}} \log p \\
   &= \sum_{k=2}^{\lfloor \log x/\log 2\rfloor} \sum_{\substack{b\mod q \\ b^k \equiv a \mod q}} \sum_{\substack{p\leq x^{1/k} \\ p \equiv b \mod q}} \log p \\
   &= \sum_{k=2}^{\lfloor \log x/\log 2\rfloor} \sum_{\substack{b\mod q \\ b^k \equiv a \mod q}} \theta(x^{1/k};q,b) \\
   &\leq \sum_{k=2}^{\lfloor \log x/\log 2\rfloor} x^{1/k} \Delta_k(x;q) =\frac{x}{\log x} \Delta(x;q),
\end{align*}
which is equivalent to the statement of the lemma.
\end{proof}

When examining the fine-scale distribution of prime counting functions such as $\theta(x;q,a)$, one often considers the limiting (logarithmic) distribution of the normalized error term $(\theta(x;q,a)-x/\phi(q))/\sqrt x$. It is known that this distribution is symmetric, but not necessarily around $0$; rather, it is symmetric around $-\xi_2(q,a)/\phi(q)$, where $\xi_2(q,a)$ is the number of square roots of $a$ modulo~$q$ as in Definition~\ref{xi defs}. There is consequently some interest in the variant error term
\[
\bigg| \theta(x;q,a)- \bigg( \frac x{\phi(q)} - \frac{\xi_2(q,a)\sqrt x}{\phi(q)} \bigg) \bigg|.
\]
For this reason, we define the slightly artificial function
\begin{equation}  \label{thetaH def}
\theta_\#(x;q,a) = \theta(x;q,a) + \frac{\xi_2(q,a)\sqrt x}{\phi(q)}
\end{equation}
and, where the effort involved is modest, establish our error bounds for $|\theta_\#(x;q,a) - x/\phi(q)|$ alongside those for $|\theta(x;q,a) - x/\phi(q)|$.

\begin{lemma} \label{grody psi theta lemma}
Let $q\ge3$ and let $\gcd (a,q)=1$. For all $x\geq 4$,
\[
\bigg| \frac{\psi(x;q,a) - \theta_\#(x;q,a)}{x/\log x} \bigg| \le \Delta(x;q).
\]
\end{lemma}

\begin{proof}
The upper bound on the quantity inside the absolute value follows immediately from Lemma~{\ref{grody psi theta lemma1}}. As for the lower bound, since $\psi(x;q,a) \ge \theta(x;q,a)$ we have
$$
-\frac{\psi(x;q,a) - \theta_\#(x;q,a)}{x/\log x} = \frac{\big( \theta(x;q,a) + \xi_2(q,a)\sqrt x/\phi(q) \big) - \psi(x;q,a)}{x/\log x}
$$
and hence
$$
-\frac{\psi(x;q,a) - \theta_\#(x;q,a)}{x/\log x} \le \frac{\xi_2(q,a)\sqrt x/\phi(q)}{x/\log x} \le \frac{\xi_2(q)\log x}{\phi(q)\sqrt x}.
$$
Observe that
$$
{\xi_2(q)\log x}/{(\phi(q)\sqrt x)}< \left(1+\frac{k}{2\log x}\right) \frac{\log x}{\sqrt{x}}
$$
as $\xi_2(q)\leq \phi(q)$, and for $x>q^k$ trivially
  \[\frac{\xi_2(q)\log x}{\phi(q)\sqrt x} < \frac{2\xi_k(q)}{\phi(q)} \bigg( 1 + \frac{\log (q^k)}{\log(x/q^k)} \bigg) \cdot  \frac{\log x}{\sqrt{x}}. \]
Thus,
  \[\frac{\xi_2(q)\log x}{\phi(q)\sqrt x} \leq  \frac{\log x}{\sqrt{x}} \Delta_2(x;q),\]
and as $x\geq 4$, we have
$$
\frac{\log x}{x^{1-1/2}} \Delta_2(x;q) \le \Delta_2(x;q) \le \Delta(x;q).
$$
\end{proof}

We cannot quite say that $\Delta(x;q)$ is a decreasing function of $x$ due to its jump discontinuities (as we can see for $q=3$ in Figure~\ref{fig:Delta(x;3)}). However, the maximum effect of these discontinuities is quite small, and the following lemma will suffice for our purposes. Thereafter we will establish an analogue of Theorem~\ref{code this funky theorem} for $\theta(x;q,a)$, which enable us to complete the proof of Theorem~\ref{main theta theorem}.

\begin{lemma}  \label{Delta almost increasing lemma}
Let $q\ge3$ be an integer and $x_2 > e^2$. For $x>x_2$,
\[
\Delta(x;q) < \Delta(x_2;q) + \frac{6\log x_2}{x_2}.
\]
\end{lemma}

\begin{proof}
From Definition~\ref{xi defs}, we see that for a given $q$ and $k\ge2$, the function $\Delta_k(x;q)$ is a decreasing function of $x$. Since $(\log x)/x^{1-1/k}$ is decreasing for $x>e^{k/(k-1)}$ and hence certainly for $x>e^2$, the function $\Delta(x;q)$, shown with $q=3$ in Figure~\ref{fig:Delta(x;3)},  is decreasing in $x$, except that it has positive jump discontinuities every time a new summand is introduced. So although we cannot say simply that $\Delta(x;q) \le \Delta(x_2;q)$, we can say that $\Delta(x;q)$ is at most $\Delta(x_2;q)$ plus the sum of all the jump discontinuities at values greater than $x_2$. It remains to show that this sum of jump discontinuities is less than $(6\log x_2)/x_2$.

The summand $k=j$ is introduced at $x=2^j$, and its value is
\[
\frac{\log(2^j)}{(2^j)^{1-1/j}} \Delta_j(2^j,q) = \frac{\log(2^j)}{(2^j)^{1-1/j}}  \bigg( 1 + \frac1{2j\log(2^j)} \bigg) = \frac{j\log2}{2^{j-1}} + \frac1{j2^j},
\]
since $2^j < q^j$. Note that for any $d\ge1$,
\[
\sum_{j=d}^\infty \frac{j\log2}{2^{j-1}} = \frac{(d+1)\log2}{2^{d-2}} \quad\text{and}\quad \sum_{j=d}^\infty \frac1{j2^j} < \frac1d \sum_{j=d}^\infty \frac1{2^j} = \frac1{d2^{d-1}}.
\]
For a given $x_2$, the first jump discontinuity lies at an integer $d$ such that $2^d > x_2$, which means that the corresponding sum of jump discontinuities can be estimated by
\begin{equation} \label{fished}
\frac{(d+1)\log2}{2^{d-2}} + \frac1{d2^{d-1}} < \frac{(\frac{\log x_2}{\log 2}+1)\log 2}{x_2/4} + \frac1{\frac{\log x_2}{\log 2} x_2/2}.
\end{equation}
This last quantity is just
\begin{equation} \label{fished2}
\frac{4\log(2x_2)+(2\log 2)/\log x_2}{x_2} < \frac{6\log x_2}{x_2}.
\end{equation}
Here, the inequality in (\ref{fished}) holds because $\frac d{2^d}$ is a decreasing function of $d$ for $2^d>e$; inequality (\ref{fished2}), which is valid already when $x_2=e^2$, holds because the ratio of the two sides is a decreasing function of $x_2$.
\end{proof}

\begin{theorem}  \label{code this funky theta}
Let $3\le q\le 10^5$ be an integer, and let $a$ be an integer that is coprime to~$q$.
Let $3\le m\le 25$ be an integer, and let $x_2\ge e^{2m+2}$, $H\ge H_1(m)$ and $R\ge0.435$ be real numbers.
Let $H_0$ be a function on the characters modulo $q$ with $0\leq H_0(\chi) \leq H$ for every such character.
Let $H_2$ be a function on the divisors of $q$ with $H_2(d) \ge \max\{H,10^8/d\}$ for all $d\mid q$, such that every character $\chi$ with modulus $q$ satisfies Hypothesis Z$(H_2(q^*),R)$, where $q^*$ is the conductor of~$\chi$.
Then for all $x\ge x_2$,
\[
\bigg| \theta(x;q,a) - \frac x{\phi(q)} \bigg| \,\bigg/\! \frac x{\log x} \le D_{q,m,R}(x_2;H_0,H,H_2) + \Delta(x_2;q) + \frac{6\log x_2}{x_2},
\]
where $D_{q,m,R}(x_2;H_0,H,H_2)$ is defined in Definition~\ref{D def} and $\Delta(x_2;q)$ is defined in Definition~\ref{xi defs}.
The same upper bound holds for
\begin{equation}  \label{theta hash object}
\bigg| \theta(x;q,a) - \frac{x-\xi_2(q,a) \sqrt x}{\phi(q)} \bigg| \,\bigg/\! \frac x{\log x} = \bigg| \theta_\#(x;q,a) - \frac x{\phi(q)} \bigg| \,\bigg/\! \frac x{\log x},
\end{equation}
where $\xi_2(q,a)$ is as in Definition~\ref{xi defs} and $\theta_\#(x;q,a)$ is as in equation~\eqref{thetaH def}.
\end{theorem}

\begin{proof}
Since $\big| \theta(x;q,a) - \frac x{\phi(q)} \big| \le \big| \psi(x;q,a) - \frac x{\phi(q)} \big| + \big| \psi(x;q,a) - \theta(x;q,a) \big|$, it suffices to combine Theorem~\ref{code this funky theorem} with Lemmas~\ref{grody psi theta lemma1} and~\ref{Delta almost increasing lemma}. To establish the inequality~\eqref{theta hash object}, we simply replace Lemma~\ref{grody psi theta lemma1} with Lemma~\ref{grody psi theta lemma}.
\end{proof}

%

\begin{proof}[Proof of Theorem~\ref{main theta theorem} for small moduli]
The remaining argument is essentially the same as the proof of Theorem~\ref{main psi theorem} (which appears at the end of Section~\ref{assembly section}), but using Theorem~\ref{code this funky theta}
instead of Theorem~\ref{code this funky theorem}.
\end{proof}

\subsection{Conversion of estimates for $\theta(x;q,a)$ to estimates for $\pi(x;q,a)$ and for \\ $p_n(q,a)$}  \label{theta to pi section}

There is a natural partial summation argument that derives information for $\pi(x;q,a)$ from information for $\theta(x;q,a)$. Two terms arise while integrating by parts in such an argument: a main term, which is a small multiple of the hypothesized error bound for $\theta(x;q,a)$; and several boundary terms, one of which is guaranteed to be negative. To obtain a simple upper bound of the type that appears in Theorem~\ref{main pi theorem}, we define a function that collects most of these boundary terms together, and work under an otherwise artificial assumption (see equation~\eqref{last two terms} below) that this function is smaller than the remaining negative boundary term.

\begin{definition}  \label{double error def}
Given a positive integer $q$, an integer $a$ that is relatively prime to~$q$ and a real number $u$, define
\[
E(u;q,a) = \pi(u;q,a) - \frac{\Li(u)}{\phi(q)} - \frac1{\log u} \bigg( \theta(u;q,a) - \frac{u}{\phi(q)} \bigg).
\]
\end{definition}

\begin{prop}  \label{first theta to pi prop}
Let $q$ be a positive integer, and let $a$ be an integer that is relatively prime to~$q$. Let $\kappa$ and $x_3$ be positive real numbers (which may depend on $q$ and~$a$).
Suppose we have an estimate of the form
\begin{equation}  \label{hypothesized theta estimate}
\bigg| \theta(x;q,a) - \frac x{\phi(q)} \bigg| \le \frac{\kappa x}{\log x} \quad\text{for } x\ge x_3,
\end{equation}
and also that the inequality
\begin{equation}  \label{last two terms}
|E(x_3;q,a)| \le \frac{\kappa x_3}{(\log x_3-2)\log^2x_3}
\end{equation}
is satisfied. Then
\begin{equation*}
\bigg| \pi(x;q,a) - \frac{\Li(x)}{\phi(q)} \bigg| \le \frac{\kappa(\log x_3-1)}{\log x_3-2} \frac x{\log^2x} \quad\text{for } x\ge x_3.
\end{equation*}
\end{prop}

\begin{proof}
By partial summation,
\begin{align}
\pi(x;q,a) - \frac{\Li(x)}{\phi(q)} = \pi(x_3;q,a) - \frac{\Li(x_3)}{\phi(q)} + \int_{x_3}^x \frac1{\log t} \,d\bigg( \theta(x;q,a) - \frac x{\phi(q)} \bigg) \notag \\
= \pi(x_3;q,a) - \frac{\Li(x_3)}{\phi(q)} + \frac{\theta(x;q,a) - x/\phi(q)}{\log x} - \frac{\theta(x_3;q,a) - x_3/\phi(q)}{\log x_3} \notag \\
\qquad{} + \int_{x_3}^x \bigg( \theta(x;q,a) - \frac x{\phi(q)} \bigg) \frac{dt}{t\log^2t} \notag \\
= E(x_3;q,a) + \frac{\theta(x;q,a) - x/\phi(q)}{\log x} + \int_{x_3}^x \bigg( \theta(x;q,a) - \frac x{\phi(q)} \bigg) \frac{dt}{t\log^2t}.  \label{theta to pi PS}
\end{align}
Using the hypothesized bound~\eqref{hypothesized theta estimate} and the triangle inequality, we see that
\begin{align*}
\bigg| \pi(x;q,a) - \frac{\Li(x)}{\phi(q)} \bigg| &\le |E(x_3;q,a)| + \frac{\kappa x}{\log^2 x} + \int_{x_3}^x \frac\kappa{\log^3t} \,dt \\
&\le |E(x_3;q,a)| + \frac{\kappa x}{\log^2 x} + \frac\kappa{\log x_3-2} \int_{x_3}^x \frac{\log t-2}{\log^3t} \,dt \\
&= |E(x_3;q,a)| + \frac{\kappa x}{\log^2 x} + \frac\kappa{\log x_3-2} \frac t{\log^2t} \bigg|_{x_3}^x \\
&= |E(x_3;q,a)| + \frac{\kappa(\log x_3-1)}{\log x_3-2} \frac x{\log^2x} - \frac{\kappa x_3}{(\log x_3-2)\log^2x_3} \\
&\le \frac{\kappa(\log x_3-1)}{\log x_3-2} \frac x{\log^2x},
\end{align*}
where the last step used the inequality~\eqref{last two terms}.
\end{proof}

\begin{proof}[Proof of Theorem~\ref{main pi theorem} for small moduli]
For any $3\le q\le10^5$, Theorem~\ref{main theta theorem} gives the hypothesis~\eqref{hypothesized theta estimate} with
$\kappa = c_\theta(q)$ and any $x_3\ge x_\theta(q)$. The results of our calculations of the quantities
$x_\theta(q)$ satisfy
\begin{equation*} \label{xtheta upper bound}
  x_\theta(q) \le x_\theta(3) = 7{,}932{,}309{,}757 < 10^{11} \text{ for all } 3\le q\le10^5
\end{equation*}
(links to the the full table of $x_\theta(q)$ can be found in Appendix~\ref{ssec xpsithetapi}), and therefore we may choose
$x_3=10^{11}$. We then computationally verify the inequality~\eqref{last two terms} for $\kappa = c_\theta(q)$ and
$x_3=10^{11}$.
By Proposition~\ref{first theta to pi prop}, we set
$$
c_\pi(q) = {c_\theta(q)(\log(10^{11})-1)}/(\log(10^{11})-2)
$$
and
verify the inequality~$c_\pi(q)<c_0(q)$. See Appendix~\ref{ssec cpsithetapi} for the details of the computations
involved.

This argument establishes the inequality~\eqref{ultimate pi bound} for all $x\ge 10^{11}$. By exhaustive computation
of $\pi(x;q,a)$ for small~$x$, we find the smallest positive real number $x_\pi(q)$ such that the
inequality~\eqref{ultimate psi bound} holds for all $x \ge x_\pi(q)$, and verify the inequality $x_\pi(q) < x_0(q)$.
See Appendix~\ref{ssec xpsithetapi} for details of the computations involved.
\end{proof}

If we prefer to compare $\pi(x;q,a)$ to $x/\log x$ (as in Theorem~\ref{pi nice lower bound theorem}) rather than to $\Li(x)$ (as in Theorem~\ref{main pi theorem}), we may do so after establishing the following two routine bounds upon $\Li(x)$.

\begin{lemma}  \label{Li ineq 4 terms lemma}
We have $\displaystyle\Li(x) > \frac x{\log x}+\frac x{\log^2x} + \frac{2x}{\log^3x} + \frac{6x}{\log^4x}$ for all $x\ge190$.
\end{lemma}

\begin{proof}
Repeated integration by parts gives from
$$
\Li(x) = \int_0^x \frac{dt}{\log t} - \int_0^2 \frac{dt}{\log t}
$$
 the identity
\[
\Li(x) =  \frac x{\log x}+\frac x{\log^2x} + \frac{2x}{\log^3x}
+ \frac{6x}{\log^4x} + \bigg( \int_0^x \frac{24\,dt}{\log^5 t} - \int_0^2 \frac{dt}{\log t} \bigg).
\]
The last term (the difference of integrals) is an increasing function of $x$ for $x>1$, and direct calculation shows
that it is positive for $x=190$.
\end{proof}

\begin{lemma}  \label{Li upper bound lemma}
We have $\displaystyle \Li(x) < \frac x{\log x} + \frac{3x}{2\log^2x}$ for all $x\ge1865$.
\end{lemma}

\begin{proof}
Define $f(x) = \big( \Li(x) - \frac x{\log x} \big) \big/ \frac x{\log^2x}$.
Since $x\ge 190$, Lemma~\ref{Li ineq 4 terms lemma} implies
\begin{align*}
{x^2}f'(x) = & {x (\log x-1)-\Li(x) (\log x-2) \log x} \\
&< x (\log x-1)- \bigg( \frac x{\log x}+\frac x{\log^2x} + \frac{2x}{\log^3x} + \frac{6x}{\log^4x} \bigg) (\log x-2)
\log x \\
&= \frac{2 x (6-\log x)}{\log ^3x},
\end{align*}
which is clearly negative for $x\ge404>e^6$. In particular, $f'(x)<0$ for $x\ge404$, whereby $f(x)$ is decreasing for such $x$. The desired result  follows
from directly calculating that $f(1865) < \frac32$.
\end{proof}

\begin{proof}[Proof of Theorem~\ref{pi nice lower bound theorem}]
From Theorem~\ref{main pi theorem}, we know that for $x>x_\pi(q)$,
\[
\pi(x;q,a) > \frac{\Li(x)}{\phi(q)} - c_\pi(q) \frac{x}{\log^2x}.
\]
The results of our calculations of the quantities $x_\pi(q)$ (see Appendix~\ref{ssec xpsithetapi} for details) satisfy
\begin{equation} \label{xpi lower bound}
  x_\pi(q) \geq x_\pi(99{,}989) = 14{,}735 \text{ for all } 3\le q\le10^5.
\end{equation}
In particular, $x_\pi(q) > 190$, and thus Lemma~\ref{Li ineq 4 terms lemma} implies that $\Li(x) > \frac{x}{\log x} +
\frac{x}{\log^2x}$. Hence
\[
\pi(x;q,a) > \frac{x}{\phi(q)\log x} \bigg( 1 + (1 - c_\pi(q) \phi(q) ) \frac{1}{\log x} \bigg),
\]
and the right-hand side exceeds $\frac{x}{\phi(q)\log x}$ under the hypothesis $c_\pi(q)\phi(q) < 1$. The fact that this
hypothesis holds for $q\le1200$ follows from direct calculation (see Appendix~\ref{ssec cpsithetapi} for details).

Similarly, combining Theorem~\ref{main pi theorem} and Lemma~\ref{Li upper bound lemma} gives us
\[
\pi(x;q,a) < \frac{x}{\phi(q)\log x} \bigg( 1 + (3 + 2c_\pi(q) \phi(q) ) \frac{1}{2 \log x} \bigg).
\]
The assumption that $c_\pi(q) \phi(q) < 1$ yields the desired result.
\end{proof}

Upper bounds for $\pi(x;q,a)$ are equivalent to lower bounds for $p_n(q,a)$, the $n$th smallest prime that is congruent to $a\mod q$, and vice versa; the following two proofs provide the details.

\begin{proof}[Proof of the upper bound in Theorem~\ref{pnqa bounds theorem}]
To simplify notation, we abbreviate the term $p_n(q,a)$ by $p_n$ during this proof.
If $p_n \le x_\pi(q)$ then there is nothing to prove, so we may assume that $p_n > x_\pi(q)$.
From Theorem~\ref{pi nice lower bound theorem} with $x=p_n$,
\[
n = \pi(p_n;q,a) > \frac{p_n}{\phi(q)\log p_n},
\]
and therefore
\begin{equation} \label{nphi inequality}
n\phi(q) > \frac{p_n}{\log p_n}.
\end{equation}
Taking logarithms of inequality~\eqref{nphi inequality},
\[
\log \big( n\phi(q) \big) > \log \bigg( \frac{p_n}{\log p_n} \bigg) = \log p_n \cdot \bigg( 1 - \frac{\log\log p_n}{\log p_n} \bigg),
\]
which implies
\[
\log \big( n\phi(q) \big) \bigg( 1 + \frac{4\log\log p_n}{3\log p_n} \bigg) > \log p_n \cdot \bigg( 1 - \frac{\log\log p_n}{\log p_n} \bigg) \bigg( 1 + \frac{4\log\log p_n}{3\log p_n} \bigg).
\]
The function $(1-t)(1+\frac43t)$ is greater than $1$ for $0<t<\frac14$, and $0<\frac{\log\log p}{\log p}<\frac14$ for
all $p\ge6000$. Since~\eqref{xpi lower bound} implies that $x_\pi(q) > 6000$, the previous inequality
thus gives
\[
\log \big( n\phi(q) \big) \bigg( 1 + \frac{4\log\log p_n}{3\log p_n} \bigg) > \log p_n.
\]
Furthermore, the function $\frac{\log\log t}{\log t}$ is decreasing for $t\ge16>e^e$. If $p_n \le n\phi(q)$ then the desired upper bound is satisfied (other than the trivial case $n\phi(q)=2$, for which $p_n \le 7 < x_\pi(q)$ is easily checked by hand), so we may also assume that $p_n > n\phi(q)$. It follows that
\[
\log \big( n\phi(q) \big) \bigg( 1 + \frac{4\log\log (n\phi(q))}{3\log (n\phi(q))} \bigg) > \log \big( n\phi(q) \big) \bigg( 1 + \frac{4\log\log p_n}{3\log p_n} \bigg) > \log p_n.
\]
Using this upper bound in inequality~\eqref{nphi inequality} yields
\begin{equation} \label{goofy}
n\phi(q) \log \big( n\phi(q) \big) \bigg( 1 + \frac{4\log\log (n\phi(q))}{3\log (n\phi(q))} \bigg) > n\phi(q)\log p_n > p_n,
\end{equation}
which is the desired inequality.
\end{proof}

\begin{proof}[Proof of the lower bound in Theorem~\ref{pnqa
bounds theorem}]
We again abbreviate $p_n(q,a)$ as $p_n$ during this proof.
If $p_n \le x_\pi(q)$ then there is nothing to prove, so we may assume that $p_n > x_\pi(q)$; in particular, $p_n >
14{,}735$ by equation~\eqref{xpi lower bound}.
In this case, we know from equation~\eqref{goofy} that
\begin{equation}  \label{sneaky nphi ineq}
f \big(\log(n\phi) \big) = n\phi(q) \big( \log(n\phi(q)) + \tfrac43\log\log(n\phi(q)) \big) > p_n > 14{,}735,
\end{equation}
where $f(t) = e^t (t+\frac{4}{3}\log t)$ is increasing for all $t>0$. Since $f(7.2) < 14{,}735$, we see that
the inequality~\eqref{sneaky nphi ineq} implies that $\log(n\phi(q)) > 7.2$.

Now, suppose for the sake of contradiction that $p_n(q,a) \le n\phi(q) \log(n\phi(q))$. In particular,
\begin{align}
n = \pi(p_n;q,a) &\le \pi\big( n\phi(q) \log(n\phi(q)) ; q,a \big) \notag \\
&\le \frac{\Li\big( n\phi(q) \log(n\phi(q)) ; q,a \big)}{\phi(q)} + c_\pi(q) \frac{n\phi(q) \log(n\phi(q))}{\log^2\big( n\phi(q) \log(n\phi(q)) \big)} \notag \\
&< \frac{n \log(n\phi(q))}{\log \big( n\phi(q) \log(n\phi(q)) \big)}
+ \frac{5 n \log(n\phi(q))}{2\log^2\big( n\phi(q) \log(n\phi(q)) \big)}, \label{loggy inequality}
\end{align}
where the middle inequality used Theorem~\ref{main pi theorem} and the assumptions
$$
n\phi(q) \log(n\phi(q)) \ge p_n > x_\pi(q),
$$
and the last inequality used Lemma~\ref{Li upper bound lemma} and the assumptions
$$
n\phi(q) \log(n\phi(q))  \ge p_n > x_\pi(q) > 430.
$$

Define the function
\[
g(t) = \frac t{t+\log t} + \frac{5t}{2 (t+\log t)^2},
\]
so that the inequality~\eqref{loggy inequality} is equivalent to the statement that $g\big( \!\log(n \phi(q)) \big)>1$.
On the other hand, $g(t)$ is decreasing for $t < t_0 \approx 21.8$ and then strictly increasing for all $t>t_0$. Since
$\lim_{t \to \infty} g(t) =1$ and $g(7.2)<1$, it follows that $g(t)<1$ for all $t>7.2$, a contradiction.
\end{proof}

For moduli $q$ that are not too large, our calculations of the constants $c_\pi(q)$ allow us to establish clean and explicit versions of Theorems~\ref{pi nice lower bound theorem} and~\ref{pnqa bounds theorem} with a bit of additional computation.

\begin{proof}[Proof of Corollary~\ref{cor nice pi pn bounds}]
For $q=1$ and $q=2$, we may quote results of Rosser and Schoenfeld: the bounds on $\pi(x;q,a)$ follow from~\cite[Theorem
1 and Corollary 1]{RS}, while the bounds on $p_n(q,a)$ follow from~\cite[Theorem 3 and its corollary]{RS}. For $3\le
q\le 1200$, we verify from the results of our calculation of the constants $c_\pi(q)$ that $c_\pi(q)\phi(q)<1$ (see
Appendix~\ref{ssec cpsithetapi} for details), which establishes the corollary in the weaker ranges $x>x_\pi(q)$ and
$p_n(q,a) > x_\pi(q)$. For each of these moduli, an explicit computation for $x$ up to $x_\pi(q)$ confirms that the
asserted inequalities in fact hold once $x\ge 50q^2$ and $p_n(q,a) \ge 22q^2$, as required. See Appendix~\ref{ssec pipn
comp} for details of these last computations.
\end{proof}

We remark that our methods for large moduli (consider for example Proposition~\ref{gory pi inequality prop} below with
$Z=3$) would allow us to obtain the inequalities in Corollary~\ref{cor nice pi pn bounds} for $q>10^5$; by altering the
constants in our arguments in Section~\ref{Sec6}, we could in fact deduce those inequalities for all moduli $q>1200$.
The established range of validity of those inequalities, however, would be substantially worse than the lower bounds
$50q^2$ and $22q^2$ given in Corollary~\ref{cor nice pi pn bounds}: they would instead take the form
$\exp\big(\kappa\sqrt q(\log q)^3\big)$ for some absolute constant~$\kappa$.

\section{Estimation of $|\psi(x;q,a) - x/\phi(q)|$, $|\theta(x;q,a) - x/\phi(q)|$, and $|\pi(x;q,a) - \Li(x)/\phi(q)|$, for $q\ge10^5$} \label{Sec6}

In this section, we will derive bounds upon our various prime counting functions for large values of the modulus $q$, specifically for $q \geq 10^5$. In this situation, our methods allow us to prove inequalities of comparable strength to those for small $q$ (and indeed even stronger inequalities), but only when the parameter $x$ is extremely large: one requires a lower bound for $x$ of the shape $\log x \gg \sqrt{q} \log^3 q$, which is well beyond computational limits. Because of this limitation, we have opted for clean statements over minimized constants.

The reason that the parameter $x$ must be extremely large in such results, as is well known, is that we must take into account the possibility of ``exceptional zeros'' extremely close to $s=1$. We use the following explicit definition of exceptional zero in this paper.

\begin{definition}  \label{R1 exceptional def}
Define $R_1 = 9.645908801$.
We define an \emph{exceptional} zero of $L(s,\chi)$ to be a real zero $\beta$ of $L(s,\chi)$ with $\beta \ge 1 - \frac{1}{R_1\log q}$. By work of McCurley~\cite[Theorem 1]{Mc1}, we know that Hypothesis~Z${}_1(9.645908801)$ holds for the relevant moduli $q\ge10^5$ (as per Definition~\ref{hypothesis z}), and therefore there can be at most one exceptional zero among all of the Dirichlet $L$-functions to a given modulus~$q$.
\end{definition}

The first goal of this section is a variant of Proposition \ref{quoted prop ABC}, which is essentially Theorem 3.6 of McCurley \cite{Mc1} but where we relax the assumption that the $L$-functions involved satisfy GRH(1):

\begin{prop}  \label{quoted prop ABC2}
Let $x>2$ and $H \geq 1$ be real numbers, let $q\ge10^5$ and $m\ge1$ be integers, and let $0<\delta<\frac{x-2}{mx}$ be a real number. Then for every integer $a$ with $\gcd (a,q)=1$,
\begin{equation}  \label{yikes}
\frac{\phi(q)}x \bigg| \psi(x;q,a) - \frac x{\phi(q)} \bigg| <
\funcU{x} + \frac{m\delta}2 + \funcV{x} + \epsilon_1,
\end{equation}
where $\funcU{x}$ and $\funcV{x}$ are as defined in equations~\eqref{U def} and~\eqref{V def} and
$$
\epsilon_1 < \frac{\phi (q)}{x} \left( \frac{\log q \cdot \log x}{\log 2} + 0.2516 q \log q \right).
$$
\end{prop}

\noindent This statement is extremely close to that of Proposition~\ref{quoted prop ABC}, with the term $\funcW{x}$ of that result replaced by a (potentially) larger quantity $\epsilon_1$. (Indeed, an easy calculation shows that the statement actually follows from Proposition~\ref{quoted prop ABC} for $29\le q\le4\cdot 10^5$, upon noting that the computations of Platt~\cite{Pla2} confirm that all Dirichlet $L$-functions to these moduli satisfy \GRH{1}.) We prove Proposition~\ref{quoted prop ABC2} at the end of Section~\ref{Sec6.2}; we remark that our argument is similar to one of Ford, Luca, and Moree~\cite[Lemma 9]{FLM}. Once this proposition is established, we will use it to deduce our upper bounds on the error terms for our prime counting functions for these large moduli, thus completing the proof of Theorems~\ref{main psi theorem}--\ref{main pi theorem}.

\subsection{Explicit upper bound for exceptional zeros of quadratic Dirichlet $L$-functions}  \label{Sec6.1}

To proceed without the assumption of GRH(1), we need to derive estimates for zeros of $L$-functions that would potentially violate this hypothesis. Motivated by the computations of Platt~\cite{Pla2}, we will prove our results for $q \geq 4 \cdot 10^5$ though,
by direct computation, we can extend these to smaller values of $q$.

\begin{lemma}  \label{L1chi lemma}
If $\chi^*$ is a primitive quadratic character with modulus $q \geq 4 \cdot 10^5$, then
\begin{multline*}
  L(1,\chi^*) \ge  \min \left\{ 46 \pi, \max \big\{  \log \big( \tfrac{\sqrt{q+4} + \sqrt{q}}{2} \big), 12 \big\} \right\} q^{-1/2} \\
  = \begin{cases}
12q^{-1/2}, &\text{if } 4 \cdot 10^5 \le q < e^{24}-2, \\
\frac12 q^{-1/2}\log q, &\text{if } e^{24}-2 < q < e^{92\pi}-2, \\
46\pi q^{-1/2}, &\text{if } q > e^{92\pi}-2.
\end{cases}
\end{multline*}
\end{lemma}

\noindent Proposition~\ref{L1chi prop} is an easy consequence of this lemma; see Section~\ref{magma sec} for the details of that deduction.

\begin{proof}
As the asserted equality is elementary, we focus upon the asserted inequality. We use the fact~\cite[Theorem 9.13]{MV} that every primitive quadratic character can be expressed, using the Kronecker symbol, in the form $\chi^*(n) = \chi_d(n) = ( \frac dn )$ for some fundamental discriminant $d$, and such a character is a primitive character\mod{q} for $q=|d|$.

First, we consider negative values of $d$, so that $d\le-400000$.
For these characters, Dirichlet's class number formula~\cite[equation (4.36)]{MV} gives
\[
L(1,\chi_d) = \frac{2\pi h(\sqrt d)}{w_d\sqrt{|d|}},
\]
where $h(\sqrt d)$ is the class number of $\mathbb{Q}(\sqrt d)$, while $w_d$ is the number of roots of unity in $\mathbb{Q}(\sqrt d)$; as is well-known, we have $w_d=2$ for $d<-3$.
Appealing to Watkins \cite[Table 4]{Wat}, since $|d|=q > 319867$,
we may conclude that $h(\sqrt{-q}) \geq 46$, and hence that
$$
L(1,\chi^*) = \frac{2\pi h(\sqrt d)}{w_d\sqrt{|d|}} \ge  46 \pi q^{-1/2}.
$$

Now, we consider $d>0$.
For these characters, Dirichlet's class number formula~\cite[equation (4.35)]{MV} gives
\[
L(1,\chi_d) = \frac{h(\sqrt d)\log \eta_d}{\sqrt d},
\]
where $h(\sqrt d)$ is the class number as above; here
$\eta_d=(v_0+u_0 \sqrt{d})/2$, where $v_0$ and $u_0$ are the minimal positive integers satisfying $v_0^2-du_0^2=4$.
Since $h(\sqrt d)\ge1$ and
$$
\eta_d = \frac{v_0 + u_0 \sqrt{d}}{2}  \geq \frac{\sqrt{d+4} + \sqrt{d} }{2},
$$
we thus have that
$$
L(1,\chi^*) \ge \log \left( \tfrac{\sqrt{q+4} + \sqrt{q}}{2} \right) q^{-1/2}.
$$

It only remains to show that $L(1,\chi^*)\geq 12 q^{-1/2}$, assuming $q=d \geq 4\cdot 10^5$. As
$ \log \big( \tfrac{\sqrt{q+4} + \sqrt{q}}{2} \big) \geq 12$
for $q \geq 2.65 \cdot 10^{10}>e^{24}-2$, we may further assume that $4 \cdot 10^5 \leq q < 2.65 \cdot 10^{10}$.
In this range, we can verify the inequality
$$
h(\sqrt d)\log \eta_d > 12
$$
computationally (see Section~\ref{magma sec} for the details), which completes the proof of the lemma.
\end{proof}

It is worth noting that work of Oesterl\'e~\cite{Oes}, making explicit an argument of Goldfeld~\cite{Gol}, provides a lower bound upon class numbers of imaginary quadratic fields, which can be used to improve the order of magnitude of our lower bound for $L(1,\chi^*)$ in Lemma \ref{L1chi lemma}.
Tracing the argument through explicitly, for $d<0$ a fundamental discriminant, we could show that
\begin{equation}  \label{oe o1}
h(\sqrt{d}) > \log |d| \exp \bigg( {-} 10.4 \sqrt{\frac{\log\log|d|}{\log\log\log|d|}} \bigg),
\end{equation}
leading to an improvement in the lower bound of Lemma~\ref{L1chi lemma} of order $(\log q)^{1-o(1)}$ for large~$q$. Unfortunately, such an improvement would not ultimately lead to a more accessible range of $x$ in Theorems~\ref{main psi theorem}--\ref{main pi theorem} for large moduli.

\begin{lemma}  \label{FLM Lprime lemma with y}
Let $q\ge3$ be an integer, and let $\chi^*$ be a primitive character with modulus $q$. Then for any real number $\sigma$ satisfying $1-\frac1{4\sqrt q} \le \sigma \le 1$ and any $y>4$,
\begin{equation}  \label{first L1chi' bound}
|L'(\sigma,\chi^*)| \le y^{1-\sigma} \bigg( \frac{\log^2y}2 + \frac{1}{10} \bigg) + \frac{2\sqrt q}\pi \log\frac{4q}\pi \cdot \frac{\log y}{y^\sigma}.
\end{equation}
\end{lemma}

\begin{proof}
We proceed as in the proof of ~\cite[Lemma 3]{FLM}. We start by considering the incomplete character sum $f_{\chi^*}(u,v) = \sum_{u<n\le v} \chi^*(n)$, which can be bounded~\cite[Section 9.4, p. 307]{MV} by
\[
f_{\chi^*}(u,v) \le \frac2{\sqrt q} \sum_{a=1}^{(q-1)/2} \frac1{\sin\pi a/q}.
\]
Since the function $1/{\sin (\pi z/q)}$ is convex for $0\le z\le \frac q2$,
$$
\frac1{\sin\pi a/q} < \int_{a-1/2}^{a+1/2} \frac{dz}{\sin\pi z/q}
$$
for each $1\le a\le(q-1)/2$, and therefore
\[
f_{\chi^*}(u,v) \le \frac2{\sqrt q} \int_{1/2}^{q/2} \frac{dz}{\sin\pi z/q} = \frac{2\sqrt q}\pi \log\cot\frac\pi{4q} < \frac{2\sqrt q}\pi \log\frac{4q}\pi,
\]
since $\tan z>z$ for $0<z<\frac\pi2$. We note that while this simple bound (an explicit version of the P\'olya--Vinogradov inequality) is sufficient for our purposes, it is possible to
sharpen it further (see~\cite{Pom,FS}).


Now for any $y>4$,
\begin{align}
|L'(\sigma,\chi^*)| & = \bigg| {-} \sum_{n\le y} \frac{\chi(n)\log n}{n^\sigma} - \sum_{n>y} \frac{\chi(n)\log n}{n^\sigma} \bigg| \notag \\
&\le \sum_{n\le y} \frac{\log n}{n^\sigma} + \bigg| \sum_{n>y} \frac{\chi(n)\log n}{n^\sigma} \bigg| \notag \\
  & \le y^{1-\sigma} \sum_{n\le y} \frac{\log n}n + \bigg| \int_y^\infty \frac{\log z}{z^\sigma} \,df_{\chi^*}(y,z) \bigg|.
\label{Lprime split}
\end{align}
Since $\frac{\log z}z$ is decreasing for $z\ge 4$, the first term in expression (\ref{Lprime split})  can be bounded by
\begin{align}
y^{1-\sigma} \sum_{n\le y} \frac{\log n}n &\le y^{1-\sigma} \bigg( \frac{\log2}2 + \frac{\log3}3
+ \frac{\log 4}{4}+ \int_4^y \frac{\log z}z\,dz \bigg) \notag \\
&= y^{1-\sigma} \bigg( \log 2 + \frac{\log3}3 + \frac{\log^2y}2 - \frac{\log^24}2 \bigg) \notag \\
& < y^{1-\sigma} \bigg( \frac{\log^2y}2 + \frac{1}{10} \bigg).
\label{Lprime head of sum}
\end{align}
The second term in expression (\ref{Lprime split}), after integrating by parts (and noting that both boundary terms vanish), becomes
\begin{align*}
\bigg| \int_y^\infty \frac{\log z}{z^\sigma} \,df_{\chi^*}(y,z) \bigg| &= \bigg| {-} \int_y^\infty f_{\chi^*}(y,z) \bigg( \frac d{dz} \frac{\log z}{z^\sigma} \bigg) \,dz \bigg| \\
&\le \frac{2\sqrt q}\pi \log\frac{4q}\pi \int_y^\infty \bigg| \frac d{dz} \frac{\log z}{z^\sigma} \bigg| \,dz = \frac{2\sqrt q}\pi \log\frac{4q}\pi \cdot \frac{\log y}{y^\sigma},
\end{align*}
since $\frac{\log z}{z^\sigma}$ is a decreasing function of $z$ for $z > e^{1/\sigma}$ and since
$$
e^{1/(1-1/4\sqrt q)} < e^{\frac{4 \sqrt{3}}{4 \sqrt{3}-1}} < 4 < y.
$$
Combining this with inequalities  \eqref{Lprime split} and \eqref{Lprime head of sum} establishes the lemma.
\end{proof}

\begin{lemma}  \label{FLM Lprime lemma}
Let $q\ge4\cdot10^5$ be an integer and let $\chi^*$ be a primitive character with modulus $q$. Then, for any real number $\sigma$ satisfying $1-\frac1{4 \sqrt q} \le \sigma \le 1$,
\[
|L'(\sigma,\chi^*)| < 0.27356 \log^2 q.
\]
\end{lemma}

\begin{proof}
The upper bound on $|L'(\sigma,\chi^*)|$ in Lemma~\ref{FLM Lprime lemma with y} has a factor of $\frac{1}{y^\sigma}$ and otherwise does not depend on $\sigma$, so it suffices to establish the lemma for $\sigma=1-\frac1{4 \sqrt q}$ itself. Setting $y=q^\alpha$ with $\alpha$ to be determined numerically later, the bound~\eqref{first L1chi' bound} becomes
\begin{equation}  \label{2nd L1chi' bound}
\frac{|L'(\sigma,\chi^*)|}{\log^2 q} \leq q^{\left(\tfrac{\alpha}{4 \sqrt{q}}\right)} \cdot
         \left( \frac{\alpha^2}{2}+\frac{2 \alpha \log (4q/\pi)}{\pi  q^{\alpha - \frac 12 }  \log q}+\frac{1}{10 \log ^2q} \right),
\end{equation}
which for every fixed $\alpha>1/2$ is a decreasing function for sufficiently large~$q$. After some numerical experimentation we choose $\alpha=0.655$, for which the right-hand side of equation~\eqref{2nd L1chi' bound} is decreasing for $q\ge3$ (as is straightforward to check using calculus) and evaluates to less than $0.27356$ at $q= 4 \cdot 10^5$.
\end{proof}

\begin{proof}[Proof of Proposition~\ref{FLM beta prop}]
If $q\le4\cdot10^5$, Platt's computations confirm that no quadratic character modulo $q$ has a nontrivial real zero, and so the lemma is vacuously true for these moduli~$q$. Assume now that $q>4\cdot10^5$ and that $0<\beta<1$ is a nontrivial real zero.

We first establish the result under the additional assumption that $\chi$ is a primitive character. Since
$$
\min \left\{ 46 \pi, \max \left\{  \log \left( \frac{1}{2} \left( \sqrt{q+4} + \sqrt{q} \right) \right), 12 \right\} \right\} \geq 12,
$$
 and $q > 4 \cdot 10^5$,  Lemma~\ref{L1chi lemma} implies that
\begin{equation}  \label{L and L'}
12 q^{-1/2} < L(1,\chi) = L(1,\chi)-L(\beta,\chi) = (1-\beta)L'(\sigma,\chi)
\end{equation}
for some $\beta\le\sigma\le1$ by the Mean Value Theorem. If $\beta < 1-\frac1{4\sqrt q}$, then the bound $q\ge4\cdot10^5$ implies that $\beta \le 1 - \frac{40}{\sqrt q\log^2 q}$ as well. On the other hand, if $\beta \ge 1-\frac1{4\sqrt q}$, then Lemma~\ref{FLM Lprime lemma} and equation~\eqref{L and L'} imply
$$
1-\beta \ge \frac{12 q^{-1/2}}{L'(\sigma,\chi)} \ge \frac{12 q^{-1/2}}{0.27356\log^2q}
> \frac{40}{\sqrt{q} \log^2 q}.
$$

This argument establishes the proposition when $\chi$ is primitive. However, if $\chi\mod q$ is induced by some quadratic character $\chi^*\mod{q^*}$, then the primitive case already established yields
\[
\beta \le 1 - \frac{40}{\sqrt {q^*}\log^2 q^*} < 1 - \frac{40}{\sqrt q\log^2 q},
\]
as required.
\end{proof}

Note that an appeal to Oesterl\'e's work~\cite{Oes}, as discussed before equation~\eqref{oe o1}, would enable us to improve the denominator on the right-hand side of our upper bound for $\beta$ in Proposition~\ref{FLM beta prop} from
$\sqrt q\log^2 q$ to a complicated (yet still explicit) function of the form $\sqrt q (\log q)^{1+o(1)}$.
The strongest such theoretical bound known, due to Haneke~\cite{Han}, would have $\sqrt q\log q$ in the denominator.

\subsection{An upper bound for $|\psi(x;q,a) - x/\phi(q)|$, including the contribution from a possible exceptional zero}  \label{Sec6.2}

Now that we have an explicit upper bound for possible exceptional zeros, we can modify McCurley's arguments from~\cite{Mc1} to obtain the upper bound for $|\psi(x;q,a) - x/\phi(q)|$ asserted in Proposition~\ref{quoted prop ABC2}. In what follows, we will assume that $q \geq 10^5$; our methods would allow us to relax this assumption, if desired, with a change in the constants we obtain but no significant difficulties.

\begin{definition}  \label{mchi bchi def}
Let us define, as in \cite[page 271, lines 9--11]{Mc1},  $b(\chi)$ to be the constant term in the Laurent expansion of $\frac{L'}L(s,\chi)$ at $s=0$ and $m(\chi)$ (a nonnegative integer) to be the order of the zero of $L(s,\chi)$ at $s=0$, so that $\frac{L'}L(s,\chi) = \frac{m(\chi)}s + b(\chi) + O(|s|)$ near $s=0$.

If $\chi$ is principal, then $L(s,\chi) = \zeta(s) \prod_{p\mid q} (1-p^{-s})$, where the first factor $\zeta(s)$ is nonzero at $s=0$ while each factor in the product has a simple zero there; the multiplicity of the zero at $s=0$ is therefore $\omega(q)$, the number of distinct primes dividing~$q$. On the other hand, if $\chi$ is nonprincipal, then it is induced by some primitive character $\chi^*\mod{q^*}$ with $q^*>1$, and
\[
L(s,\chi) = L(s,\chi^*) \prod_{\substack{p\mid q \\ p\nmid q^*}} (1-\chi^*(p)p^{-s}),
\]
where the first factor $L(s,\chi^*)$ has at most a simple zero at $s=0$ while each factor in the product has a simple zero there; the multiplicity of the zero at $s=0$ is therefore at most $1+\omega(q)-\omega(q^*) \le \omega(q)$. In either case, we see that the order of the zero of $L(s,\chi)$ at $s=0$ is at most $\omega(q)$, and therefore
\begin{equation}  \label{mchi omegaq}
m(\chi)\le \omega(q)
\end{equation}
by the properties of logarithmic derivatives.
\end{definition}

Our immediate goal is to establish the upper bound for $|b(\chi)|$ asserted in Proposition~\ref{bchi prop}; we do so by adapting a method of McCurley to address the possible existence of exceptional zeros. Afterwards, we will be able to establish Proposition~\ref{quoted prop ABC2}.

\begin{lemma}  \label{isolate zeros below 1 lemma}
For any positive integer $q$ and any Dirichlet character $\chi\mod q$,
\begin{equation} \label{equ:rho exceptional}
\sum_{\substack{\rho \in \Zchi \\ |\gamma| \leq 1}} \frac{2}{|\rho(2-\rho)|} < \frac{\sqrt q \log^2 q}{40}+ 3.4596 \log ^2 q +12.938 \log q+7.3912.
\end{equation}
\end{lemma}

\begin{proof}
Since $|\rho|\ge\beta$ and $|2-\rho|\ge2-\beta$, it suffices to show that
\[
\sum_{\substack{\rho \in \Zchi \\ |\gamma| \leq 1}} \frac{2}{\beta(2-\beta)} < \frac{\sqrt q \log^2 q}{40}+ 3.4596 \log ^2 q +12.938 \log q+7.3912.
\]
We recall that Hypothesis~Z${}_1(9.645908801)$ is true~\cite[Theorem 1]{Mc1}, and therefore every zero $\rho$ being counted by the sum on the right-hand side, except possibly for a single exceptional zero $\beta_0$ and its companion $1-\beta_0$, satisfies
$$
\frac{1}{R_1 \log q} < \beta <1-\frac{1}{R_1 \log q}
$$
by Definition~\ref{hypothesis z} (where the lower bound holds by symmetry---see the remarks following equation~\eqref{Zchi def}). We will argue separately according to whether or not there are any exceptional zeros of $L(s,\chi)$, as per Definition~\ref{R1 exceptional def}.

We first assume that there is no such exceptional zero.
If $\beta=1/2$, then we have that ${2}/{\beta(2-\beta)} = 8/3$.
If $\beta\neq 1/2$, then we pair the two zeros $\rho_1 = \beta+i\gamma$ and $\rho_2=1-\beta+i\gamma$. Clearly one of $\beta$ and $1-\beta$ is less than 1/2 and the other greater, say $1-\beta<1/2<\beta$, whence
\begin{align}
  \frac{2}{\beta  (2-\beta ) }+ \frac{2}{(1-\beta)(2-(1-\beta))} = & \frac1{1-\beta} + \frac1{1+\beta} + \frac2{\beta(2-\beta)} \nonumber \\
  & < \frac{1}{1-\beta}+\frac23+\frac83   \label{10/3 eqn} \\ & < R_1\log q + \frac{10}3. \nonumber  \\ \nonumber
  \end{align}
In particular, the average contribution per zero is at most $\frac12
R_1\log q+\frac53$, whether the zero has real part $1/2$ or not
(recall that $R_1\approx 9.6$ and $q\geq
10^5$); thus
  \begin{equation}\label{equ:smallrho}
    \sum_{\substack{\rho\in\Zchi \\|\gamma| \leq 1}} \frac{2}{\beta(2-\beta)} \le \bigg(\frac{R_1}2 \log q+\frac53\bigg) N(1,\chi)
  \end{equation}
when there is no exceptional zero.

If, on the other hand, $L(s,\chi)$ has an exceptional zero $\beta_0$, then by definition
  \[0<1-\beta_0 \le \frac{1}{R_1\log q} <\frac 12 < 1- \frac{1}{R_1\log q}\le \beta_0<1;\]
furthermore, by Proposition~\ref{FLM beta prop},
  \[ \frac{40}{\sqrt q \log^2 q} \le 1-\beta_0.\]
By the same initial computation as in equation~\eqref{10/3 eqn},
  \[\frac{2}{\beta_0(2-\beta_0)} + \frac{2}{(1-\beta_0)(1+\beta_0)} < \frac{1}{1-\beta_0} + \frac{10}{3} \le \frac{\sqrt q \log^2 q}{40} + \frac{10}{3},\]
so that
  \begin{equation}\label{N-2 eqn}
    \sum_{\substack{\rho \in \Zchi \\ |\gamma| \leq 1}} \frac{2}{|\rho (2-\rho)|}
    < \frac{\sqrt q \log^2 q}{40}+ \frac{10}{3} + \left(\frac{R_1}2 \log q + \frac{5}{3}\right) (N(1,\chi)-2)
  \end{equation}
when there is an exceptional zero.
Proposition~\ref{quoting Trudgian} tells us that
 \begin{equation} \label{pickles}
 N(1,\chi) =  N(1,\chi^*) < \frac{1}{\pi}\log \frac{q^*}{2\pi e} + C_1 \log q^* + C_2 < 0.71731 \log q+4.4347
 \end{equation}
(since $q^*\le q$), and therefore the right-hand side of the inequality~\eqref{N-2 eqn} is larger than that of the inequality~\eqref{equ:smallrho}. The lemma now follows upon combining the inequalities~\eqref{N-2 eqn} and~\eqref{pickles} and rounding the constants upward.
\end{proof}

We remark that this proof shows that the first term on the right-hand side of the inequality~\eqref{equ:rho exceptional} can be replaced by the much smaller $2(0.71731 \log q+4.4347)$ if $L(s,\chi)$ has no exceptional zero.

\begin{proof}[Proof of Proposition~\ref{bchi prop}]
Our starting point is an inequality of McCurley~\cite[equation (3.16)]{Mc1}:
\begin{equation} \label{bee-chi}
 |b(\chi)| \leq \left| \frac{\zeta'(2)}{\zeta(2)} \right| + 1 + \sum_{\rho\in\Zchi} \frac{2}{|\rho (2-\rho)|} + \frac{q \log q}{4},
\end{equation}
where the sum runs over zeros of $L(s,\chi)$ in the critical strip. (We remark that an examination of McCurley's proof shows that the term $(q \log q)/{4}$ can be omitted if $\chi$ is primitive, as noted by Ramar\'e and Rumely~\cite[page 415]{RR}.)

For the zeros satisfying $|\gamma|>1$, McCurley \cite[page 275]{Mc1} finds that
\begin{equation}  \label{gamma>1 bound}
\sum_{\substack{\rho \in \Zchi \\ |\gamma| > 1}} \frac{2}{|\rho (2-\rho)|} < 4 \int_1^\infty \frac{N(t,\chi)}{t^3} dt.
\end{equation}
Since Proposition~\ref{quoting Trudgian} implies the inequality
\[
   N(t,\chi) < \frac{t}{\pi} \log \frac{q^* t}{2\pi e} +C_1 \log q^*t + C_2 \le \frac{t}{\pi} \log \frac{q t}{2\pi e} + C_1 \log qt + C_2,
\]
the bound~\eqref{gamma>1 bound} becomes
  \begin{align*}
    \sum_{\substack{\rho \in \Zchi \\ |\gamma| > 1}} \frac{2}{|\rho (2-\rho)|}
    &<4 \int_1^\infty \left(\frac{t}{\pi} \log \frac{q t}{2\pi e} + C_1 \log qt + C_2 \right) t^{-3} \,dt \notag \\
    &=4 \left( \frac{\log q-\log 2\pi }{\pi }+ C_1 \cdot \frac{ 2 \log q+1}{4}+C_2 \cdot \frac 12 \right) \\
    &< 2.0713 \log q+8.735.
  \end{align*}
Combining this bound with Lemma~\ref{isolate zeros below 1 lemma} yields
\begin{equation} \label{referee's comment?}
\sum_{\rho\in\Zchi} \frac{2}{|\rho(2-\rho)|} \leq \frac{\sqrt{q} \log^2 q}{40} + 3.4596 \log ^2 q+15.01 \log q+16.126.
\end{equation}
From equation~\eqref{bee-chi}, it follows that
$$
  |b(\chi)| \leq\left| \frac{\zeta'(2)}{\zeta(2)} \right| +  \frac{\sqrt{q} \log^2 q}{40} + 3.4596 \log ^2 q+15.01 \log q+17.126 + \frac{q \log q}{4}
$$
and hence
$$
|b(\chi)| <  0.2515 q \log q,
$$
where the last inequality is a consequence of the assumption that $q \geq 10^5$.
\end{proof}

\begin{proof}[Proof of Proposition~\ref{quoted prop ABC2}]
Arguing as in the proof of Theorem 3.6 of McCurley \cite{Mc1}, but without the assumption of GRH(1), one obtains the inequality~\eqref{yikes} with
\begin{equation} \label{wicked}
\epsilon_1 < \frac{\phi (q)}{x} \left( \frac{\log 2}{2} + |d_2| \log (2x) + |d_1+d_2| \right),
\end{equation}
where (as in McCurley~\cite[equations (3.4) and (3.5)]{Mc1})
\begin{align*}
d_1 = \frac1{\phi(q)} \sum_{\chi\mod q} \overline\chi(a) \big( m(\chi) - b(\chi) \big) \quad\text{and}\quad d_2 = -\frac1{\phi(q)} \sum_{\chi\mod q} \overline\chi(a) m(\chi),
\end{align*}
with $m(\chi)$ and $b(\chi)$ as in Definition~\ref{mchi bchi def}.
It follows that
\begin{equation}  \label{mchi average}
|d_2| \leq \frac1{\phi(q)} \sum_{\chi\mod q}  m(\chi)  \leq \frac1{\phi(q)} \sum_{\chi\mod q}  \omega (q) = \omega(q) \leq \frac{\log q}{\log 2}
\end{equation}
by equation~\eqref{mchi omegaq} and
\begin{equation}  \label{bchi average}
|d_1 + d_2| = \bigg| \frac1{\phi(q)} \sum_{\chi\mod q} \overline\chi(a) b(\chi) \bigg|
\leq  \frac1{\phi(q)} \sum_{\chi\mod q}  |b(\chi)| < 0.2515 q \log q
\end{equation}
by Proposition~\ref{bchi prop}.
Inserting the inequalities~\eqref{mchi average} and~\eqref{bchi average} into the upper bound~\eqref{wicked} results in
 $$
 \epsilon_1 < \frac{\phi (q)}{x} \left( \frac{\log 2}{2} + \frac{\log q}{\log 2} \log (2x) + 0.2515 q \log q \right).
 $$
It is easy to check that the assumption $q \ge 10^5$ implies
$$
\frac{\log 2}{2} + \frac{\log q}{\log 2} \log (2x) + 0.2515 q \log q < \frac{\log q}{\log 2} \log x + 0.2516 q \log q,
$$
which completes the proof of the proposition.
\end{proof}

\subsection{Explicit upper bounds for $|\psi(x;q,a) - x/\phi(q)|$ and $|\theta(x;q,a) - x/\phi(q)|$}  \label{sec53}

To apply Proposition \ref{quoted prop ABC2} for $q \geq 10^5$, we could argue carefully as in Sections \ref{Sec2} and \ref{Sec4} to bound the various quantities on the right-hand side of equation~\eqref{yikes}. Our inability to rule out the existence of possible exceptional zeros for $L$-functions of large modulus $q$ forces us to assume that the parameter $x$ is exceptionally large, however, making such a refined analysis somewhat unnecessary. Instead, we will simply set $m=2$ in Proposition \ref{quoted prop ABC2}, to take advantage of existing inequalities, and proceed from there over the next three lemmas to obtain an explicit upper bound for $|\psi(x;q,a) - x/\phi(q)|$. Afterwards, we will convert that upper bound to a simpler error estimate (for both $\psi(x;q,a)$ and $\theta(x;q,a)$) that is a multiple of $x/(\log x)^Z$ for an arbitrary $Z>0$.

Define the quantities
\begin{equation}  \label{XalphaH def}
X = \sqrt{\frac{\log x}{R_1}}, \quad \alpha = \frac X{\log q}-1, \quad \mbox{ and } H = q^{\alpha} = \frac{e^X}q,
\end{equation}
and recall that $R_1 = 9.645908801$ as in Definition~\ref{R1 exceptional def}.

\begin{lemma}  \label{one rho down lemma}
Let $q\ge10^5$ be an integer, and let $\chi$ be a character\mod q. For $x\ge e^{4  R_1\log^2 q}$,
\[
\sum_{\substack{\rho\in\Zchi \\ \rho \neq \beta_0\\ |\gamma|\le H} \\ } \frac{x^{\beta-1}}{|\rho|} < 0.5001 X e^{-X},
\]
where the index of summation means that an exceptional zero $\beta_0$ for $L(s,\chi)$, if it exists, is excluded.
\end{lemma}

\begin{proof}
We first compute the given sum with the symmetric zero $1-\beta_0$ also excluded.
Combining the proof of \cite[Lemma 3.7]{Mc1} with Proposition \ref{quoting Trudgian}, for each character $\chi$ modulo $q$ we have
$$
\sum_{\substack{\rho\in\Zchi \\ \rho \notin \{\beta_0,1-\beta_0\} \\ |\gamma|\le H} \\ } \frac{x^{\beta-1}}{|\rho|} < \epsilon_2+\epsilon_3+\epsilon_4,
$$
where
\begin{align*}
\epsilon_2 &< \frac{q \log q + \alpha \log^2 q}{x} + \frac{1}{2\sqrt{x}} \left( \frac{1+4 \alpha+\alpha^2}{2 \pi} \log^2 q + \frac{2+\alpha}{\pi} \log q \right.\\
& \; \; \; \; \; \; \; \hskip16ex  \left.  + \frac{C_1  (\alpha+1) \log(q) +C_2}{q^\alpha} + 0.798 \log(q)+11.075 \right) \\
\end{align*}
and
\begin{align*}
\epsilon_3 &= \frac{C_1 X + C_2}{q^\alpha} e^{-X} \\
\epsilon_4 &= \frac{1}{2} \int_1^{q^\alpha} t^{-1} e^{-\frac{\log x}{R_1 \log (qt)}} \log (qt/2\pi) \,dt = \frac{1}{2} \int_1^{q^\alpha} t^{-1} e^{-X^2/\log (qt)} \log (qt/2\pi) \,dt.
\end{align*}
Since $x = e^{(1+\alpha)^2 R_1 \log^2 q}$, $q \geq 10^5$ and $\alpha \geq 1$, straightforward calculus exercises yield
$$
\epsilon_2 < 10^{-1000} X e^{-X} \; \; \mbox{ and } \; \; \epsilon_3 < 10^{-5} X e^{-X},
$$
while the change of variables $u = -X^2/\log (qt)$ (as in~\cite[page 1473]{FLM}) gives an upper bound upon $\epsilon_4$ of the shape
$$
\frac{1}{2} \int_1^{q^\alpha} e^{-X^2/\log (qt)} \log qt \,\frac{dt}t = \frac{X^4}{2} \int_{X}^{X^2/\log q} \frac{e^{-u}}{u^3} \, du < \frac {X^4}2 \int_{X}^\infty \frac{e^{-u}}{X^3} \, du = \frac{Xe^{-X}}2.
$$
We thus have
$$
\epsilon_2+\epsilon_3+\epsilon_4 < 0.50005 X e^{-X}.
$$

As for the special zero $1-\beta_0$ (when it exists), the bounds
$$
\beta_0\ge 1-1/R_1\log q \ge 0.99
$$
 from Definition~\ref{R1 exceptional def} and $q\ge10^5$ and $\beta_0 \le 1 - 40/\sqrt q\log^3 q$ from Proposition~\ref{FLM beta prop}, together with the hypothesis $x \ge e^{4R_1\log^2q}$ which is equivalent to $\log q \le X/2$, imply
\[
\frac{x^{(1-\beta_0)-1}}{1-\beta_0} \le \frac{\sqrt q\log^3 q}{40} x^{-0.99} \le \frac{X^3e^{X/4}}{320x^{0.99}} < 10^{-1000} X e^{-X}
\]
via another straightforward calculus exercise. Therefore the entire sum is at most $0.50005 X e^{-X}+10^{-1000} X e^{-X}<0.5001 X e^{-X}$ as required.
\end{proof}

\begin{lemma}  \label{three rhos down lemma}
Let $q\ge10^5$ be an integer, and let $\chi$ be a character\mod q. For $x\ge e^{4  R_1\log^2 q}$,
\[
\sum_{\substack{\rho\in\Zchi \\ |\gamma|>H}} \frac{x^{\beta-1}}{|\rho(\rho+1)(\rho+2)|} < 0.511 X e^{-X} q^{-2 \alpha}
\]
where $H$, $X$, and $\alpha$ are defined in equation~\eqref{XalphaH def}.
\end{lemma}

\begin{proof}
As in the proof of Lemma~\ref{one rho down lemma}, we combine Proposition~\ref{quoting Trudgian} with the proof of \cite[Lemma 3.8]{Mc1}; for each character $\chi$ modulo $q$ we have
$$
\sum_{\substack{\rho\in\Zchi \\ |\gamma|>H}} \frac{x^{\beta-1}}{|\rho(\rho+1)(\rho+2)|} <
\epsilon_5 + \epsilon_6+\epsilon_7,
$$
where
\begin{align*}
\epsilon_5 &< \frac{1}{2 q^{3 \alpha} \sqrt{x}} \left( \frac{q^\alpha}{2 \pi} (1+\alpha) \log q + 0.798  (\alpha+1) \log(q) +10.809 \right)
+ \frac{4 \log q}{x q^{2 \alpha}} \\
\epsilon_6 &= \frac{C_1}{2} \int_{q^\alpha}^{\infty}  t^{-4} e^{-\frac{\log x}{R_1 \log (qt)}} \,dt
+ \frac{1}{2} \int_{q^\alpha}^{\infty}  t^{-3} e^{-\frac{\log x}{R_1 \log (qt)}} \log (qt/2\pi) \,dt \\
\epsilon_7 &= \frac{C_1 X + C_2}{q^{3 \alpha}} e^{-X}.
\end{align*}
Again via calculus, it is routine to show that
$$
\epsilon_5 < 10^{-1000} X e^{-X} q^{-2 \alpha} \; \;
\mbox{ and } \; \;
\epsilon_7 < 0.00001 X e^{-X} q^{-2 \alpha}.
$$
To estimate $\epsilon_6$, note that
$$
\epsilon_6 <  \frac{1}{2} \int_{q^\alpha}^{\infty}  t^{-3} e^{-\frac{\log x}{R_1 \log (qt)}} \log (qt) dt = \frac{1}{2} I_{2,2} \left( (1+\alpha)^2 \log^2 q, q; q^\alpha \right),
$$
in the notation of Definition \ref{Inuk def}. Applying Lemma \ref{I to K cov lemma}, we have
$$
I_{2,2} \left( (1+\alpha)^2 \log^2 q, q; q^\alpha \right) = (1+\alpha)^2 q^2 (\log q)^2 K_2 \left( 2 \sqrt{2} (1+ \alpha) \log q; \sqrt{2} \right)
$$
and so
$$
\epsilon_6 <  \frac{1}{2} (1+\alpha)^2 q^2 (\log q)^2 K_2 \left( 2 \sqrt{2} (1+ \alpha) \log q; \sqrt{2} \right).
$$
Work of Rosser--Schoenfeld \cite[Lemmas~4 and~5]{RS2} yields
$$
\epsilon_6 <  \frac{1}{2} q^2 \left( X + \frac{1}{2} \right) e^{-3X} =  \frac{1}{2} \left( 1 + \frac{1}{2X} \right) \left( X e^{-X} q^{-2 \alpha} \right)
<  0.5109 X e^{-X} q^{-2 \alpha}.
$$
It follows that
$\epsilon_5+\epsilon_6+\epsilon_7 < 0.511 X e^{-X} q^{-2 \alpha}$ as required.
\end{proof}

\begin{lemma}  \label{FLM lemma}
For $q\ge10^5$ and $x\ge e^{4  R_1\log^2 q}$,
\[
\bigg| \psi(x;q,a) - \frac x{\phi(q)} \bigg| \le \frac{1.012}{\phi(q)} x^{\beta_0} + 1.4579 x \sqrt{\frac{\log x}{R_1}} \exp \bigg( {-} \sqrt{\frac{\log x}{R_1}} \bigg),
\]
where the first term on the right-hand side is present only if some Dirichlet $L$-function\mod q has an exceptional zero~$\beta_0$ (in the sense of Definition~\ref{R1 exceptional def}).
\end{lemma}

\begin{proof}
Recall the definitions of $\alpha$, $H$, and $X$ in equation~\eqref{XalphaH def}, and note that $\alpha \geq 1$ due to our hypothesis on~$x$. Applying Proposition \ref{quoted prop ABC2} with $m=2$ and $\delta = \frac2H \leq 2 \cdot 10^{-5}$, we
have an upper bound for $\big| \psi(x;q,a) - \frac x{\phi(q)} \big|$ of the shape
\begin{equation} \label{cliff}
\frac{x}{\phi (q)} \left( U_{q,2} \left( x; \frac{2}{q^\alpha}, q^\alpha \right) + V_{q,2} \left( x; \frac{2}{q^\alpha}, q^\alpha \right) +  \frac{2}{q^\alpha} \right) + \frac{\log q \log x}{\log 2} + 0.2516 q \log q.
\end{equation}
Here,
\begin{align*}
U_{q,2} \left( x; \frac{2}{q^\alpha}, q^\alpha \right) &= A_2(\delta) \sum_{\chi\mod q} \sum_{\substack{\rho\in\Zchi \\ |\gamma|>H}} \frac{x^{\beta-1}}{|\rho(\rho+1)(\rho+2)|} \\
&= \bigg( H^2 + 6H + 18 + \frac{20}H \bigg) \sum_{\chi\mod q} \sum_{\substack{\rho\in\Zchi \\ |\gamma|>H}} \frac{x^{\beta-1}}{|\rho(\rho+1)(\rho+2)|} \\
&< 1.001 q^{2 \alpha} \phi(q) \cdot 0.511 X e^{-X} q^{-2 \alpha} < 0.512 \phi(q) X e^{-X}
\end{align*}
by Lemma~\ref{three rhos down lemma} and a simple calculation,
while
$$
V_{q,2} \left( x; \frac{2}{q^\alpha}, q^\alpha \right)= \bigg(1+\frac2H \bigg) \sum_{\chi\mod q} \sum_{\substack{\rho\in\Zchi \\ |\gamma|\le H}} \frac{x^{\beta-1}}{|\rho|}.
$$
It follows that
$$
V_{q,2} \left( x; \frac{2}{q^\alpha}, q^\alpha \right) \le \frac{(1+2 q^{-\alpha})x^{\beta_0-1} }{\beta_0} + (1+2 q^{-\alpha})\phi(q) \cdot 0.5001 X e^{-X}
$$
by Lemma~\ref{one rho down lemma}, where the first term is present only if some Dirichlet $L$-function\mod q has an exceptional zero.

We may thus conclude from expression~\eqref{cliff} that $\big| \psi(x;q,a) - \frac x{\phi(q)} \big|$ is bounded above by
\begin{flalign*}
& \frac{(1+2 q^{-\alpha})x^{\beta_0} }{\phi (q) \beta_0}  +0.5001 x (1+2 q^{-\alpha}) X e^{-X} \\
& + 0.512 x X e^{-X}  + \frac{2x}{\phi (q) q^\alpha} + \frac{\log q \log x}{\log 2} + 0.2516 q \log q,
\end{flalign*}
where we may omit the first term if no exceptional zero $\beta_0$ exists. From $x = e^{(1+\alpha)^2 R_1 \log^2 q}$ and   $\alpha \geq 1$, we may verify by explicit computation for $10^5 \leq q < 3 \cdot 10^5$ that
\begin{equation} \label{whacha doin}
0.5001(1+2 q^{-\alpha}) + 0.512 + \frac{2e^X}{\phi (q) q^\alpha X} + \frac{e^X \log q \log x}x X {\log 2} + \frac{0.2516 e^X q \log q}{xX} ,
\end{equation}
is at most $1.4579$ (and in fact
maximal for $\alpha=1$ and $q=120120$). For $q \geq 3 \cdot 10^5$, we appeal to
(\cite[Theorem 15]{RS}) which provides the inequality
$$
\frac{n}{\phi (n)} < e^\gamma \log \log n + \frac{2.50637}{\log \log n},
$$
and again conclude that inequality~\eqref{whacha doin} obtains.
Since $\beta_0 \ge 1 - {1}/{R_1\log q}$,
it thus follows that
$$
\bigg| \psi(x;q,a) - \frac x{\phi(q)} \bigg| < 1.012 \frac{x^{\beta_0} }{\phi (q)} + 1.4579 x X e^{-X},
$$
as desired.
\end{proof}

The next two easy lemmas will help us prepare the upper bound just established for simplication to the form we eventually want.

\begin{lemma} \label{psi-theta lemma}
Let $a$ and $q$ be integers with $q\ge3$ and  $\gcd (a,q)=1$.
Then, if $x\ge10^{500}$,
$$
|\psi(x;q,a) - \theta(x;q,a)| < 1.001\sqrt x \quad\text{and}\quad |\psi(x;q,a) - \theta_\#(x;q,a)| < 1.001\sqrt x,
$$
where $\theta_\#(x;q,a)$ is defined in equation~\eqref{thetaH def}.
\end{lemma}

\begin{proof}
We will use Rosser-Schoenfeld~\cite[Theorem 4, page 70]{RS}: for all $y > 1$,
\[
\theta(y) < y + \frac y{2\log y}.
\]
Define $f(x) = x^{1/2} + \frac{x^{1/2}}{\log x} + \frac{x^{1/3}\log x}{\log 2} + \frac{3x^{1/3}}{2\log 2}$.
Even if we pretend that every proper prime power is congruent to $a\mod q$, we have
\begin{align*}
0 \le \psi(x;q,a) - \theta(x;q,a) &\le \sum_{k=2}^{\lfloor \log x/\log 2 \rfloor} \theta(x^{1/k}) \\
&\le \theta(x^{1/2}) + \theta(x^{1/3}) \frac{\log x}{\log 2} \\
&\le x^{1/2} + \frac{x^{1/2}}{\log x} + \bigg( x^{1/3} + \frac{3x^{1/3}}{2\log x} \bigg) \frac{\log x}{\log 2} = f(x).
\end{align*}
Recall that $\xi_2(q,a)$ is defined in Definition~\ref{xi defs}; trivially from this definition, we have the inequality $\xi_2(q,a)\leq \phi(q)$, and therefore $\xi_2(q,a)\sqrt x/\phi(q)\le \sqrt x$. Therefore
\[
-f(x) < -\sqrt x \le \psi(x;q,a) - \bigg( \theta(x;q,a) + \frac{\xi_2(q,a) \sqrt x}{\phi(q)} \bigg) \le f(x).
\]
It follows that both $|\psi(x;q,a) - \theta(x;q,a)|$ and $|\psi(x;q,a) - \psi(x;q,a) - \theta_\#(x;q,a) |$ are bounded by $f(x)$. It is easily checked that the decreasing function $f(x)/\sqrt x$ is less than $1.001$ when $x\ge10^{500}$.
\end{proof}

\begin{lemma}  \label{two-prong lemma}
For $q\ge10^5$ and $x\ge e^{4R_1\log^2 q}$,
$$
\bigg| \psi(x;q,a) - \frac x{\phi(q)} \bigg| \le \frac{1.012}{\phi(q)} x^{1-40/(\sqrt q\log^2 q)} + 1.4579 x \sqrt{\frac{\log x}{R_1}} \exp \bigg( {-} \sqrt{\frac{\log x}{R_1}} \bigg)
$$
and
\begin{flalign*}
\bigg| \theta(x;q,a) - \frac x{\phi(q)} \bigg| & \le \frac{1.012}{\phi(q)} x^{1-40/(\sqrt q\log^2 q)} \\
& + 1.4579 x \sqrt{\frac{\log x}{R_1}} \exp \bigg( {-} \sqrt{\frac{\log x}{R_1}} \bigg) + 1.001\sqrt x, \\
\end{flalign*}
where the first term on each right-hand side is present only if an exceptional zero exists for a quadratic $L$-function with conductor~$q$.
\end{lemma}

\begin{proof}
We simply combine Proposition~\ref{FLM beta prop} with Lemmas~\ref{FLM lemma} and~\ref{psi-theta lemma} (and note that $e^{{4 R_1}(\log 10^5)^2} > 10^{500}$).
\end{proof}

The bounds of Lemma~\ref{two-prong lemma} are both $O \big({x}/{(\log x)^Z} \big)$ for every fixed real number~$Z$. The purpose of this subsection is to provide several explicit versions of this observation. The first summand in the bounds, with its unfortunate dependence on $q$, is the one that really drives the growth. For that term, we need to take $x$ extremely large before the asymptotic behavior is seen, rendering the resulting bounds on $\psi(x;q,a)$, $\theta(x;q,a)$, and $\pi(x;q,a)$ impractical, although explicit. Consequently, we bound all three summands rather carelessly.

\begin{lemma}  \label{first part of three lemma}
Let $q\geq 10^5$ be an integer and $Z$ a real number, and let $\kappa_1\ge 0.0132$ be a real number satisfying
\[
\frac{460.516 \kappa_1}{\log \kappa_1+13.087} \ge Z.
\]
Then for all $x\ge\exp\big(\kappa_1\sqrt q\log^3q\big)$,
\[
\frac{1.012}{\phi(q)} x^{1-40/(\sqrt q\log^2 q)} \le 10^{-4} \frac x{(\log x)^Z}.
\]
\end{lemma}

\begin{proof}
By taking logarithmic derivatives, it is easy to show that the quotient
\[
\frac{\kappa_1 \log q}{\log(\kappa_1\sqrt q\log^3 q)}
\]
is an increasing function of $q$ for $q\ge \exp(e/\kappa_1^{1/3})$; in particular, since $\kappa_1\ge0.0132$, it is an increasing function for $q>10^5$. Therefore
\[
\frac{\kappa_1 \log q}{\log(\kappa_1\sqrt q\log^3 q)} \ge \frac{\kappa_1 \log 10^5}{\log\Big( \kappa_1\sqrt {10^5} \log^3(10^5) \Big)}
\]
and thus
\begin{align*}
\frac{\kappa_1 \sqrt q\log^3 q}{\log(\kappa_1\sqrt q\log^3 q)} & \ge \frac{5\kappa_1 \log 10}{\log \kappa_1+\log(10^{5/2}(\log10^5)^3)} \sqrt q\log^2q \\
& > \frac{11.5129\kappa_1}{\log \kappa_1+13.087} \sqrt q\log^2q, \\
\end{align*}
for all $q\geq 10^5$. The function $(\log x)/\log\log x$ is increasing for $\log x\ge e$; since the hypotheses of the lemma imply 
$$
\log x \ge \kappa_1\sqrt q\log^3q \ge 0.0132\sqrt{10^5}\log^3(10^5) > e, 
$$
we conclude that
\[
\frac{\log x}{\log\log x} \ge \frac{11.5129\kappa_1}{\log \kappa_1+13.087} \sqrt q\log^2q,
\]
and in particular
\[
\frac{40\log x}{\sqrt q\log^2q} \ge Z\log\log x
\]
given the assumption on~$Z$ (noting that $40\cdot11.5129 = 460.516$).
By \cite[Theorem 15]{RS}, for $q\geq 1.2\cdot10^5$, we have $\phi(q) \geq 20736$, and by direct computation of $\phi$ we extend this bound down to $q\geq 10^5$.
This implies that
$$
\frac{40\log x}{\sqrt q\log^2q} + \log \phi(q) \ge \log 20736 + Z \log\log x,
$$
$$
\phi(q)x^{40/(\sqrt q\log^2q)} \ge 20736 (\log x)^Z
$$
and
$$
\frac x{(\log x)^Z} \ge \frac{20736}{\phi(q)} x^{1-40/(\sqrt q\log^2q)} \ge 10^4 \frac{1.012}{\phi(q)} x^{1-40/(\sqrt q\log^2q)},
$$
as desired.
\end{proof}

\begin{lemma}  \label{second part of three lemma}
Suppose that $R, \kappa_2$ and $Z$ are real numbers with $1 \leq R \leq 10$,  $\kappa_2 > 1$ and
$$
Z \leq \frac{\sqrt{\kappa_2/R} +\log \left( {\sqrt{R_1}}/{7.2895} \right)}{\log \kappa_2} - \frac 12.
$$
Then for all $x \geq e^{\kappa_2}$,
$$
1.4579 x \sqrt{\frac{\log x}{R}} \exp \bigg( {-} \sqrt{\frac{\log x}{R}} \bigg) \le \frac1{5} \frac x{(\log x)^Z}.
$$
\end{lemma}

\begin{proof}
Consider for $u>1/\sqrt{R}$ the function
 $$
    f(u) = \frac{\log \left( {e^u}/{7.2895 u} \right) }{\log (Ru^2)},
 $$
 whose derivative satisfies
 $$
 \frac{df}{du} = \frac{(u-1) \log (R u^2) -2 \log \left( {e^u}/{7.2895 u} \right)}{u \log^2 (R u^2) }.
$$
The denominator of the derivative is clearly positive, and its numerator is continuous, goes to $\infty$ with $u$, has derivative $\log Ru^2>0$, and is positive for $u=1/\sqrt{R}$ (using that $1\leq R \leq 10$). Therefore, $f(u)$ is increasing.

By our hypothesis on $Z$, we have that $Z \leq f(\sqrt{\kappa_2/R})$. As $f$ is increasing, it follows that $Z\leq f(\sqrt{\log(x)/R})$ provided  $\log x \geq \kappa_2$ and $\sqrt{\log(x)/R} > 1/\sqrt{R}$, whence our hypotheses that $\kappa_2 > 1$ and $x\geq e^{\kappa_2}$. But $Z\leq f(u)$ is equivalent to
  \[\frac 15 \cdot \frac{1}{(Ru^2)^Z} \geq 1.4579 \frac{u}{e^u}.\]
The lemma follows upon setting $u=\sqrt{\log(x)/R}$ and multiplying both sides by $x$.
\end{proof}

\begin{lemma}  \label{third part of three lemma}
Let $\kappa_3$ and $Z$ be real numbers with $\kappa_3>1$ and
$$
Z \leq \frac{\kappa_3-6.44}{2\log \kappa_3}.
$$
Then for all $x\geq e^{\kappa_3}$,
\[
1.001\sqrt x \le \frac1{25} \frac x{(\log x)^Z}.
\]
\end{lemma}

\begin{proof}
Consider $f(u)=\frac{u-6.44}{2\log u}$ for $u>1$. Clearly $f$ is increasing and our hypothesis on $Z$ is that $Z\leq f(\kappa_3)$. Thus $Z\leq f(u)$ for all $u\geq \kappa_3$, and in particular $Z\leq f(\log x)$. But this is equivalent to
  \[1.001\sqrt x \leq \frac{1.001}{e^{3.22}} \frac{x}{(\log x)^Z},\]
and $1.001/e^{3.22} < 1/25$.
\end{proof}

With these three lemmas in place, we may now convert Lemma~\ref{two-prong lemma} into an explicit upper bound for the error terms related to $\psi(x;q,a)$ and $\theta(x;q,a)$.

\begin{prop}  \label{gory inequality prop}
Let $q\geq 10^5$ be an integer and $Z, \kappa_1 \ge 0.0132, \kappa_2>1, \kappa_3>1$ be real numbers satisfying
\begin{equation} \label{licorice}
Z\leq \min \left\{ \frac{460.516  \kappa_1}{\log \kappa_1+13.087}, \;  \frac{\sqrt{\kappa_2/{R_1}} -0.85317}{\log \kappa_2} - \frac12, \; \frac{\kappa_3-6.44}{2\log\kappa_3} \right\},
\end{equation}
for $R_1$ as defined in Definition \ref{R1 exceptional def}.
Then for all $x\ge \exp \big( \max\{\kappa_1\sqrt q\log^3q, \kappa_2, \kappa_3\} \big)$,
$$
\bigg|\psi(x;q,a) - \frac x{\phi(q)} \bigg| \le \frac 14 \frac x{(\log x)^Z}
\quad\text{and}\quad
\bigg|\theta(x;q,a) - \frac x{\phi(q)} \bigg| \le \frac 14 \frac x{(\log x)^Z}.
$$
\end{prop}

\begin{proof}
To apply Lemma~\ref{two-prong lemma}, we need $x\geq 4 R_1\log^2 q$, and here we have the stronger assumptions that $q\geq 10^5$ and $x \geq \kappa_1\sqrt q\log^3q$. Now, using Lemmas~\ref{first part of three lemma}--\ref{third part of three lemma} (choosing $R=R_1$ in Lemma~\ref{second part of three lemma}, and using the fact that
$\log \left( {\sqrt{R_1}}/{7.2895} \right) > -0.85317$) shows that
\begin{equation}  \label{2401}
\max \left\{ \bigg|\psi(x;q,a) - \frac x{\phi(q)} \bigg|, \bigg|\theta(x;q,a) - \frac x{\phi(q)} \bigg| \right\} \le \bigg( \frac15+\frac1{25}+10^{-4} \bigg) \frac x{(\log x)^Z},
\end{equation}
which suffices to establish the proposition.
\end{proof}

The following corollary completes the proof of Theorems~\ref{main psi theorem} and~\ref{main theta theorem} for large moduli $q > 10^5$, with $c_\psi(q) = c_\theta(q) = \frac1{160}$ and
$$
x_\psi(q) = x_\theta(q) = \exp \big( 0.03 \sqrt q \log^3q \big)
$$
(upon taking $A=1$).

\begin{cor}  \label{same result as for small q cor}
Let $q \geq 10^5$ be an integer and let $A$ be any real number with $1 \leq A \leq 8$. If $x$ is a real number satisfying $x\ge \exp \big( 0.03 A \sqrt q \log^3q \big)$, then
$$
\bigg|\psi(x;q,a) - \frac x{\phi(q)} \bigg| \le \frac1{160} \frac x{(\log x)^A}
\quad\text{and}\quad
\bigg|\theta(x;q,a) - \frac x{\phi(q)} \bigg| \le \frac1{160}  \frac x{(\log x)^A}.
$$
\end{cor}

It is worth observing that, appealing to the previously mentioned work of Oesterl\'e \cite{Oes}, we could improve the lower bound on $x$ here to $x \ge \exp(\kappa'\sqrt q(\log q)^{2+o(1)})$ for some $\kappa'>0$, where the $o(1)$ can be made explicit as in equation~\eqref{oe o1}.

\begin{proof}
Set $\kappa_1=0.03 A$, $\kappa_2=\kappa_3=14400 A$ and  $Z=A+0.4$. By calculus, the hypotheses of Proposition~\ref{gory inequality prop} are satisfied, for  $1 \leq A \leq 8$. Moreover, as $q\geq 10^5$,
$$
\kappa_1\sqrt q\log^3q \ge 0.03 A \sqrt{10^5}(\log10^5)^3 > 14400 A = \max\{\kappa_2, \kappa_3\},
$$
and therefore the conclusion of Proposition~\ref{gory inequality prop} holds for $x \ge \exp(\kappa_1\sqrt q\log^3q)$. Since $\log x > 14400 A$ in this range, we conclude that
$$
\frac 14 \frac{x}{(\log x)^Z} = \frac{1}{4} \frac{x}{(\log x)^A} \frac{1}{(\log x)^{Z-A}} < \frac{1}{4} \frac{1}{14400^{0.4}}  \frac{x}{(\log x)^A} < \frac{1}{160}  \frac{x}{(\log x)^A}.
$$
\end{proof}

Observe here that we were able to obtain a ``small'' constant factor of $1/160$ in Corollary \ref{same result as for small q cor},  by starting with a higher power of $\log x$ in the denominator of our error term than we ultimately desired. Arguing similarly, we can replace the constant $1/160$ with a function of the parameter $q$ that decreases to $0$ as $q$ increases, by starting again with extraneous powers of $\log x$ in the denominator of our error term, and using our assumption that $\log x \ge \kappa_1 \sqrt{q} \log^3 q$.

In a recent preprint of  Yamada~\cite[Theorem 1.2]{Yam}, one finds similar results of the shape
$$
\bigg|\psi(x;q,a) - \frac x{\phi(q)} \bigg| = O \left( \frac x{(\log x)^{A}} \right),
$$
for integers $1 \leq A \leq 10$, valid also for $\log x \gg \sqrt{q} \log^3 q$. Corollary \ref{same result as for small q cor}  is not directly comparable to Yamada's results, as the latter contain estimates that have been normalized to contain factors of the shape $\phi (q)$ in their denominators. One may, however, readily appeal to Proposition \ref{gory inequality prop} to sharpen \cite[Theorem 1.2]{Yam} for $q > 10^5$, as described in the previous paragraph.

If $q$ is a modulus for which the corresponding quadratic $L$-functions have no exceptional zero, all these results hold with a much weaker condition on the size of~$x$. In particular, this is the case, via Platt~\cite{Pla2}, for $10^5<q\le4\cdot10^5$.

\begin{prop}  \label{no exceptional zero prop}
Let $q \geq 10^5$ be an integer and suppose that no quadratic Dirichlet $L$-function with conductor $q$ has a real zero exceeding $1-R_1/\log q$. Let $\kappa_2$ and $Z$ be real numbers with $\kappa_2 > 1$ and
$$
Z\leq \frac{\sqrt{\kappa_2/{R_1}} -0.85317}{\log \kappa_2} - \frac12 .
$$
Then for all $x\ge \exp \big( \max\{\kappa_2,4 R_1\log^2 q\} \big)$,
$$
\bigg|\psi(x;q,a) - \frac x{\phi(q)} \bigg| \le \frac14 \frac x{(\log x)^Z}
\quad\text{and}\quad
\bigg|\theta(x;q,a) - \frac x{\phi(q)} \bigg| \le \frac14 \frac x{(\log x)^Z}.
$$
\end{prop}

\begin{proof}
The first assertion, for $\psi(x;q,a)$, follows immediately from Lemma~\ref{FLM lemma} (in the case where no exceptional zero is present) and Lemma~\ref{second part of three lemma}. The second assertion, for $\theta(x;q,a)$, follows from Lemma~\ref{two-prong lemma}, together with Lemmas~\ref{second part of three lemma} and~\ref{third part of three lemma}.
\end{proof}

\subsection{Conversion of estimates for $\theta(x;q,a)$ to estimates for $\pi(x;q,a)$}  \label{sec54}

Our final task is to convert our upper bounds for $|\theta(x;q,a) - x/\phi(q)|$ for large $q$ to upper bounds for $|\pi(x;q,a) - \Li(x)/\phi(q)|$. We do so using the same standard partial summation relationship that we exploited in Proposition~\ref{first theta to pi prop} for smaller~$q$; the proof is complicated slightly by our desire to achieve a savings of an arbitrary power of $\log x$ in the error term.

\begin{prop}  \label{gory pi inequality prop}
Let $q\geq 10^5$ be an integer and let $Z>0$, $\kappa_1 \ge 0.0132$, $\kappa_2>1$, and $\kappa_3>1$ be real numbers satisfying
the inequality~\eqref{licorice}.
Then if $x$ is a real number for which
$$
x/(\log x)^{Z+1}\ge 2000\exp \big( \max\{\kappa_1\sqrt q\log^3q, \kappa_2,\kappa_3,Z+28\} \big),
$$
it follows that
\begin{equation}  \label{pi Li upper bound large q}
\bigg|\pi(x;q,a) - \frac{\Li(x)}{\phi(q)} \bigg| \le \frac14 \frac x{(\log x)^{Z+1}}.
\end{equation}
\end{prop}

\begin{proof}
Define $x_4 = \exp \big( \max\{\kappa_1\sqrt q\log^3q, \kappa_2,\kappa_3,Z+28\} \big)$. The function $f(x) = x/(\log x)^{Z+1}$ is increasing for $x>e^{Z+1}$ and hence increasing for $x\ge x_4$; its value $f(x_4)$ is certainly less than $2000x_4$. Therefore the equation $f(x) = 2000x_4$ has a unique solution greater than $x_4$, which we call~$x_5$, so that the proposition asserts the upper bound~\eqref{pi Li upper bound large q} for $x\ge x_5$.
Start at equation~\eqref{theta to pi PS} (note Definition~\ref{double error def} for $E(x;q,a)$):
\begin{align*}
\pi(x;q,a) - \frac{\Li(x)}{\phi(q)} & = E(x_4;q,a) + \frac{\theta(x;q,a) - x/\phi(q)}{\log x} \\
&+ \int_{x_4}^x \bigg( \theta(x;q,a) - \frac x{\phi(q)} \bigg) \frac{dt}{t\log^2t}. \\
\end{align*}
So by the upper bound~\eqref{2401} and the fact that $\log x_4 \ge Z+28 > Z+1$,

\begin{flalign*}
&\bigg| \pi(x;q,a) - \frac{\Li(x)}{\phi(q)} \bigg| \le \big| E(x_4;q,a) \big| + 0.2401 \frac x{(\log x)^{Z+1}} + 0.2401 \int_{x_4}^x \frac{dt}{(\log t)^{Z+2}} \\
&\le \big| E(x_4;q,a) \big| + 0.2401 \frac x{(\log x)^{Z+1}} + \frac{0.2401}{(\log x_4 - (Z+1))} \int_{x_4}^x \frac{\log t - (Z+1)}{(\log t)^{Z+2}}\, dt \\
&= \big| E(x_4;q,a) \big| + 0.2401 \frac x{(\log x)^{Z+1}} + \frac{0.2401}{(\log x_4 - (Z+1))} \frac t{(\log t)^{Z+1}}\bigg|_{x_4}^x \\
&\le \big| E(x_4;q,a) \big| + \frac{0.2401(\log x_4 - Z)}{\log x_4 - (Z+1)} \frac x{(\log x)^{Z+1}} - \frac{0.2401}{(\log x_4 - (Z+1))} \frac{x_4}{(\log x_4)^{Z+1}} \\
&\le \big| E(x_4;q,a) \big| + \frac{0.2401(\log x_4 - Z)}{\log x_4 - (Z+1)} \frac x{(\log x)^{Z+1}}.
\end{flalign*}
A trivial upper bound  for $| E(u;q,a) |$ is, for $u > 3$, simply $2u$. To see this, note that, from Definition~\ref{double error def},
$$
\left|  E(u;q,a) \right| \leq \max \left\{ \pi(u;q,a) + \frac{u}{\phi (q) \log u},  \frac{\Li(u)}{\phi(q)} + \frac{\theta(u;q,a)}{\log u}  \right\}
$$
whereby, replacing $\pi(u;q,a)$ by $\pi (u)$ and ${\theta(u;q,a)}$ by $\theta(u)$, and appealing to bounds of Rosser-Schoenfeld \cite{RS} leads to the desired conclusion.
It follows that, for $x\ge x_4$,
\begin{align*}
\bigg| \pi(x;q,a) - \frac{\Li(x)}{\phi(q)} \bigg| &\le 2x_4 + \frac{0.2401(\log x_4 - Z)}{\log x_4 - (Z+1)} \frac x{(\log x)^{Z+1}} \\
&= \frac x{(\log x)^{Z+1}} \bigg( \frac{0.2401(\log x_4 - Z)}{\log x_4 - (Z+1)} + \frac{2x_4 (\log x)^{Z+1}}x \bigg).
\end{align*}
Note that $\frac{(\log x)^{Z+1}}x$ is decreasing for $x>e^{Z+1}$; since
$$
\log x_5 > \log x_4 \ge Z+28 > Z+1,
$$
we see that for $x\ge x_5$,
\begin{align*}
\bigg| \pi(x;q,a) - \frac{\Li(x)}{\phi(q)} \bigg| &= \frac x{(\log x)^{Z+1}} \bigg( \frac{0.2401(\log x_4 - Z)}{\log x_4 - (Z+1)} + \frac{2x_4 (\log x_5)^{Z+1}}{x_5} \bigg) \\
&= \frac x{(\log x)^{Z+1}} \bigg( \frac{0.2401(\log x_4 - Z)}{\log x_4 - (Z+1)} + \frac1{1000} \bigg)
\end{align*}
by the definition of~$x_5$.
The first summand in parentheses is a decreasing function of $\log x_4$ (when $\log x_4>Z+1$), and its value when we replace $\log x_4$ with the smaller quantity $Z+28$ is less than $0.249$, which completes the proof.
\end{proof}

\begin{cor}  \label{same result as for small q for pi cor}
For all $q>10^5$ and $x\ge\exp\big(0.03 \sqrt q\log^3q\big)$,
\begin{equation*}
\bigg|\pi(x;q,a) - \frac{\Li(x)}{\phi(q)} \bigg| \le \frac1{160}  \frac x{\log^2 x}.
\end{equation*}
\end{cor}

\begin{proof}
Set $Z=1.4$, $\kappa_1=0.0295$ and $\kappa_2=\kappa_3=14200$. By direct calculation, the hypotheses of Proposition~\ref{gory pi inequality prop} are satisfied. Moreover, as $q\geq 10^5$,
\[
\kappa_1\sqrt q\log^3q \ge \kappa_1\sqrt{10^5}(\log10^5)^3 > 14200 \geq \max\{\kappa_2,\kappa_3,Z+28\},
\]
and therefore the conclusion of Proposition~\ref{gory pi inequality prop} holds as long as we have
$$
x/(\log x)^{Z+1} \ge 2000\exp(\kappa_1\sqrt q\log^3q).
$$

Since we assume that $x\ge \exp\big(0.03 \sqrt q\log^3q\big)$,
$$
\frac{x}{(\log x)^{2.4}} \geq \frac{ \exp\big(0.03 \sqrt q\log^3q\big)}{(0.03 \sqrt q \log^3q)^{2.4}}
$$
and hence it remains to show that
$$
\exp(0.0005 \sqrt q\log^3q) > 2000 (0.03 \sqrt q \log^3q)^{2.4}.
$$
Since $q \geq 10^5$, we may verify that this inequality is satisfied for $q=10^5$ and then check that the quotient of the left-hand side and the right-hand side is increasing by taking its logarithmic derivative.
We may thus apply Proposition~\ref{gory pi inequality prop} to conclude that
$$
\bigg|\pi(x;q,a) - \frac{\Li(x)}{\phi(q)} \bigg| \le \frac 14 \frac{x}{(\log x)^{Z+1}} = \frac{1}{4}  \frac{x}{\log ^2x} \frac{1}{(\log x)^{Z-1}}
$$
and hence that
$$
\bigg|\pi(x;q,a) - \frac{\Li(x)}{\phi(q)} \bigg|
 < \frac{1}{4} \frac{1}{14400^{0.4}}  \frac{x}{\log^2 x} < \frac{1}{160} \frac{x}{\log^2 x}.
$$
\end{proof}

\setcounter{section}{0}
\renewcommand{\thesection}{\Alph{section}}

\section{Appendix: Computational details}  \label{computational appendix}
Many of the proofs in this paper required considerable computations, which we carried out using a variety of \texttt{C++},
\texttt{Perl}, \texttt{Python}, and \texttt{Sage} code. The resulting data files were manipulated using standard Unix
tools such as \texttt{awk}, \texttt{grep}, and \texttt{sort}. The smallest of the required computations were easily
performed on a laptop in a few seconds, while the largest required thousands of hours of CPU time on a computing
cluster. In the appendices below we give explanations of the computations and also links to the computer code and
resulting data. The interested reader can find a summary of the available files at the following webpage:
\begin{center}
\texttt{\href{http://www.nt.math.ubc.ca/BeMaObRe/}{\url{
http://www.nt.math.ubc.ca/BeMaObRe/}}}
\end{center}

\subsection{Verification of bound on $N(T,\chi_0)$ for principal characters $\chi_0$ and the computation of $\nu_2(x)$}
 \label{zeta comp sec}
In order to complete the proof of Proposition~\ref{quoting Trudgian}, we need to verify the asserted bound for $\chi$ principal and $1 \leq T \leq
1014$. This can be done quite directly by comparing the bound against a table of zeta function zeros. Such data is
available from websites such as the $L$-functions and Modular Forms Database~\cite{LMFDB} or other computer algebra software (such as Sage). At the
$k$th zero of the zeta function, which is of the form $\frac{1}{2} + i\gamma_k$, we compute the upper and lower
bounds implicit in the statement of the bound at $t=\gamma_k$, remembering that when we take limits from left and right
the quantity $N(T,\chi_0)$ is set to $2(k-1)$ and $2k$ respectively. We give Sage code to perform this verification and
its output in the
\begin{center}
\texttt{\href{http://www.nt.math.ubc.ca/BeMaObRe/prop2.6/}{BeMaObRe/c-psi-theta-pi/prop2.6/}}
\end{center}
subdirectory.

\subsection{Using {\bf \lcalc} to compute $\nu_2$}
We make use of  Rubinstein's \lcalc\ program to compute zeros of $L$-functions. For the sake of interfacing with \lcalc, we compute $\nu_2$ in the following way. While Definition~\ref{Have you heard the good nus?} allows for more general  $H_0(\chi)$, we only use functions $H_0$ that are constant on characters with the same conductor. Letting $H_0(d)$ be that constant, we have
  \[
  \nu_2(q,H_0) = \sum_{\chi \mod q} \nu_1(\chi,H_0(\chi))
  = \sum_{d | q} \sum_{\substack{\chi \mod q \\ q^* = d}} \nu_1(\chi,H_0(d)).
  \]
Further, the functions we use for $H_0$ take on the value 0 (no {\tt lcalc} data) or are at least 10.

If $H_0(d)=0$,  i.e., if we have made no calculations with {\tt lcalc} for characters with conductor $d$, we have
  \begin{align*}
  \sum_{\substack{\chi \mod q \\ q^* = d}}\nu_1(\chi,H_0(d))
   &= \sum_{\substack{\chi \mod q \\ q^* = d}} \left( -\Theta(d,1)+ \left\lfloor \frac 1\pi \log \frac d{2\pi e} + C_1 \log d + C_2 \right\rfloor \right) \\
   &= -\phi^*(d) \Theta(d,1) + \phi^*(d) \left\lfloor \frac 1\pi \log \frac d{2\pi e} + C_1 \log d + C_2 \right\rfloor \\
   &= \nu_0(d,0) - \overline{\nu_0}(d,0),
  \end{align*}
where we set
  \begin{align*}
  \nu_0(d,0) &= 0 \\
   \overline{\nu_0}(d,0) &= \phi^*(d) \Theta(d,1) - \phi^*(d) \left\lfloor \frac 1\pi \log \frac d{2\pi e} + C_1 \log d + C_2 \right\rfloor
  \end{align*}

If $H_0(d)\geq 1$, we must address some peculiarities of {\tt lcalc}. For real characters, {\tt lcalc} only gives the zeros with positive imaginary part, and for each complex-conjugate pair of nonreal characters, {\tt lcalc} returns the zeros of only one of the pair. Let $N'(h,\chi)$ be the number of zeros of $L(s,\chi)$ with imaginary part in $[0,h]$ if $\chi$ is real, and $N'(h,\chi)=N(h,\chi)$ if $\chi$ is nonreal.
We define, for real $h\geq1$,
  \[\overline{\nu_0}(d,h) = \phi^*(d)\Theta(d,h) + \frac{2}{h} \sideset{}{'}\sum_{\substack{\chi \mod q \\ q^*=d}} N'(h,\chi), \]
where $\sum^\prime$ indicates that the sum includes only one of each pair of complex conjugate characters.
We have (saving the definition of $\nu_0(d,h)$ for $h=H_0(d)\geq 1$ until after its use):
  \begin{align*}
  \sum_{\substack{\chi \mod q \\ q^* = d}} \nu_1(\chi, & H_0(d))
    = \sum_{\substack{\chi \mod q \\ q^* = d}} \left( -\Theta(d,h) - \frac{N(h,\chi)}{h}
          + \sum_{\substack{\rho \in \Zchistar \\ |\gamma| \leq h}} \frac{1}{\sqrt{\gamma^2+1/4}} \right) \\
    &=-\phi^*(d)\Theta(d,h)
           - \sum_{\substack{\chi\\ q^*=d}} \frac{N(h,\chi)}{h}
           +  \sum_{\substack{\chi \\ q^* = d}} \sum_{\substack{\rho \in \Zchistar \\ |\gamma| \leq h}}
                          \frac{1}{\sqrt{\gamma^2+1/4}}\\
    &= \nu_0(d,h) - \overline{\nu_0}(d,h).
  \end{align*}
The definition of $\nu_0(d,h)$ for $h\geq 1$ is then forced to be
  \begin{align*}
    \nu_0(d,h) &= \sum_{\substack{\chi\mod q \\ q^*=d}} \sum_{\substack{\rho \in \Zchi \\ |\gamma| \leq h}} \frac{1}{\sqrt{\gamma^2+1/4}} \\
    &=2\bigg( \sum_{\substack{\chi \text{ real}\\ q^*=d}} \sum_{\substack{\rho \in \Zchi \\ 0<\gamma<\leq h}} \frac{1}{\sqrt{\gamma^2+1/4}}
    + \sideset{}{'}\sum_{\substack{\chi \text{ not real} \\ q^*=d}} \sum_{\substack{\rho \in \Zchi \\ 0<\gamma<\leq h}} \frac{1}{\sqrt{\gamma^2+1/4}} \bigg).
  \end{align*}

With these definitions, we have
  \[\nu_2(q,H_0)  = \sum_{d|q} \left( \nu_0(d,H_0(d)) -  \overline{\nu_0}(d,H_0(d)) \right).\]

  We used $H_0(d) = 10^4$ for $d\leq 12$, $H_0(d)= 10^3$ for $d\leq 1000$, $H_0(d)=10^2$ for $d\leq 2500$, and $H_0(d) = 10$ for $d\leq 10^4$. Then, for a given choice of $H$, we use the largest value of $H_0(d)$ that is less than $H$. For example, with $H=120$, we use:
  \[H_0(d) = \begin{cases} 100, & \text{if $d\leq 2500$,} \\
      10, &\text{if $2500<d\leq 10^4$,} \\
      0, &\text{if $d>10^4$.}
    \end{cases}\]

\subsection{Computations of worst-case error bounds for $q\le 10^5$ and for $x\le x_2(q)$}
\label{ssec comp bpsi}
All our computations were split according to the modulus $q$. For each $q$, we
generated the sequence of primes using the \texttt{primesieve} library for
\texttt{C++} \cite{Wal}. This implements a very highly optimized sieve of Eratosthenes with wheel
factorisation. We experimented with storing the primes in a file on disc, but found that it was faster to
generate them each time using \texttt{primesieve}. As each
prime was generated, its residue was computed and the three functions
$$
  \pi(x;q,a) = \!\! \sum_{\substack{p\le x\\ p\equiv a\mod q}} \!\! 1, \;
  \theta(x;q,a) = \!\! \sum_{\substack{p\le x\\ p\equiv a\mod q}} \!\! \log p, \;
  \psi(x;q,a) = \!\! \sum_{\substack{p^n\le x\\ p^n\equiv a\mod q}} \!\!\log p
$$
were updated.

The function $\pi(x;q,a)$ is straightforward,  simply requiring integer arithmetic.
However the functions $\theta(x;q,a)$ and $\psi(x;q,a)$ involve summing anywhere up to
$10^{12}$ floating point numbers. In such computations considerable rounding error can
occur. To deal with these errors, we used interval arithmetic to keep track of upper
and lower bounds on $\theta$ and $\psi$.

As we computed $\psi$, $\theta$ and $\pi$ for increasing $x$, we also stored data about the functions
$$
  \frac{1}{\sqrt{x}}\!\left(\psi(x;q,a)- \frac{x}{\phi(q)}\right), \;
  \frac{1}{\sqrt{x}}\!\left(\theta(x;q,a)- \frac{x}{\phi(q)}\right), \;
  \frac{\log x}{\sqrt{x}}\!\left(\pi(x;q,a)- \frac{\Li(x)}{\phi(q)}\right),
$$
as well as the variant
$$
  \frac{1}{\sqrt{x}}\left(\theta_\#(x;q,a)- \frac{x}{\phi(q)}\right) =
  \frac{1}{\sqrt{x}}\left(\theta(x;q,a)- \frac{x-\xi_2(q,a)\sqrt{x}}{\phi(q)}\right)
$$
as defined in equation~\eqref{thetaH def}. Each of these expressions is monotone
decreasing between jumps at primes and prime powers. Hence to keep track of the maximum
value of each on a given interval, it suffices to check their left and right limits at each
prime power (including the primes themselves) and at the ends of each interval. A running maximum was kept for each
function and was dumped to a file at each change. For $2\leq x \leq 10^{11}$, for example,
each modulus took approximately 1 hour on a single core on the WestGrid
computing cluster. Spread over the cluster, which is shared with other users, the whole computation took about
a month of real time.

As part of these computations, we needed to be able to evaluate the logarithmic integral $\Li(z)$ quickly. We exploited the exponential integral $\Ei(u) = -\int_{-u}^\infty \frac{e^{-t}}t\,dt$ via the formula $\Li(z) = \Ei(\log z) - \Ei(\log 2)$.
Initially, we computed $\Ei(u)$ using the
series~\cite[equation 5.1.10]{AS}
$$
    \Ei(u) = C_0 + \log|u| + \sum_{k=1}^\infty \frac{u^k}{k \cdot k!};
$$
in practice, however, this turned out to be too slow for our purposes. Instead we pre-computed
$\Ei(u)$ using the above series at $33\cdot1000$ equally spaced points $u$ over the range $0 \leq u \leq 33$
(corresponding to $1 \leq z \leq e^{33} \approx 2\cdot10^{14}$). Then, in order to compute $\Ei(u)$ away
from those points, we precomputed the Taylor expansion of $\Ei(u)$ at each of those
$33\cdot1000$ points, namely
\begin{align} \label{ei}
  \Ei(u) &= \Ei(v) + e^v \left(
  \frac{1}{v}(u-v) + \frac{v-1}{2v^2}(u-v)^2 + \frac{v^2-2v+2}{6v^3}(u-v)^3 + \cdots
\right).
\end{align}
We found that the error in this approach was sufficiently small when we truncated
the Taylor expansion~\eqref{ei} at the cubic term. We could then build the error in Taylor approximation
into our interval arithmetic via the Lagrange remainder theorem.

For $1\le x\le x_2(q)$, where $x_2(q)$ is defined in (\ref{x2 definition}), for example, we computed that for all $q$ with $3 \leq q \leq 10^5$ and $q \not\equiv 2 \mod{4}$,
\begin{equation}  \label{sup bs}
\begin{split}
  \frac{1}{\sqrt{x}}\left|\psi(x;q,a)- \frac{x}{\phi(q)}\right|
 &\leq 1.118034 \; \; (\text{supremum achieved at }q=4,\, x=5^-) \\
  \frac{1}{\sqrt{x}}\left|\theta(x;q,a)- \frac{x}{\phi(q)}\right|  &\leq
1.817557 \; \; (\text{supremum achieved at }q=8,\, x=11257^-) \\
  \frac{1}{\sqrt{x}}\left|\theta_\#(x;q,a)- \frac{x}{\phi(q)}\right|  &\leq
1.053542 \; \; (\text{supremum achieved at }q=3,\, x=227^-) \\
  \frac{\log x}{\sqrt{x}}\left|\pi(x;q,a)- \frac{\Li(x)}{\phi(q)}\right|
&\leq 2.253192 \; \; (\text{supremum achieved at }q=4,\, x=229^-).
\end{split}
\end{equation}
Indeed, our computations gave corresponding constants $b_\psi(q)$, $b_\theta(q)$, $b_{\theta\#}(q)$, and
$b_\pi(q)$ for each modulus $q$ under discussion, which are the smallest constants such that the inequalities
\begin{equation}  \label{b const defs}
\begin{split}
\left|\psi(x;q,a)- \frac{x}{\phi(q)}\right| &\le b_\psi(q) \sqrt x \\
\left|\theta(x;q,a)- \frac{x}{\phi(q)}\right| &\le b_\theta(q) \sqrt x \\
\left|\theta_\#(x;q,a)- \frac{x}{\phi(q)}\right| &\le b_{\theta\#}(q) \sqrt x \\
\left|\pi(x;q,a)- \frac{\Li(x)}{\phi(q)}\right| &\le b_\pi(q) \frac{\sqrt x}{\log x}
\end{split}
\end{equation}
are satisfied for $1\le x\le x_2(q)$. A number of these are given in the following table, rounded up in the last decimal place; notice the four constants in equation~\eqref{sup bs} appearing in the rows corresponding to $q=3$, $4$, and~$8$.
\begin{align*}
 \begin{array}{|c|c||c|c|c|c|}
 \hline
  q & x_2(q) & b_\psi(q) & b_\theta(q) & b_{\theta\#}(q)  & b_\pi(q)
  \\
  \hline \hline
  3 & 4\cdot10^{13} & 1.070833 &1.798158 &1.053542  & 2.186908\\
  4 & 4\cdot10^{13} & 1.118034 &1.780719 &1.034832  & 2.253192\\
  5 & 4\cdot10^{13} & 0.886346 &1.412480 &0.912480  & 1.862036\\
  \hline
  7 &10^{13} & 0.782579 &1.116838 &0.829249  & 1.260651 \\
  8 &10^{13} & 0.926535 &1.817557 &0.887952  & 2.213119\\
  9 &10^{13} & 0.788900 &1.108042 &0.899812  & 1.229315 \\
  11 &10^{13} & 0.878823 & 0.976421 & 0.885771& 1.103821\\
  12 &10^{13} & 0.906786 & 1.735501 & 0.906786 & 2.001350 \\
  \vdots & \vdots & \vdots & \vdots & \vdots & \vdots \\
  \hline
  101 & 10^{12} & 0.709028 & 0.709028 & 0.717402 & 0.777577 \\
  \vdots & \vdots & \vdots & \vdots & \vdots & \vdots \\
  \hline
  10001 & 10^{11} & 0.735215 & 0.735215 & 0.735215 & 0.735207 \\
  \vdots & \vdots & \vdots & \vdots & \vdots & \vdots \\
  10^5 &10^{11} & 0.735419 & 0.735419 & 0.735419 & 0.735417 \\
  \hline
 \end{array}
\end{align*}

(Similar data for $x$ in the (smaller) range $1 \leq x \leq 10^{10}$  can be found in~\cite[Table~2]{RR}. Historically, computations of this type have been viewed as evidence supporting the Generalized Riemann Hypothesis, since these error terms would grow like a larger power of $x$ should GRH be false.) Note that we have skipped the moduli $q\equiv2 \mod{4}$, since the distribution of prime powers in arithmetic progessions modulo such $q$ is essentially equivalent to the distribution of prime powers modulo $\frac{q}{2}$; see Lemma~\ref{2mod4 lemma} below.

In the course of running these computations, we chose a computational-time trade-off between large values of $x_2(q)$ for fewer smaller moduli and lesser values of $x_2(q)$ for the entire range of moduli.
The total time for the $x_2(q)= 10^{12}$ run (for $q$ with $100 <  q \leq 10^4$) was similar to the initial $10^{11}$ run (to $q=10^5$), while the $10^{13}$ and $4\cdot
10^{13}$ runs (to $q=5$ and $q=100$, respectively) took approximately 2 weeks of real time. The data for all of these computations can be found in the
\begin{center}
\texttt{\href{http://www.nt.math.ubc.ca/BeMaObRe/b-psi-theta-pi/}{BeMaObRe/b-psi-theta-pi/}}
\end{center}
subdirectory and are described in the associated \texttt{readme} file.

As has been observed before with similar computations, most of the entries in this table (particularly for large $q$) are extremely close to $(\log7)/\sqrt7 \approx 0.735485$. For the relatively small values of $x$ under consideration, the maximum value of (for example) $|\theta(x;q,a) - x/\phi(q)|/\sqrt x$ occurs at the first prime $p$ congruent to $a\mod q$, leading to the value $|\log p - p/\phi(q)/\sqrt p$ which, for $q$ large, is very close to $(\log p)/\sqrt p$; and the function $(\log x)/\sqrt x$ is maximized at $x=e^2$, to which $p=7$ is the closest prime. If one were to continue these calculations for larger and larger $x$, we would see these values $b_\psi(q)$, $b_{\theta}(q)$, and $b_{\theta\#}(q)$ increase irregularly to infinity.

We also observe, for the small moduli $q$ where the single prime $7$ is not dictating the values of the constants $b_\theta(q)$ and $b_{\theta\#}(q)$, that the latter constants are significantly smaller than the former; this observation reflects the fact that the distribution of $\big(\theta_\#(x;q,a) - x/\phi(q)\big) / \sqrt x$ is centered around $0$ (which is the precise reason for the definition~\eqref{thetaH def} of $\theta_\#(x;q,a)$ in the first place), unlike the distribution of $\big(\theta(x;q,a) - x/\phi(q)\big) / \sqrt x$.

If $q$ is twice an odd number, then the distribution of prime powers in arithmetic progressions modulo $q$ is almost completely equivalent to the distribution of prime powers modulo~$q/2$ (the powers of $2$ are the only ones that are counted differently).

\begin{lemma}  \label{2mod4 lemma}
Let $k\ge3$ be an odd integer, and let $a$ be an odd integer that is coprime to~$k$. Then for all $x\ge2$,
\begin{align*}
\big| \psi(x;2k,a) - \psi(x;k,a) \big| &\le \bigg( 1 + \frac{\log(x/2)}{\log(k+1)} \bigg)\log 2 \le \log x, \\
\big| \theta(x;2k,a) - \psi(x;k,a) \big| &\le \log 2 < 1, \\
\big| \pi(x;2k,a) - \pi(x;k,a) \big| &\le 1.
\end{align*}
\end{lemma}

\begin{proof}
We note that $\psi(x;k,a) = \psi(x;2k,a) + \psi(x;2k,a+k)$ exactly. On the other hand, every integer that is congruent to $a+k\mod{2k}$ is even, so the only prime powers that could be counted by $\psi(x;2k,a+k)$ are powers of~$2$; and note that a power of $2$ is congruent to $a+k\mod{2k}$ if and only if it is congruent to $a\mod k$. If such a power exists, let $2^m$ be the smallest prime power congruent to $a\mod k$, and let $n$ be the order of $2$ modulo $k$, so that the powers of $2$ that are congruent to $a\mod k$ are precisely $2^m,2^{m+n},2^{m+2n},\dots$. The number of such powers of $2$ not exceeding $x$ is exactly
\[
1 + \bigg\lfloor \frac{\log(x/2^m)}{\log(2^n)} \bigg\rfloor \le 1 + \frac{\log(x/2^m)}{\log(2^n)} \le 1 + \frac{\log(x/2)}{\log(k+1)},
\]
where the last inequality is due to $m\ge1$ and the fact that $2^n>1$ is congruent to $1\mod k$ and therefore must be at least $k+1$. The first inequality asserted in the statement of the lemma follows from the fact that each such power of $2$ contributes $\log 2$ to $\psi(x;2k,a+k) = \psi(x;k,a) - \psi(x;2k,a)$. The second and third asserted inequalities have similar proofs (easier, in fact, since those two functions count only primes and not prime powers).
\end{proof}

\subsection{Computations of the leading constants $c_\psi$, $c_\theta$, and $c_\pi$ for $q\le 10^5$}
\label{ssec cpsithetapi}

The constants $c_\psi(q)$ and $c_\theta(q)$ were computed using Theorem~\ref{code this funky theorem}
and Theorem~\ref{code this funky theta}, after which the constants $c_\pi(q)$ were computed using Proposition~\ref{first theta
to pi prop}. While the expressions in Theorem~\ref{code this funky theorem}
and Theorem~\ref{code this funky theta} are cumbersome, evaluating them is actually a straightforward (if ugly) computation using
\texttt{C++}. To simplify our code we precomputed data for some of the auxillary functions (the totient function $\phi(q)$ and the
factorisations involved in the function $\Delta(x;q)$ from Definition~\ref{xi defs}) using the \texttt{Sage} computer algebra
system. We also verified our
$c_\psi(q)$, $c_\theta(q)$, and $c_\pi(q)$ values using the \texttt{Mathematica} computer algebra system.

The resulting code is quite fast, and all of these constants can be computed for $q \leq 10^5$ and a given $m, H$ and $x_2$ in only a few seconds. For a given choice of $q$ and $x_2$, we computed the constants for $4 \leq m \leq 12$ and computed the minimum value over \(H_1(m) \leq H \leq 10^9\); it turned out that $m \in \{6,7,8,9\}$ gave the best bound in every case. Our results are given  in the
\begin{center}
\texttt{\href{http://www.nt.math.ubc.ca/BeMaObRe/c-psi-theta-pi/}{BeMaObRe/c-psi-theta-pi/}}
\end{center}
subdirectory and described in the corresponding \texttt{readme} file. By way of example, we have
\begin{align*}
 \begin{array}{|c||c|c|c|}
 \hline
 q & c_\psi(q) & c_\theta(q) & c_\pi(q)
 \\
 \hline \hline
 3  &  0.0003964   &    0.0004015   &    0.0004187 \\
 4  &  0.0004770   &    0.0004822   &    0.0005028 \\
 5  &  0.0003665   &    0.0003716   &    0.0003876 \\
 6  &  0.0003964   &    0.0004015   &    0.0004187 \\
 7  &  0.0004584   &    0.0004657   &    0.0004857 \\
 8  &  0.0005742   &    0.0005840   &    0.0006091 \\
 9  &  0.0005048   &    0.0005122   &    0.0005342 \\
 10 &  0.0003665   &    0.0003716   &    0.0003876 \\
 11 &  0.0004508   &    0.0004553   &    0.0004748 \\
 12 &  0.0006730   &    0.0006829   &    0.0007121 \\
 \vdots & \vdots & \vdots & \vdots  \\
 101 & 0.0008443   &    0.0008460   &    0.0008822 \\
 \vdots & \vdots & \vdots & \vdots \\
10001& 0.0034386   &    0.0034403    &   0.0035878 \\
 \vdots & \vdots & \vdots & \vdots  \\
 10^5  & 0.0051178 & 0.0051196 & 0.0053391  \\
 \hline
 \end{array}
\end{align*}

Note that  in order to compute $c_\pi(q)$ from $c_\theta(q)$ using Proposition~\ref{first theta to pi prop}, we must verify the hypothesis~\eqref{last two terms} of that proposition. To avoid having to
explicitly check inequality~\eqref{last two terms} for $x>10^{11}$, we examined $x_1(q)$ (see Appendix~\ref{ssec xpsithetapi}) and confirmed  that $x_1(q) < 10^{11}$. Hence it sufficed to evaluate $E(x_3;q,a)$ at $x_3=10^{11}$. To do this, we computed $\max_{\gcd(a,q)=1} |E(10^{11},q,a)|$ (using code similar to that used to compute the constants $b_\theta(q)$ and $b_\pi(q)$) for each modulus~$q$ and verified inequality~\eqref{last two terms}. This computation took about 1 hour for each modulus and so approximately 1 month of real time. The data from this computation can be found in the
\begin{center}
\texttt{\href{http://www.nt.math.ubc.ca/BeMaObRe/c-psi-theta-pi/E-bound/}{BeMaObRe/c-psi-theta-pi/E-bound/}}
\end{center}
subdirectory.

\subsection{Dominant contributions to $c_\psi(q)$, $c_\theta(q)$, and $c_\pi(q)$ for $q\le 10^5$}  \label{analysis of constants section}

Let us recall the function $D_{q,m,R}(x_2;H_0,H,H_2)$ from Definition~\ref{D def}, certain values of which are exactly equal to
$c_\psi(q)$.
While $D_{q,m,R}(x_2;H_0,H,H_2)$ is programmable and hence suffices for our numerical results, it would be helpful to
have some intuition about which terms in the expression contribute the most to its value. Here we report on numerical
investigations into the relative sizes of the constituent expressions, for the relevant ranges of parameters ($3\le
q\le10^5$, $10^{11}\le x_2\le4\cdot10^{13}$, $R=5.6$, $3\le m\le 12$, and various choices for $H, H_0$ and $H_2$).

After running our various computations and analyzing the resulting data, our conclusions are as follows; recall that the quantities $T_1, T_2, T_3$ and $T_4$ are defined in Definition~\ref{D def} and satisfy
$$
D_{q,m,R}(x_2;H_0,H,H_2) =\frac{1}{\phi(q)} \left( T_1 + T_2 + T_3 + T_4 \right).
$$

\begin{itemize}
\item As noted previously, the optimal value for $m$ is always in $\{6,7,8,9\}$, a fact for which we have no explanation.
\item The optimal value for $H$ quickly becomes small, hitting our floor of $H_1(q)$ around $q=5000$. The parameter $H$ controls the zeros which get smoothed, and larger $q$, which have more low-height zeros, benefit more from this.
\item The term $T_4$ is negligible, always several orders of magnitude smaller than the other terms. The term $T_3$ is nearly always negligble, accounting for less than $2\%$ of the total.
\item The term $T_1$, where low-height zeros hold sway, accounts for 20\%-50\% of the total for $q\leq 100$, and growing to around 60\% for $q$ near $10^5$. Note that for large $q$, we don't compute these zeros, instead relying on Trudgian's bound.
\item The term $T_2$, where zeros potentially close to $\sigma=1$ have their influence, accounts for 50\%-80\% of the total for smaller $q$, and about 40\% for larger $q$.
\item The balance between $T_1$ and $T_2$ depends heavily on the zeros of extremely low height, and so bounces around considerably for small $q$. For $q$ near $10^5$, for which we do not calculate any zeros, the balance is consistently about $59.5\%$ for $T_1$, about $39.5\%$ for $T_2$, and about $1\%$ for $T_3$.
\end{itemize}

\begin{figure}[H]
\[
  \begin{array}{|c|ccccc|}\hline
 q & \text{factorization of $q$} &m & x_2(q) & H & c_\psi(q)  \\ \hline
 3 & 3 & 8 & 4\cdot 10^{13} & 492130 & 0.0003964  \\
 4 & 2^2 & 7 & 4\cdot 10^{13} & 337539 & 0.0004770  \\
 5 & 5 & 8 & 4\cdot 10^{13} & 276297 & 0.0003665  \\
 101 & 101 & 6 & 10^{12} & 7484 & 0.0008443 \\
 5040 & 2^4 \cdot 3^2 \cdot 5\cdot 7& 6 & 10^{12} & 262 & 0.0011204  \\
 55440 & 2^4 \cdot 3^2 \cdot 5\cdot 7 \cdot 11& 7 & 10^{11} & 137 & 0.0034065  \\
 55441 & 55441 & 8 & 10^{11} & 120 & 0.0048288  \\
 99991 & 99991 & 8 & 10^{11} & 120 & 0.0058889  \\
 100000 & 2^5 \cdot 5^5 & 8 & 10^{11} & 120 & 0.0051178 \\ \hline
\end{array}
\]
\end{figure}
\begin{figure}[H]
\[
  \begin{array}{|c|ccc|}\hline
 q  & T_1 & T_2 &T_3  \\ \hline
 3 & 27.73 \% & 72.27 \% & 0 \% \\
 4 & 22.18 \% & 77.82 \% & 0 \% \\
 5 &  30.39 \% & 69.61 \% & 0 \% \\
 101 & 69.27 \% & 30.71 \% & 0.02 \% \\
 5040 &  37.58 \% & 61.54 \% & 0.88 \%  \\
 55440 &  62.09 \% & 37.30 \% & 0.61 \% \\
 55441 &  69.93 \% & 29.40 \% & 0.67 \%  \\
 99991 &  59.14 \% & 39.87 \% & 0.99 \% \\
 100000 &  58.63 \% & 40.44 \% & 0.94 \%  \\ \hline
\end{array}
\]
\caption{A sampling of $q$ values, with $x_2(q)$, the optimal choices for $m$ and $H$, and corresponding $c_\psi(q)$. The second table lists the percentage of the bound on $c_\psi(q)$ that comes from each of $T_1, T_2$ and $T_3$; in each case $T_4$ contributes essentially $0 \%$.}
\end{figure}


\subsection{Computations of $x_\psi(q)$, $x_\theta(q)$, $x_{\theta\#}(q)$, $x_\pi(q)$, and $x_0(q)$ for $q\le 10^5$} \label{ssec xpsithetapi}

The computation of $x_0(q)$ was a three-step process. For the purposes of describing this process, we focus on
$\theta(x;q,a)$ since the approach for the other functions is very similar.

In brief, we start by calculating a crude upper bound on
$x_\theta(q)$ which we call $x_1(\theta;q)$, which is easily computed from our $b_\theta(q)$ and $c_\theta(q)$ data (see Appendices~\ref{ssec comp bpsi} and~\ref{ssec cpsithetapi}); typically $x_1(\theta;q)$ is significantly smaller than
$x_2(q)$. Now to compute
$x_\theta(q)$ we need only examine $x\le x_1(\theta;q)$, a much smaller range than $x\le x_2(q)$, which saves us considerable
computer time. Finally, from the accumulated data we found a simple upper bound $x_0(q)$  on our
more precise constants~$x_\theta(q)$.

We now discuss each of these steps in more detail (still concentrating on $\theta(x;q,a)$).  We
wish to find the smallest value of $x_\theta(q)$ so that for all $x \geq x_\theta(q)$ and all integers $a$ coprime to~$q$,
\begin{equation} \label{theta ctheta}
  \left|\theta(x;q,a) - \frac{x}{\phi(q)} \right| < c_\theta(q) \frac{x}{\log x}.
\end{equation}
We have already verified, for $x \leq x_2(q)$, that
\begin{align*}
  \left|\theta(x;q,a)- \frac{x}{\phi(q)}\right| < b_\theta(q) \sqrt{x}
\end{align*}
using the exhaustive computations described in Appendix~\ref{ssec comp bpsi} above. Accordingly we compute
$x_1=x_1(\theta;q)$ so that
\begin{align*}
  c_\theta(q) \frac{x_1}{\log x_1} = b_\theta(q) \sqrt{x_1},
\end{align*}
using a simple \texttt{Python} script and a bisection solver from the \texttt{scipy} library for \texttt{Python}, and then rounded up that value. From this argument we know that we will be able to take $x_\theta(q) \leq x_1(\theta;q)$.
Since we did not compute $b_\theta(q)$ for $q\equiv 2\mod{4}$, we instead make use of Lemma~\ref{2mod4
lemma} to infer that
\begin{align*}
  \left|\theta(x;q,a)- \frac{x}{\phi(q)}\right| < b_\theta(\tfrac q2) \sqrt{x} + 1;
\end{align*}
thus to compute $x_1(\theta;q)$ for $q\equiv 2\mod{4}$ we instead solve the slightly different equation
\begin{align*}
  c_\theta(q) \frac{x_1}{\log x_1} = b_\theta(q) \sqrt{x_1}+1.
\end{align*}

The process for calculating $x_1(\psi;q)$, $x_1(\theta_\#;q)$, and $x_1(\pi;q)$ is very similar: when $q\not\equiv2\mod4$ they are the positive solutions $x_1$ to the equations
$$
  c_\psi(q) \frac{x_1}{\log x_1} = b_\psi(\tfrac q2) \sqrt{x_1},  \; \; \;
  c_{\theta\#}(q) \frac{x_1}{\log x_1} = b_{\theta\#}(\tfrac q2) \sqrt{x_1}
  $$
  and
  $$
  c_\pi(q) \frac{x_1}{\log^2 x_1} = b_\pi(\tfrac q2)\frac{\sqrt{x_1}}{\log x_1},
$$
respectively, while when $q\equiv2\mod{4}$ they are the solutions to
$$
  c_\psi(q) \frac{x_1}{\log x_1} = b_\psi(\tfrac q2) \sqrt{x_1}+\log x_1,  \; \; \;
  c_{\theta\#}(q) \frac{x_1}{\log x_1} = b_{\theta\#}(\tfrac q2) \sqrt{x_1}+1
$$
and
$$
  c_\pi(q) \frac{x_1}{\log^2 x_1} = b_\pi(\tfrac q2)\frac{\sqrt{x_1}}{\log x_1} +1,
$$
respectively (using the results in Lemma~\ref{2mod4 lemma}).
The first few values for $x_1$ for the indicated functions are given below.
\begin{align*}
 \begin{array}{|c||c|c|c|c|}
 \hline
 q & x_1(\psi;q) & x_1(\theta;q) & x_1(\theta\#;q) & x_1(\pi;q) \\
 \hline \hline
 3   &  3.5290\cdot 10^{9}  & 1.0701\cdot 10^{10}  &  3.3100\cdot 10^{9}  &  1.4980\cdot 10^{10}  \\
 4   &  2.5810\cdot 10^{9}  & 7.0120\cdot 10^{9}  &  2.1260\cdot 10^{9}  &  1.0712\cdot 10^{10}  \\
 5   &  2.7660\cdot 10^{9}  & 7.4690\cdot 10^{9}  &  2.8590\cdot 10^{9}  &  1.2479\cdot 10^{10}  \\
 6   &  3.5320\cdot 10^{9}  & 1.0701\cdot 10^{10}  &  3.3100\cdot 10^{9}  &  1.4983\cdot 10^{10}  \\
 7   &  1.2830\cdot 10^{9}  & 2.7140\cdot 10^{9}  &  1.4080\cdot 10^{9}  &  3.2310\cdot 10^{9}  \\
 8   &  1.1320\cdot 10^{9}  & 4.8160\cdot 10^{9}  &  9.9300\cdot 10^{8}  &  6.7670\cdot 10^{9}  \\
 9   &  1.0550\cdot 10^{9}  & 2.1630\cdot 10^{9}  &  1.3660\cdot 10^{9}  &  2.4790\cdot 10^{9}  \\
 10  &  2.7680\cdot 10^{9}  & 7.4690\cdot 10^{9}  &  2.8600\cdot 10^{9}  &  1.2482\cdot 10^{10}  \\
 11  &  1.7200\cdot 10^{9}  & 2.1220\cdot 10^{9}  &  1.7120\cdot 10^{9}  &  2.5350\cdot 10^{9}  \\
 12  &  7.6000\cdot 10^{8}  & 3.0840\cdot 10^{9}  &  7.3600\cdot 10^{8}  &  3.8480\cdot 10^{9}  \\
\vdots & \vdots & \vdots & \vdots & \vdots \\
10^5 & 5.0\cdot 10^{6} & 5.0\cdot 10^6 & 5.0\cdot 10^6  & 5.0\cdot 10^6
 \\
 \hline
 \end{array}
\end{align*}
We give the full table of $x_1$ data in the
\begin{center}
\texttt{\href{http://www.nt.math.ubc.ca/BeMaObRe/x-psi-theta-pi/compute-x1/}{BeMaObRe/x-psi-theta-pi/compute-x1/}}
\end{center}
subdirectory.

We are now faced with the problem of determining the supremum $x_\theta(q)$ of those real numbers $x$ such that the inequality~\eqref{theta ctheta}
fails (again using $\theta(x;q,a)$ as the example for our discussion); from the previous calculation we know that this supremum is at most~$x_1(\theta;q)$.
In practice $x_1(\theta;q)$ is significantly smaller than $x_2(q)$, and so determining $x_\theta(q)$ from an exhaustive search over $x
\leq x_1(\theta;q)$ is substantially faster.
We again compute the left-hand side of the inequality~\eqref{theta ctheta} for $x$ equal to all primes and prime powers in the given range, using code similar to that
used to compute $b_\theta(q)$. For each residue class $a \mod{q}$ we record the largest prime or prime power $p^*(q;a)$ so
that
\begin{align*}
  \left|\theta(p^*(q,a);q,a)- \frac{p^*(q,a)}{\phi(q)}\right| > c_\theta(q) \cdot \frac{p^*(q,a)}{\log p^*(q,a)}.
\end{align*}
The procedure then breaks into two cases depending on the sign of $\big(\theta(p^*(q,a);q,a)-
\frac{p^*(q,a)}{\phi(q)}\big)$. Consider the figure below that gives a schematic comparison between
$\theta(x;q,a) - \frac{x}{\phi(q)}$ (the jagged paths denoting functions with jump discontinuities) and $\pm c_\theta(q) \frac{x}{\log x}$ (the curved lines).
\begin{center}
 \includegraphics[width=0.75\textwidth]{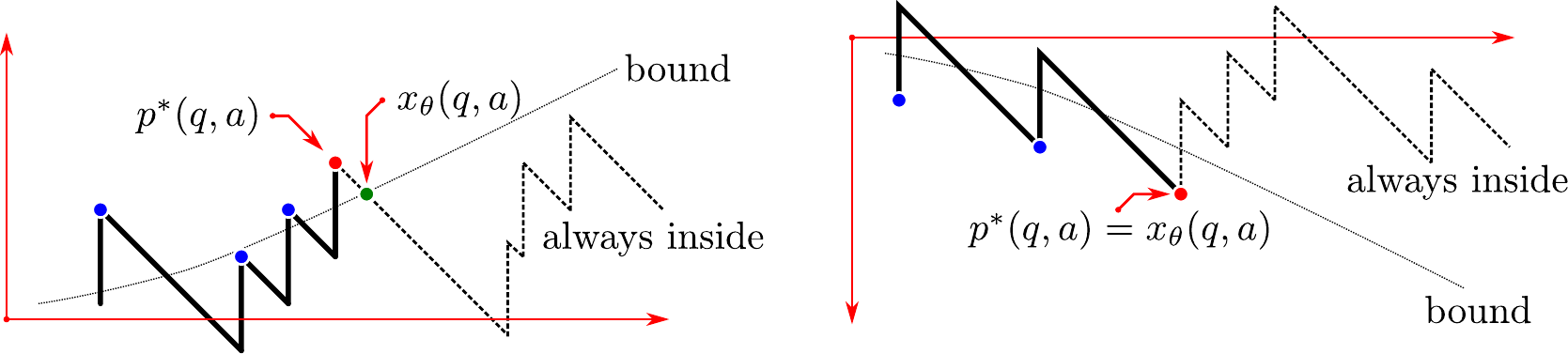}
\end{center}
\begin{itemize}
 \item If $\theta\big(p^*(q,a);q,a\big)- \dfrac{x}{\phi(q)} > 0$, then we use Newton's method or a bisection method to
solve
\begin{align*}
   \theta(p^*(q,a);q,a)- \frac{x}{\phi(q)} &= c_\theta \cdot \frac{x}{\log x}
\end{align*}
for $x= x_\theta(q,a)$ to the desired level of precision.

\item On the other hand, if $\theta\big(p^*(q,a);q,a\big)- \dfrac{x}{\phi(q)} < 0$ then simply $x_\theta(q,a) = p^*(q,a)$.
\end{itemize}
We then set $x_\theta(q) = \max_{\gcd(a,q)=1} x_\theta(q,a)$. We did analogous exhaustive computations to find $x_\psi(q)$,
$x_{\theta\#}(q)$, and $x_\pi(q)$; we give the first few values below (rounded up to the nearest integer).
\begin{align*}
 \begin{array}{|c||r|r|r|r|}
  \hline
  q & x_\psi(q) & x_\theta(q) & x_{\theta\#}(q) & x_\pi(q)
  \\
  \hline \hline
3  &  576{,}470{,}759  & 7{,}932{,}309{,}757 &  576{,}587{,}783  & 7{,}940{,}618{,}683 \\
4  &  952{,}930{,}663  & 4{,}800{,}162{,}889 &  952{,}941{,}971  & 5{,}438{,}260{,}589 \\
5  & 1{,}333{,}804{,}249 &  3{,}374{,}890{,}111 & 1{,}333{,}798{,}729 &  3{,}375{,}517{,}771 \\
6  &  576{,}470{,}831  & 7{,}932{,}309{,}757 &  576{,}587{,}783  & 7{,}940{,}618{,}683 \\
7  &  686{,}060{,}664  & 1{,}765{,}650{,}541 &  500{,}935{,}442  & 1{,}765{,}715{,}753 \\
8  &  603{,}874{,}695  & 2{,}261{,}078{,}657 &  603{,}453{,}377  & 2{,}265{,}738{,}169 \\
9  &  415{,}839{,}496  &  929{,}636{,}413  & 415{,}620{,}108  &     929{,}852{,}953  \\
10 & 1{,}333{,}804{,}249 &  3{,}374{,}890{,}111 & 1{,}333{,}798{,}729 &  3{,}375{,}517{,}771 \\
11 &  770{,}887{,}529  & 1{,}118{,}586{,}379 &  770{,}871{,}139  &     838{,}079{,}951  \\
12 &  501{,}271{,}535  & 1{,}305{,}214{,}597 &  501{,}062{,}258  & 1{,}970{,}827{,}897 \\
\vdots & \vdots\qquad & \vdots\qquad & \vdots\qquad & \vdots\qquad \\
10^5   &      17{,}876  &             17{,}870  &         17{,}931  &             16{,}871   \\
  \hline
\end{array}
\end{align*}
All of this data can be found in the
\begin{center}
\texttt{\href{http://www.nt.math.ubc.ca/BeMaObRe/x-psi-theta-pi/compute-x0/}{BeMaObRe/x-psi-theta-pi/compute-x0/}}
\end{center}
subdirectory.


\subsection{Computations of inequalities for $\pi(x;q,a)$ and $p_n(q,a)$, for $q\le 1200$ and very small $x$}
\label{ssec pipn comp}

To deduce Corollary~\ref{cor nice pi pn bounds} from Theorems~\ref{pi nice lower bound theorem} and \ref{pnqa bounds
theorem} for a particular modulus $3\le q\le 1200$,
we need to determine the largest $x$ at which each of the four inequalities
$$
 \pi(x;q,a)  >  \frac{x}{\phi(q) \log x}, \; \;   \pi(x;q,a)  <  \frac{x}{\phi(q) \log x} \left( 1 + \frac{5}{2 \log x}\right),
 $$
 $$
  x > \pi(x;q,a) \phi(q) \log(\pi(x;q,a) \phi(q)),
$$
and
$$
x  < \pi(x;q,a) \phi(q) \left( \log(\pi(x;q,a) \phi(q)) +\frac{4}{3} \log( \log(\pi(x;q,a) \phi(q)))\right)
$$
fails. (When $q=1$ and $q=2$, Corollary~\ref{cor nice pi pn bounds} follows from results of Rosser and Schoenfeld~\cite[equations (3.2), (3.5), (3.12), and (3.13)]{RS}.) More precisely, when $q\ge3$ we know that the inequalities hold for $x \geq x_0(q)$, so it suffices to check the inequalities for
$x<x_0(q)$.  Again, as was the case for calculating $b_\pi(q)$ in Appendix~\ref{ssec comp bpsi},
we compute $\pi(p;q,a)$ at each prime $p$ and then check the inequalities as $x$ approaches $p$ from the left and from
the right. Since $\pi(x;q,a)$ is an integer quantity, this can be done very efficiently with simple $\texttt{C++}$
code.

The data giving the last $x$ violating the inequalities is in the
\begin{center}
\texttt{\href{http://www.nt.math.ubc.ca/BeMaObRe/pi-pn-bounds/}{BeMaObRe/pi-pn-bounds/}}
\end{center}
subdirectory. Given this data, one can verify that the $x$ values are bounded by the simple quadratic functions of $q$
stated in Corollary~\ref{cor nice pi pn bounds}.

\subsection{Computations of error terms for $\psi (x; q, a)$, $\theta (x; q, a)$, and $\pi (x; q, a)$, for very small $x$}
\label{ssec comp tiny theta}

To prove Corollary~\ref{tiny-theta-cor} from Theorems~\ref{main psi theorem},~\ref{main theta theorem}, and~\ref{main
pi theorem} we found, for each $3\le q\le 10^5$, the largest values of

\begin{align} \label{goober}
  &\frac{\log x}{x} \left| \psi(x;q,a) - \frac{x}{\phi(q)} \right|,  \;
  \frac{\log x}{x} \left| \theta(x;q,a) - \frac{x}{\phi(q)} \right| \;  \\  \nonumber
 &  \text{ and } \; \; \frac{\log^2 x}{x} \left| \pi(x;q,a) - \frac{\Li(x)}{\phi(q)} \right|  \nonumber
\end{align}
\noindent for all $10^3 \leq x \leq \max\{x_\psi(q), x_\theta(q), x_\pi(q)\}$. Those largest values tend to occur quite close to $10^3$, as all three error terms are decaying roughly like $\log x/\sqrt x$. We confirmed that none of these maximal values exceeded $0.19$, $0.40$, and $0.59$, respectively. Since our main results ensure bounds for $x \geq x_\psi(q), x_\theta(q), x_\pi(q)$ (as required), it suffices to check that our computed values for $c_\psi(q)$, $c_\theta(q)$, and $c_\pi(q)$ (see Appendix~\ref{ssec cpsithetapi}) were also bounded by those three constants. The worst case bounds for $\psi(x;q,a), \theta(x;q,a)$ and $\pi(x;q,a)$ are achieved at $q=4, x=1423^-$, $q=4, x=1597^-$, and $q=3, x=1009^-$ (respectively), giving constants of $0.1659,  0.3126$ and $0.4236$ (respectively).

We then repeated this process for the range $10^6 \leq x \leq \max\{ x_\psi(q), x_\theta(q), x_\pi(q)\}$, comparing the results against the constants $0.011$, $0.024$, and $0.027$, respectively. In this case, the  worst case bounds for $\psi(x;q,a), \theta(x;q,a)$ and $\pi(x;q,a)$ are achieved at $q=46, x=1015853^-$, $q=4, x=100117^-$, and $q=4, x=1000117^-$ (respectively), giving constants of $0.0106, 0.0233$ and $0.0267$ (respectively).

While the methods in this paper work in theory for $q=1$ and $q=2$, we do use the assumption $q\ge3$ in many small ways to improve the constants in our intermediate arguments. We can, however, recover results for $q=1$ and $q=2$ from our existing results, by noting that (for example) every prime other than $3$ itself is counted by $\pi(x;3,1)+\pi(x;3,2)$.
In the case $q=2$, we observe that, for $x \geq 3$,
\begin{align*}
 \psi(x;2,1) &= \psi (x;3,1) + \psi(x;3,2) + \left\lfloor \frac{\log x}{\log 3} \right\rfloor \log 3 - \left\lfloor \frac{\log x}{\log 2} \right\rfloor \log 2, \\
 \theta (x; 2,1) &=  \theta (x;3,1) + \theta(x;3,2) + \log (3/2), \\
\pi(x;2,1) &= \pi (x;3,1) + \pi(x;3,2).
\end{align*}
Appealing to Theorems~\ref{main psi theorem}, \ref{main theta theorem}, and~\ref{main pi theorem}, and applying the triangle inequality, we thus have
\begin{align*}
\left| \psi (x; 2,1) - x \right| &< 2 c_{\psi}(3) \frac{x}{\log x} +1 \; \; \mbox{ for all } x \geq x_{\psi} (3), \\
\left| \theta (x; 2,1) - x \right| &< 2 c_{\theta}(3) \frac{x}{\log x}  + \log (3/2) \; \; \mbox{ for all } x \geq x_{\theta} (3), \\
\left| \pi (x; 2,1) -  \Li (x) \right| &< 2 c_{\pi}(3) \frac{x}{\log^2 x} \; \; \mbox{ for all } x \geq x_{\pi} (3).
\end{align*}
Similarly, in the case $q=1$, we find that
\begin{align}
\left| \psi (x) - x \right| &< 2 c_{\psi}(3) \frac{x}{\log x} +\log x \; \; \mbox{ for all } x \geq x_{\psi} (3), \notag \\
\left| \theta (x) - x \right| &< 2 c_{\theta}(3) \frac{x}{\log x}  + \log 3 \; \; \mbox{ for all } x \geq x_{\theta} (3), \notag \\
\left| \pi (x) -  \Li (x) \right| &< 2 c_{\pi}(3) \frac{x}{\log^2 x} + 1\; \; \mbox{ for all } x \geq x_{\pi} (3). \label{frodo}
\end{align}
Now
$$
c_{\psi}(3) = 0.0003964, \; \; c_{\theta}(3) = 0.0004015 \; \; \mbox{ and } \; \; c_{\pi}(3) = 0.0004187,
$$
and
$$
x_{\psi}(3) =  576{,}470{,}759, \; \; x_{\theta}(3) =  7{,}932{,}309{,}757\; \; \mbox{ and } \; \; x_{\pi}(3) = 7{,}940{,}618{,}683.
$$
It follows, after a short computation, that we have the desired proof of Corollary~\ref{tiny-theta-cor} for $q \in \{ 1, 2 \}$ and, crudely, $x \geq \max \{ x_{\psi} (3),  x_{\theta} (3), x_{\pi} (3) \} = 7{,}940{,}618{,}683$.  A final calculation, as in the cases $3 \leq q \leq 10^5$, completes the proof.

We now find that for \(1 \leq q \leq 10^5\) and \(x\geq 10^3\), the worst case bounds for $\psi(x;q,a), \theta(x;q,a)$ and $\pi(x;q,a)$ are achieved at $q=2, x=1423^-$, $q=2, x=1423^-$, and $q=2, x=1423^-$ (respectively), giving constants of $0.18997,  0.3987$ and $0.5261$ (respectively). Similarly, when we consider all \(1 \leq q \leq 10^5\) and \(x\geq 10^6\), the worst case bounds for $\psi(x;q,a), \theta(x;q,a)$ and $\pi(x;q,a)$ are achieved at $q=46, x=1015853^-$, $q=4, x=100117^-$, and $q=2, x=1090697^-$ (respectively), giving constants of $0.0106, 0.0233$ and $0.0269$ (respectively).

The upper bound  upon $\left| \pi (x) -  \Li (x) \right|$ given by (\ref{frodo}) implies that we have
$$
\left| \pi (x) -  \Li (x) \right| < 0.0008375 \frac{x}{\log^2 x} \; \; \mbox{ for all } x \geq 7{,}940{,}618{,}683.
$$
Explicitly checking this inequality for all $x < 7{,}940{,}618{,}683$ leads to the reported
inequality~\eqref{PiLi}.

The maximal values of the three quantities in equation~\eqref{goober} for $1\le q\le 10^5$ can be found in the
\begin{center}
\texttt{\href{http://www.nt.math.ubc.ca/BeMaObRe/cor1.7/}{BeMaObRe/cor1.7/}}
\end{center}
subdirectory. This computation strongly resembles the one undertaken to obtain the constants $b_\psi(q)$, $b_\theta(q)$, and
$b_\pi(q)$ (see Appendix~\ref{ssec comp bpsi}), and similar \texttt{C++} code was used.


\subsection{Computations of uniform range of validity for error terms for $\psi (x; q, a)$, $\theta (x; q, a)$, and $\pi (x; q, a)$}  \label{ssec verify cor18}

To establish Corollary~\ref{uniform-cor} from Theorems~\ref{main psi theorem},  \ref{main theta theorem},
and~\ref{main pi theorem}, it suffices to compute a constant $A \geq 0.03$ so that the inequalities
\begin{align*}
  x_\psi(q), x_\theta(q), x_{\theta\#}(q), x_\pi(q) \leq \exp(A \sqrt{q}\log^3q)
\end{align*}
hold for all $3 \leq q \leq 10^5$. Using the quantity
$$
x_m(q) = \max\{x_\psi(q), x_\theta(q), x_{\theta\#}(q), x_\pi(q)\},
$$
\begin{align*}
  \max_{3 \leq q \leq 10^5} \left\{ \frac{\log x_m(q)}{\sqrt{q} \log^3 q} \right\}.
\end{align*}
This maximum was a number close to $9.92545$, obtained at $q=3$, but the quantity under consideration decreases rapidly with $q$ (and is always at most $4.21$ for $q \geq 4$). For $q\ge74$ the maximum is in fact less than the constant $0.03$ from the definition~\eqref{x_0(q) definition} of $x_0(q)$.

Fixing now $q=3$, we verify by direct computation (assuming $x \leq x_m(3)$), that the conclusion of Corollary \ref{uniform-cor} holds for
$$
x \geq 16548949 \approx \exp (7.237439 \sqrt{3} \log^3 3).
$$
Arguing similarly for $3 \leq q \leq 73$, we again obtain the conclusions of Corollary \ref{uniform-cor}, under the weaker assumption that $x \geq \exp (0.03 \sqrt{q} \log^3 q)$, for all $q \geq 58$.

The code and data associated with this computation can be found in the
\begin{center}
\texttt{\href{http://www.nt.math.ubc.ca/BeMaObRe/cor1.8/}{BeMaObRe/cor1.8/}}
\end{center}
subdirectory.

\subsection{Computations of lower bounds for $L(1,\chi)$ for medium-sized moduli $q$ for Lemma~\ref{L1chi lemma} and
Proposition~\ref{L1chi prop}}  \label{magma sec}

We describe one final computation that was used at the end of the proof of Lemma~\ref{L1chi lemma} and the deduction therefrom of Proposition~\ref{L1chi prop}.
Explicit computation using Sage \cite{Sage}, over
fundamental discriminants $d$ with  $4 \cdot 10^5 \leq d \leq 10^7$, shows that the quantity $h(\sqrt d)\log \eta_d$ is minimal when $d=405{,}173$, where we find that $h(\sqrt d)=1$
and $\eta_d=(v_0+u_0 \sqrt{d})/2$ with
$$
v_0 =25{,}340{,}456{,}503{,}765{,}682{,}334{,}430{,}473{,}139{,}835{,}173
$$
and
$$
 u_0 =
39{,}810{,}184{,}088{,}138{,}779{,}581{,}856{,}559{,}421{,}585.
$$
It follows that $h(\sqrt d)\log \eta_d > 79.2177$ for all fundamental discriminants $d$ with $4 \cdot 10^5 \leq d \leq 10^7$.

For each pair of positive integers $(d,u_0)$ for which $d >10^7$ is a fundamental discriminant, $d u_0^2 < 2.65 \cdot 10^{10}$ and $du_0^2+4$ is square, we check via Sage \cite{Sage} that, in all cases,
$$
h(\sqrt d)\log \eta_d = h(\sqrt d)\log \bigg( \frac{\sqrt{du_0^2+4}+ u_0 \sqrt{d}}{2} \bigg) > 417;
$$
indeed, $h(\sqrt d)\log \eta_d$ is minimal in this range when $d=11{,}109{,}293$, for which we find that $h(\sqrt d)=36$ and $\eta = \frac{1}{2} ( 10991 + 33 \sqrt{d} )$.
We may therefore suppose that $d u_0^2 \ge 2.65 \cdot 10^{10}$, which then implies that
$$
\log \eta_d =\log \bigg( \frac{v_0 + u_0 \sqrt{d}}{2}  \bigg) > \log (u_0 \sqrt{d}) \geq  \frac{1}{2} \log ( 2.65 \cdot 10^{10} ) >12,
$$
and so $h(\sqrt d)\log \eta_d > 12$, as desired.
The Sage  \cite{Sage} code used for this computation and its output can be found in the \texttt{\href{http://www.nt.math.ubc.ca/BeMaObRe/lemma5.3/}{BeMaObRe/lemma5.3/}} subdirectory.

\subsection{Concluding remarks from a computational perspective}

From our code, it is relatively easy to examine the effect of sharpening various quantities upon our final constant $c_\psi(q)$ (and its relatives). A decrease of $10 \%$ in the value $R$ defining our zero-free region (from its current values of $5.6$) has a very small effect upon $c_\psi(q)$, leading to a decrease of much less than $1 \%$ in all cases (assuming we leave all other parameters unchanged). Doubling the value of $c_2(q)$, on the other hand, reduces $c_\psi(q)$ by, typically, $25 \%$ or more, for $q$ with $10^4 < q \leq 10^5$; a somewhat less substantial benefit would accrue from confirming GRH for all Dirichlet $L$-functions of conductor $q$, up to height, say, $2 \cdot 10^8/q$.

\section{Notation reference}  \label{reference section}

\begin{table}[H]
\centering
\caption {Notation reference : A to Q}
\begin{tabular}{|l|l|} \hline
$A_m(\delta)$ & equation~\eqref{A def} \\ \hline
$b(\chi)$ & Definition~\ref{mchi bchi def} \\ \hline
$b_\psi(q)$, $b_\theta(q)$, $b_{\theta\#}(q)$,  $b_\pi(q)$ & equation~\eqref{b const defs} \\ \hline
$\funcB$, $\funcBone{x}$, $\funcBtwo{x}$ & Definition~\ref{boundary functions to maximize} \\ \hline
$c_0(q)$ & equation~\eqref{c_0(q) definition} \\ \hline
$c_{\theta}(q), c_{\pi}(q), c_{\psi}(q)$ & Theorems \ref{main psi theorem}, \ref{main theta theorem}, \ref{main pi theorem} \\ \hline
$C_1, C_2$ & Definition~\ref{C1C2} \\ \hline
$D_{q,m,R}(x_2;H_0,H,H_2)$ & Definition~\ref{D def} \\ \hline
$E(u;q,a)$ & Definition~\ref{double error def} \\ \hline
$\erfc(u)$ & Definition~\ref{erfc def} \\ \hline
$\funcF[q=\chi]{x}$ & Definition~\ref{phim def} \\ \hline
$\funcF[q=d]{x}$ & Definition~\ref{def: G} \\ \hline
$g_{d,m}^{(1)}(H,H_2)$, $g_{d,m}^{(2)}(H,H_2)$, $g_{d,m,R}^{(3)}(x;H,H_2)$ & Definition~\ref{phim def} \\ \hline
$\funcG{x}$ & Definition~\ref{def: G} \\ \hline
$\funcG{x_2,r}$ & Definition~\ref{def:G with log} \\ \hline
$h_3(d)$ & Definition~\ref{h3} \\ \hline
$H_1(m)$ & Definition~\ref{H_0} \\ \hline
$\funcHone{x}$, $\funcHtwo{x}$ & Definition~\ref{H1 and H2 def} \\ \hline
Hypotheses Z$(H,R)$, Z${}_1(R)$  & Definition~\ref{hypothesis z} \\ \hline
$I_{n,m}(\alpha,\beta;\ell)$ & Definition~\ref{Inuk def} \\ \hline
$J_{1a}(z;y)$, $J_{1b}(x;y)$, $J_{2a}(z;y)$, $J_{2b}(z;y)$ & Definition~\ref{K pieces def} \\ \hline
$K_n(z;y)$ & Definition~\ref{Knu def} \\ \hline
$\Li (x)$ & equation~\eqref{Li def} \\ \hline
$M_d(\ell,u)$ & Definition~\ref{LtqH def} \\ \hline
$m(\chi)$ & Definition~\ref{mchi bchi def} \\ \hline
$N(T)$ & proof of Proposition \ref{quoting Trudgian principal} \\ \hline
$N(T,\chi)$ & Definition~\ref{NT def} \\ \hline
$P_{*}(x;m,r,\lambda,H,R)$ (various values of ${*}$) & Definition~\ref{J def} \\ \hline
$Q_{*}(m,r,\lambda,H,R)$ (various values of ${*}$) & Definition~\ref{M def} \\ \hline
\end{tabular}
\end{table}

\begin{table}[H]
\centering
\caption {Notation reference : R to $\omega$}
\begin{tabular}{|l|l|} \hline
$R_1$ & Definition~\ref{R1 exceptional def} \\ \hline
$S_{d,m,R}(r,H)$ & Definition~\ref{collect bounds def} \\ \hline
$S(T)$ & proof of Proposition \ref{quoting Trudgian principal} \\ \hline
$T_1, T_2, T_3, T_4$ & Definition~\ref{D def} \\ \hline
$\funcU{x}$ & equation~\eqref{U def} \\ \hline
$\funcV{x}$ & equation~\eqref{V def} \\ \hline
$\funcW{x}$ & equation~\eqref{W def} \\ \hline
$x_0(q)$ & equation~\eqref{x_0(q) definition} \\ \hline
$x_{\theta}(q), x_{\pi}(q), x_{\psi}(q)$ & Theorems \ref{main psi theorem}, \ref{main theta theorem}, \ref{main pi theorem} \\ \hline
$x_2(q)$ & equation (\ref{x2 definition}) \\ \hline
$x_3(m,q,H,R)$ & Definition~\ref{x3 def} \\ \hline
$\funcY{u}$ & Definition~\ref{phim def} \\ \hline
$y_{d,m,R}(x;H_2)$ & Definition~\ref{z and y def} \\ \hline
$z_{m,R}(x)$ & Definition~\ref{z and y def} \\ \hline
$\Zchi$ & Definition~\ref{Zchi def} \\ \hline
$\alpha_{m,k}$ & Definition~\ref{alpha def} \\ \hline
$\Delta_k(x;q)$, $\Delta(x;q)$ & Definition~\ref{xi defs} \\ \hline
$\theta(x;q,a)$ & equation~\eqref{theta xqa and psi xqa} \\ \hline
$\theta_\#(x;q,a)$ & equation~\eqref{thetaH def} \\ \hline
$\Theta(d,t)$ & equation~\eqref{Theta(d,t)} \\ \hline
$\nu(q,H_0,H)$ & Definition~\ref{Have you heard the good nus?} \\ \hline
$\nu_1(\chi,H_0)$ & Definition~\ref{Have you heard the good nus?} \\ \hline
$\nu_2(q,H_0)$ &  Definition~\ref{Have you heard the good nus?}  \\ \hline
$\nu_3(q,H)$ &  Definition~\ref{Have you heard the good nus?}  \\ \hline
$\xi_k(q)$, $\xi_k(q,a)$ & Definition~\ref{xi defs} \\ \hline
$\Xi_{m,\lambda,\mu,R}(x)$ & Definition~\ref{e def} \\ \hline
$\tau_m$ & Definition~\ref{tau and omega def} \\ \hline
$\pi(x;q,a)$ & equation~\eqref{pi xqa} \\ \hline
$\Upsilon_{q,m}(x;H)$ & Definition~\ref{Upsilon and Psi def} \\ \hline
$\phi^*(d)$ & Definition~\ref{def: phistar} \\ \hline
$\psi(x;q,a)$ & equation~\eqref{theta xqa and psi xqa} \\ \hline
$\Psi_{q,m,r}(x;H)$ & Definition~\ref{Upsilon and Psi def} \\ \hline
$\omega_m$ & Definition~\ref{tau and omega def} \\ \hline
\end{tabular}
\end{table}

\section*{Acknowledgments}

The first, second, and fourth authors were supported by NSERC Discovery Grants.
The authors also gratefully acknowledge the Banff International Research Station (BIRS) for providing the first three authors with a stimulating venue to begin this project, Westgrid and The Ha for computational support, Habiba Kadiri and Allysa Lumley for
providing access to their work in progress, Olivier Ramar\'e for helpful conversations, Kirsten Wilk for helpful comments on the manuscript, and the anonymous referees for their suggestions for improving this article.


\end{document}